\documentclass[lettersize,journal]{IEEEtran}
\def\BibTeX{{\rm B\kern-.05em{\sc i\kern-.025em b}\kern-.08em
T\kern-.1667em\lower.7ex\hbox{E}\kern-.125emX}}
\usepackage{balance}

\usepackage{array}
\usepackage{amsmath,amssymb,amsfonts,amsthm}
\newtheorem{theorem}{Theorem}
\newtheorem{lemma}{Lemma}

\newtheorem{definition}{Definition} 

\usepackage{graphicx}
\usepackage{epstopdf}
\usepackage{subfigure}
\usepackage{url}
\usepackage{bm}
\usepackage{caption}
\usepackage{multirow}
\usepackage[T1]{fontenc}
\usepackage{stfloats}
\usepackage{tabularx}
\usepackage{booktabs}
\usepackage{lipsum}
\usepackage{mathrsfs}
\usepackage{enumitem} 

\usepackage{tabularx}
\usepackage{cite}
\usepackage[switch]{lineno}
\usepackage{multirow}
\usepackage{makecell}
\usepackage{ amssymb,lineno,url}
\usepackage{amssymb,amsthm}
\usepackage{color}
\usepackage{verbatim}

\usepackage{makecell}
\usepackage{algorithm}
\usepackage{algpseudocode}
\usepackage{ bm}
\usepackage{multirow}
\usepackage{float}
\usepackage{booktabs}

\usepackage[numbers]{natbib}           
\usepackage[colorlinks]{hyperref}    
\usepackage{ragged2e}

\usepackage{lipsum} 
\usepackage{caption} 
\usepackage{adjustbox} 
\usepackage{tikz}

\newtheorem{remark}{Remark}
\newcommand*\circled[1]{\tikz[baseline=(char.base)]{
\node[shape=circle,draw,inner sep=0.5pt] (char) {#1};}}

\begin{document}
\title{Feature Incremental Clustering with Generalization Bounds}
\author{Jing~Zhang, and Chenping~Hou*,~\IEEEmembership{Member,~IEEE}
\IEEEcompsocitemizethanks{\IEEEcompsocthanksitem Jing~Zhang is with Key Laboratory of Computing and Stochastic Mathematics (Ministry of Education), School of Mathematics and Statistics, Hunan Normal University, Changsha, Hunan 410081, P. R. China. }
\thanks{Chenping~Hou is with College of Liberal Arts and Sciences, National University of Defense Technology, Changsha, Hunan, 410073, China.}
\thanks{E-mail: zhangjing\_nudt@163.com, hcpnudt@hotmail.com.}
\thanks{*Chenping Hou is the corresponding author.}
}

\maketitle

\begin{abstract}
In many learning systems, such as activity recognition systems, 
as new data collection methods continue to emerge in various dynamic environmental applications, the attributes of instances accumulate incrementally, with data being stored in gradually expanding feature spaces. How to design theoretically guaranteed algorithms to effectively cluster this special type of data stream, commonly referred to as activity recognition, remains unexplored. Compared to traditional scenarios, we will face at least two fundamental questions in this feature incremental scenario. (i) How to design preliminary and effective algorithms to address the feature incremental clustering problem? (ii) How to analyze the generalization bounds for the proposed algorithms and under what conditions do these algorithms provide a strong generalization guarantee? To address these problems, by tailoring the most common clustering algorithm, i.e., $k$-means, as an example, we propose four types of Feature Incremental Clustering (FIC) algorithms corresponding to different situations of data access: Feature Tailoring (FT), Data Reconstruction (DR), Data Adaptation (DA), and Model Reuse (MR), abbreviated as FIC-FT, FIC-DR, FIC-DA, and FIC-MR. Subsequently, we offer a detailed analysis of the generalization error bounds for these four algorithms and highlight the critical factors influencing these bounds, such as the amounts of training data, the complexity of the hypothesis space, the quality of pre-trained models, and the discrepancy of the reconstruction feature distribution. The numerical experiments show the effectiveness of the proposed algorithms, particularly in their application to activity recognition clustering tasks.
\end{abstract}


\section{Introduction} 
\label{sec1}
Clustering stands as a fundamental problem in data analysis, extensively explored across various domains including image segmentation, bioinformatics, data science, and engineering \cite{shi2000normalized,li2019clustering,Mishro2021}. Traditional clustering approaches are often constrained to closed environments \cite{Peng2022,Zhou2025}, requiring the space of features to be maintained consistently across both the training and testing phases. Although the scope of clustering has been extended to dynamic environments (e.g., the clustering of data stream \cite{fahy2018ant,sui2020dynamic, chen2022two,Urio2025}), the existing literature predominantly considers time-varying data distributions under a fixed feature manifold. Such a paradigm fails to account for many practical tasks in open environments, such as feature shifts (feature increment or feature decrement \cite{cheng2025tabfsbench}). This study focuses specifically on the former-the scenario of feature increment, which is a critical yet under-explored challenge. 
For instance, in activity recognition, different clusters involve different kinds of exercises \cite{yan2015egocentric}. At the beginning, a few sensors are usually worn on the body during warm-up. As the warm-up phase comes to an end, additional sensors may be employed to monitor subsequent formal exercises. Then new types of features can be collected with new sensors and accumulated to the old ones. Besides, in the context of clustering news articles \cite{ors2020event}, the emergence of news sites can expand the dictionary of keywords, thereby introducing new types of features. This illustrates another example of incremental clustering with new features. This shift violates the assumption of a static feature space, which not only compromises input consistency but also inevitably leads to a decline in clustering accuracy and robustness. Such challenges motivate the design of adaptive clustering algorithms that can handle the feature increment.
\begin{figure}[!t]
\centering
\includegraphics[width=0.4\textwidth]{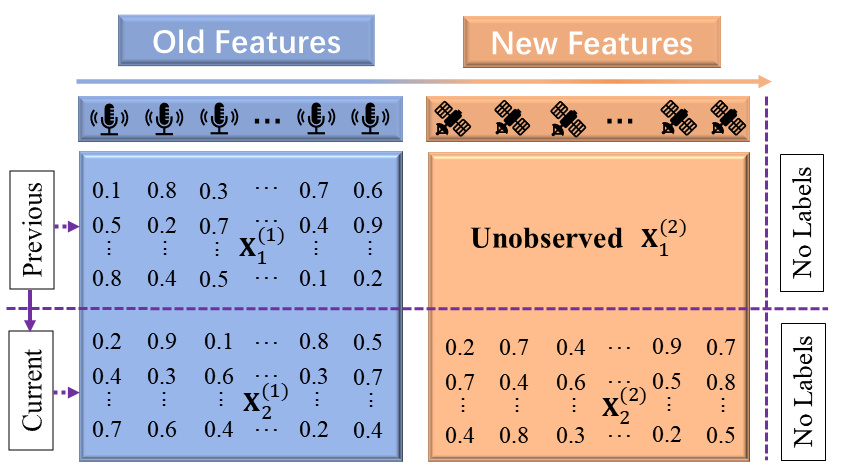}
\vskip -0.08 in
\caption{Diagram illustrating the feature incremental clustering scenario in a dynamic open environment. In the activity recognition task, let $\mathbf{X}^{(1)}_{1}$ be the old feature obtained by the previous sensors in the previous stage; In the current stage, new features $\mathbf{X}^{(2)}_{2}$ and old features $\mathbf{X}^{(1)}_{2}$ are collected by the new and the old sensors, respectively. Then, $\mathbf{X}_{2}=[\mathbf{X}^{(1)}_{2},\mathbf{X}^{(2)}_{2}]$ represents new data in the current stage.}
\label{scenariofig}	
\end{figure}

To better understand the scenario mentioned above, we provide a schematic diagram shown in Fig.\,\ref{scenariofig}. Specifically, in the feature incremental clustering scenario, data collection is partitioned into two stages: the previous 
stage and the current stage. Take the activity recognition clustering task as an example, during the warm-up exercise, five sensors are positioned on the body to gather features of the previous stage, viewed as old features $\mathbf{X}_{1}^{(1)}$. After completing the warm-up phase, to enhance activity monitoring, we introduce additional four sensors to collect features generated in the current stage, viewed as new features $\mathbf{X}_{2}^{(2)}$. These two stages are associated with distinct feature spaces: the old feature space $\mathcal{X}_{p}$, where $\mathbf{X}_{1}^{(1)}\in  \mathcal{X}_{p}$, and the new feature space $\mathcal{X}_{c}$, where $\left[\mathbf{X}^{(1)}_{2},\mathbf{X}^{(2)}_{2}\right]\in \mathcal{X}_{c}$. Besides, it is important to note that there are typically no labels available for the data at each stage in feature incremental clustering. Due to the cost of collecting new feature attributes, there are usually only a limited number of data points available in the current stage with new features.

Given this data context, some similar studies have emerged to address the feature incremental challenge in the classification tasks. For instance, in tackling the challenge posed by feature-evolving streaming data, Hou et al. present a classification model that leverages the relationship between existing and newly introduced features \cite{hou2017learning}. Subsequently, numerous classification algorithms have been proposed to address different aspects of feature evolution, including feature shifts \cite{hou2018one, sadreddin2021incremental}, safe integration of new features \cite{hou2019safe}, joint feature and distribution shifts \cite{zhang2020learning}, and simultaneous expansion of feature and class spaces \cite{hou2023incremental}. Other related studies can be found in \cite{shu2023incremental, cheng2025tabfsbench}. Different from our setting, a common aspect among these studies is that they all focus on classification scenarios, whereas our article centers on clustering tasks without any label information. Consequently, feature incremental clustering problems even more challenging. 

Unlike traditional scenarios, feature incremental clustering raises two fundamental questions: (i) how can we design preliminary and effective algorithms to address the feature incremental clustering problem? (ii) from a theoretical perspective, how do we analyze the generalization bounds for the proposed algorithms, and under what conditions do these algorithms provide a strong generalization guarantee?

For the first question, traditional clustering algorithms are unsuitable for direct application due to our unique incremental setting. We need to design a specialized algorithm for this scenario. For the second question, existing statistical theories cannot be directly applied to the feature-incremental problem, even though they have been extensively developed for static scenarios such as such as $k$-means \cite{antos2005improved, tang2016lloyd, zhang2023imbalanced}, kernel $k$-means \cite{biau2008performance, li2021sharper, liu2021refined, yin2022randomized, liang2024consistency}, and spectral clustering \cite{li2023understanding}. Overall, feature incremental clustering remains a critical yet underexplored problem, necessitating the development of methods that are both practically effective and theoretically sound.

To tackle the challenges mentioned in this new learning scenario, we extend the standard $k$-means algorithm to propose four Feature Incremental Clustering (FIC) methods tailored to different data access conditions: Feature Tailoring (FT), Data Reconstruction (DR), Data Adaptation (DA), and Model Reuse (MR). These algorithms are abbreviated as FIC-FT, FIC-DR, FIC-DA, and FIC-MR. We then investigate the generalization error bounds of these four algorithms, providing a detailed analysis of the key factors influencing these bounds. Specifically, the generalization performance of FIC-FT is primarily influenced by the clustering capability of the old features; FIC-DRis sensitive to the discrepancy in the reconstructed feature distribution; FIC-DA is mainly constrained by the sample size; 
and FIC-MR is governed by the quality of the pre-trained model. Notably, for FIC-MR, a high-quality pre-trained model tightens the bound, achieving a faster convergence rate of $\tilde{\mathcal{O}}(\frac{1}{n_2})$ instead of $\tilde{\mathcal{O}}(\sqrt{\frac{k}{n_2}}+\frac{1}{n_2})$. 
In addition to these theoretical results, we validate and compare the generalization performance of the four models through numerical experiments. Considering both storage requirements and performance, FIC-MR may be the preferred choice for feature incremental clustering.

The major contributions are summarized as follows:
\begin{itemize}
\item We propose four feature incremental clustering algorithms, representing, to the best of our knowledge, the first comprehensive investigation into this important yet rarely studied problem.
\item We derive generalization error bounds for the proposed algorithms and analyze the critical factors influencing their tightness.
\item We systematically evaluate the proposed algorithms across diverse datasets. Numerical results validate their effectiveness and support our theoretical findings.

\end{itemize}

The rest of this paper is organized as follows. Section\,\ref{BackgroundSection} introduces relevant notations and provides an overview of $k$-means. Section\,\ref{Feature Incremental Clustering Section} develops four feature incremental clustering algorithms based on the framework of $k$-means. Section\,\ref{Generalization Error bound Analyses Section} establishes the generalization error bounds of the four algorithms. Section\,\ref{Experiments Section} presents experimental results on various real-world datasets. Section\,\ref{ConclusionSection} concludes this paper. 

\section{Background}
\label{BackgroundSection}
\subsection{Notations} 
In Fig.\,\ref{scenariofig}, we consider a $k$-class clustering task spanning two sequential stages:
(i) \textbf{Previous Stage}: Let $\mathcal{X}_{p}$ be the previous feature space. Denote the previous data set as $ D_p = \{\mathbf{x}_i^{(1)} \in \mathbb{R}^{d_1}\}_{i=1}^{n_1}$, where $n_1$ denotes the number of instances and $d_1$ stands for the dimensions of old features. Let $\mathbf{X}_1^{(1)} \in \mathbb{R}^{n_1 \times d_1}$ denote the data matrix. Each row $\mathbf{x}_i^{(1)}$ corresponds to the $i$-th instance. (ii) \textbf{Current Stage}: New instances arrive in an expanded feature space $\mathcal{X}_{c}$ that includes both old and new features, are represented as a set $ D_c=\{\mathbf{x}_i\in \mathbb{R}^{d_1+d_2}\}_{i=1}^{n_2}$, where $d_2$ is the dimensions of new features, and $n_2$ denotes the number of instances. To elaborate further, instances involving the previous features during the current stage are denoted as a set $D_{c^{p}}=\{\mathbf{x}_i^{(1)} \in \mathbb{R}^{d_1}\}_{i=n_1+1}^{n_1+n_2}$. Furthermore, instances containing the new features in the current stage are represented as $\{ \mathbf{x}_i^{(2)} \in \mathbb{R}^{d_2}\}_{i=1}^{n_2}$. 
For ease of reference, let $\mathbf{X}_2 = [\mathbf{X}_2^{(1)},\mathbf{X}_2^{(2)}] \in \mathbb{R}^{n_2\times (d_1+d_2)}$ represent the data matrix in the current stage, where $\mathbf{X}_2^{(1)} \in \mathbb{R}^{n_2 \times d_1}$ and $\mathbf{X}_2^{(2)} \in \mathbb{R}^{n_2 \times d_2}$ correspond to the data matrices of the old features and new features, respectively. 
(iii) \textbf{Testing Stage}:  Test data $\mathbf{X}_t = [\mathbf{X}_t^{(1)}, \mathbf{X}_t^{(2)}] \in \mathbb{R}^{n_t \times (d_1+d_2)}$ contains both old and new features, with $n_t$ denoting the number of test instances. 
Additional notation will be defined upon its first use to ensure clarity.

\subsection{$k$-Means }
Denote $\mathbb{Q}$ as an unknown distribution over the input space $\mathcal{X}$. Let $D =\left\{\mathbf{x}_1,\dots,\mathbf{x}_n\right\} \subset \mathcal{X}$ be a sample set generated i.i.d from unknown distribution$\mathbb{Q}$. Let $\mathcal{H}$ be a separable Hilbert space, with inner product $\langle\cdot,\cdot\rangle$  and the corresponding norm $\lVert\cdot\rVert$. $k$-means partitions  $\left\{\mathbf{x}_i\right\}_{i=1}^n$ into $k$ clusters by optimizing the cluster centers $\left\{\mathbf{u}_s\right\}_{s=1}^k$. The objective is to minimize the empirical clustering risk, defined as the squared norm criterion:
\begin{equation}
\label{k-meansOb}
\hat{\mathcal{L}}_{D}(\mathbf{U}):= \frac{1}{n} \sum_{i=1}^{n}\min _{s=1, \ldots, k}\lVert\mathbf{x}_i-\mathbf{u}_s\rVert^2
\end{equation}
on the all possible cluster centers $\mathbf{U}=\left[\mathbf{u}_1,\dots,\mathbf{u}_k\right]\in \mathcal{H}^{k}$.
Simply, an instance $\mathbf{x}_{i}$ is assigned to the $s$-th cluster if the squared distance between $\mathbf{x}_{i}$ and the cluster center $\mathbf{u}_s$ is the smallest among all centers.
The performance of a clustering approach under distribution $\mathbb{Q}$ is usually measured by the expected clustering risk (see Definition\,\ref{expected clustering risk definition}).
\begin{definition}[\cite{biau2008performance}]The expected clustering is defined as
\label{expected clustering risk definition}
\begin{equation}
\mathcal{L}_{\mathbb{Q}}(\mathbf{U}):=\int \min _{s=1, \ldots, k}\lVert \mathbf{x}-\mathbf{u}_s \rVert^2d \mathbb{Q}(\mathbf{x}).
\end{equation}
\end{definition}	
\noindent The Empirical Risk Minimizer (ERM) is defined as follows:
\begin{definition}[\cite{biau2008performance}]
\label{emprical center define}
Given the empirical squared norm criterion $\hat{\mathcal{L}}_{D}(\mathbf{U})$,  the definition of the empirical risk minimizer is given by
\begin{equation}
\tilde{\mathbf{U}}:= \mathop {\arg\min }\limits_{\mathbf{U}\in \mathcal{H}^{k}}	\hat{\mathcal{L}}_{D}(\mathbf{U}).
\end{equation}
\end{definition}

In this work, we implement our proposed algorithms within the $k$-means framework, leveraging its well-established theoretical foundations. Nevertheless, the methodological principles introduced here are widely applicable and can be extended to other advanced clustering algorithms without loss of generality. We implement our algorithms within the $k$-means framework due to its solid theoretical foundation. However, \textbf{the proposed principles are general and can be readily extended to other clustering methods.}

Our goal is to develop algorithms that leverage prior information and limited current data, together with a comprehensive theoretical analysis, which we present step by step in the following sections.

\begin{figure*}[!t] 
	\centering
	\includegraphics[width=0.95\textwidth,height=0.6\textwidth]{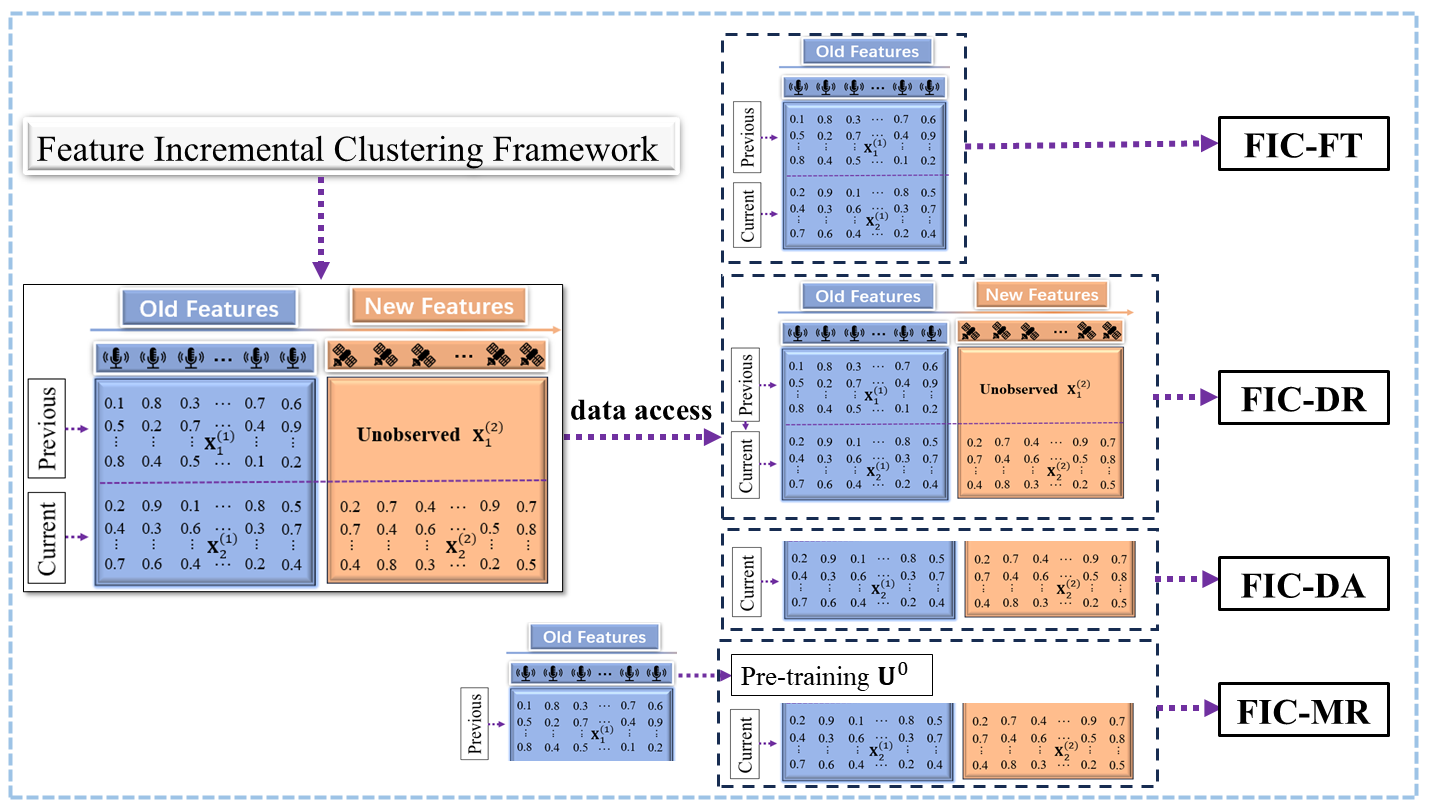}
	\vskip -0.05in
	\caption{Illustration of the four feature incremental clustering.}
	\label{fourmethodsfig}	
\end{figure*}

\section{Feature Incremental Clustering} 
\label{Feature Incremental Clustering Section}
To the best of our knowledge, research on feature incremental learning has been largely confined to classification tasks, leaving it under-explored in the development of Feature Incremental Clustering (FIC) models. Due to the inherent mismatch in feature dimensionality over time, directly extending conventional clustering paradigms to the FIC setting is non-trivial. 
In this section, we present a systematic exploration of FIC by introducing four distinct algorithmic frameworks: Feature Tailoring (FT), Data Reconstruction (DR), Data Adaptation (DA),  and Model Reuse (MR), denoted as FIC-FT, FIC-DR, FIC-DA, and FIC-MR. This is an initial exploration of FIC algorithms, marking the first example of such research. 

The four proposed clustering algorithms, based on different strategies of data access, are illustrated in Fig.\,\ref{fourmethodsfig}. When data from the previous stage is accessible, both FIC-FT and FIC-DR can be applied by combining information from these two stages. Specifically, FIC-FT discards the new features from the current stage and merges the old features from both the previous and current stage data. Its objective is to train a clustering model utilizing only the old features from both stages. In contrast, FIC-DR first recovers the unobserved features from the previous stage data. The reconstructed data is then combined with the current stage data, which includes both old and new features, to train a clustering model. On the other hand, when the information from the previous stage is completely unavailable, such as data feature or model information,  FIC-DA can be employed to address this scenario. This algorithm trains a clustering model utilizing only the limited data available in the current stage. Finally, when the feature information from the previous stage is inaccessible but the clustering model is preserved, we apply FIC-MR. This algorithm leverages the pre-trained model as a prior to guide and improve clustering performance in the current stage.

Next, we detail the four proposed FIC algorithms.

\subsection{FIC-FT Algorithm}
In this subsection, we introduce Feature Tailoring for Feature Incremental Clustering (FIC-FT). 

Specifically, when data from the previous stage is accessible, FIC-FT addresses the feature-mismatch problem by aligning feature spaces across stages. It discards the newly observed features in the current stage to create a unified dataset $D_p \cup D_{c^p}$ for training (see Fig.\,\ref{fourmethodsfig} for the data access scheme). Consequently, the $k$-means algorithm can be applied to develop a clustering model based on this tailored dataset. 

Let $\mathbf{U}_1\in \mathcal{H}^{k}$ be all possible cluster centers for FIC-FT.  According to the squared norm criterion, the objective function of FIC-FT is formulated as:
\begin{equation}
\label{Feature Tailoring_Ob}
\min _{\mathbf{U}_1 \in \mathcal{H}^k } \frac{1}{n_1+n_2} \sum_{\mathbf{x}_i\in D_p \bigcup D_{c^p}} \min _{s=1, \ldots, k}\lVert\mathbf{x}_i-\mathbf{u}_s\rVert^2,
\end{equation}
where the cluster centers $\mathbf{U}_1=[\mathbf{u}_{1},\ldots,\mathbf{u}_{k}]\in \mathbb{R}^{d_1 \times k}$ is optimized from the dataset $D_p \bigcup D_{c^p}$.

\subsection{FIC-DR Algorithm}
We introduce Data Reconstruction for Feature Incremental Clustering (FIC-DR), which addresses the scenario where previous-stage data is accessible by jointly training on data from both previous and current stages.

Specifically, in the first step, the unobserved features from the previous stage (i.e., $\mathbf{X}_{1}^{(2)}$) are treated as missing new features and are approximately reconstructed using the current-stage data, which contains both old and new features, together with the previous-stage data containing only old features. This reconstruction is based on a mild assumption that the newly introduced features are correlated with existing ones, an assumption commonly adopted in feature incremental learning \cite{Hou2023,Gu2025}, and often satisfied in practice, e.g., when new features originate from upgraded sensors or extended dimensions within the same modality.
Obviously, the existing feature completion methods can be adapted to recover the missing features from the previous stage and obtain a recovered feature matrix $\tilde{\mathbf{X}}_{1}^{(2)}$. 
For the second step, the first step enables us to combine data from the two stages (i.e., $ [\tilde{\mathbf{X}} _{1}, \mathbf{X} _{2}]$ or $ \tilde{D}_p \bigcup D_{c} $), where $\tilde{\mathbf{X}} _{1} = [\mathbf{X} _{1}^{(1)},\tilde{\mathbf{X}}_{1}^{(2)}]$, $\mathbf{X} _{2} = [\mathbf{X} _{2}^{(1)},\mathbf{X}_{2}^{(2)}] $, and $\tilde{D}_p = \left\{\tilde{\mathbf{x}}_1,\dots, \tilde{\mathbf{x}}_{n_1}\right\}$ is the reconstructed dataset in the previous stage.  Then, $k$-means algorithm can be employed to train a model on the combined dataset $ \tilde{D}_p \bigcup D_{c} $ to optimize the cluster centers $\mathbf{U}_2\in \mathcal{H}^{k}$.
Below, we offer a completion method used for the first step of FIC-DR. 

Based on the framework of $k$-means, we introduce a reconstructed method to adapt to our scenario of learning an unobserved feature matrix.
More precisely, let $\mathbf{X}_{\text{all}} =\left[\mathbf{X}_{1},\mathbf{X}_{2}\right]$, where $\mathbf{X} _{1} = \left[\mathbf{X} _{1}^{(1)},\mathbf{X}_{1}^{(2)}\right]$ and $\mathbf{X} _{2} = \left[\mathbf{X} _{2}^{(1)},\mathbf{X}_{2}^{(2)}\right] $.  
Our goal is to learn the unobserved feature matrix to obtain a recovered version $\tilde{\mathbf{X}}_{1}^{(2)}$ that closely approximates the true feature matrix $\mathbf{X}_{1}^{(2)}$.
Importantly, the optimization process of the proposed method maintains the observed features $\mathbf{X}_{1}^{(1)}$ and $\mathbf{X}_{2}$ unchanged.
The objective function of this reconstructed method can be expressed as follows:
\begin{equation}
\label{Data Reconstruction_Ob}
\begin{aligned}
\begin{split}
&\min _{\mathbf{U} \in \mathcal{H}^k,\mathbf{X}_{\text{all}} }\frac{1}{n_1+n_2} \sum_{i=1}^{n_1+n_2} \min _{s=1, \ldots, k}\lVert\mathbf{x}_i-\mathbf{u}_s\rVert^2\\
& \ \ \quad \mathrm{ s.t. } \quad \mathbf{X}_{1}^{(1)} =\mathbf{X}_{1}^{(1)}, \ \ \mathbf{X}_{2}=\mathbf{X}_{2}, 
\end{split}
\end{aligned}
\end{equation}
where unobserved feature matrix $\mathbf{X}_{1}^{(2)}$ can be reconstructed by this objective function. The optimization of this objective function follows a procedure similar to that of $k$-means.

\textbf{It is important to note that the approaches for recovering the unobserved feature $\mathbf{X}_{1}^{(2)}$ can be varied. } 
For example, the techniques proposed in \cite{pmlr-v119-muzellec20a,pmlr-v80-ye18c} can be readily adapted to reconstruct $\mathbf{X}_{1}^{(2)}$ using optimal transport strategies and feature meta-information encoding. 

\subsection{FIC-DA Algorithm}
This subsection introduces Data Adaptation for Feature Incremental Clustering (FIC-DA). When previous-stage data is inaccessible, the most straightforward strategy is to train the model exclusively on the limited data $D_c$ available from the current stage(see Fig.\,\ref{fourmethodsfig}).

Let $\mathbf{U}_3\in \mathcal{H}^{k}$ denote all possible cluster centers of FIC-DA. The objective function of FIC-DA is given by
\begin{equation}
\label{Data Adaptation_Ob}
\min _{\mathbf{U}_3 \in \mathcal{H}^k }\frac{1}{n_2} \sum_{\mathbf{x}_i\in D_c} \min _{s=1, \ldots, k}\lVert\mathbf{x}_i-\mathbf{u}_s\rVert^2.
\end{equation}
One drawback of employing FIC-DA is that, in some scenarios, limited data in the current stage may prevent the development of a robust model. 

\subsection{FIC-MR Algorithm}
In practice, access to previous data is often restricted by storage or privacy constraints. However, saving the parameters of a well-trained model from the previous stage remains feasible. In this case, we introduce Model Reuse for Feature Incremental Clustering (FIC-MR), which leverages the pre-trained model from the previous stage alongside the limited data available in the current stage(see Fig.\,\ref{fourmethodsfig}). 

Specifically, the pre-trained model from the previous stage provides the initial cluster centers $\mathbf{U}_0=[\mathbf{u}^{0}_{1},\ldots,\mathbf{u}^{0}_{k}]\in \mathbb{R}^{d_1\times k}$.
In the current stage, each cluster center is extended to a higher-dimensional space, denoted as 
$\mathbf{u}_s = [\mathbf{u}_s^{(1)}, \mathbf{u}_s^{(2)}] \in \mathbb{R}^{d_1 + d_2}$. It is reasonable to assume that the sub-components $\mathbf{u}_s^{(1)}$ are consistent with the pre-trained centers $\mathbf{u}_s^0$. Instead of independent learning, FIC-MR incorporates the prior model $\mathbf{u}_s^0$ into the current stage through a consistency constraint. By optimizing the cluster centers $\mathbf{U}_4 \in \mathcal{H}^k$ under this guidance, the objective function is formulated as:
\begin{equation}
\label{modelresue_Ob}
\min _{\mathbf{U}_4 \in \mathcal{H}^k}(\frac{1}{n_2} \sum_{\mathbf{x}_i\in D_c} \min _{s=1, \ldots, k}\lVert\mathbf{x}_i-\mathbf{u}_s\rVert^2 + \theta \lVert\mathbf{U}_4^{(1)} - \mathbf{U}^{0}\rVert^2),
\end{equation}
where $\mathbf{U}_4 = [\mathbf{U}_4^{(1)}; \mathbf{U}_4^{(2)}] \in \mathbb{R}^{(d_1+d_2) \times k}$ partitions the centers into old and new feature components.
Here, $\theta$ serves as a positive trade-off parameter. 

While the first three FIC algorithms can be solved using by referring to $k$-means algorithm, FIC-MR incorporates a regularization term that necessitates a more specialized optimization strategy. To this end, we develop a new iterative algorithm to minimize the objective function in Eq.\,(\ref{modelresue_Ob}). Specifically, the objective function is reformulated as:
\begin{equation}\label{modelresue_Ob3}
\min _{\mathbf{U}_4 \in \mathcal{H}^k}(\frac{1}{n_{2}} \sum_{s=1}^k \sum_{\mathbf{x}_i \in (\mathcal{C}_s\cap D_c)}(\lVert\mathbf{x}_i-\mathbf{u}_s \rVert^2)+ \theta \sum_{s=1}^k\lVert\mathbf{u}_s^{(1)} - \mathbf{u}_s^{0}\rVert^2).
\end{equation}
Here, $\mathcal{C}_s$ is the $s$-th cluster associated with center $\mathbf{u}_s = [\mathbf{u}_s^{(1)}, \mathbf{u}_s^{(2)}]$. Thus, the above formula (\ref{modelresue_Ob3}) can be rewritten as:
\begin{equation}
\label{modelresue_Ob2}
\begin{aligned}
\begin{split}
&\min _{\mathbf{U}_4 \in \mathcal{H}^k}(\sum_{s=1}^k \sum_{\mathbf{x}_i \in (\mathcal{C}_s\cap D_c)}\lVert\mathbf{x}_i^{(1)}-\mathbf{u}_s^{(1)} \rVert^2+\lVert\mathbf{x}_i^{(2)}-\mathbf{u}_s^{(2)} \rVert^2
\\  
&\quad  \quad \quad \quad \quad 
+ \theta\sum_{s=1}^k \lVert\mathbf{u}_s^{(1)} - \mathbf{u}_s^{0}\rVert^2),
\end{split}
\end{aligned}
\end{equation}
where $\mathbf{x}_i^{(1)}$ and $\mathbf{x}_i^{(2)}$ denote the old and new features of the current instances, respectively. Since the objective function (\ref{modelresue_Ob2}) is separable for each $\mathbf{u}_s$, the optimization yields $k$ independent quadratic sub-problems. Consequently, $\mathbf{u}_s^{(1)}$ and $\mathbf{u}_s^{(2)}$ can be updated separately via their respective closed-form solutions.

\textbf{Updating $\mathbf{u}_s^{(1)}(s=1, \ldots, k)$ with fixed $\mathbf{u}_s^{(2)}$}: For each cluster center $\mathbf{u}_s^{(1)}$, the objective function (\ref{modelresue_Ob2}) can be formulated as a sub-problem for each of the $k$ clusters,
\begin{equation*}
\label{modelresue_old centerupdate}
\begin{aligned}
\begin{split}
&\min _{\mathbf{U}_4^{(1)} \in \mathcal{H}^k}\sum_{\mathbf{x}_i\in (\mathcal{C}_s\cap D_c)}(\lVert\mathbf{x}_i^{(1)}-\mathbf{u}_s^{(1)} \rVert^2+\lVert\mathbf{x}_i^{(2)}-\mathbf{u}_s^{(2)} \rVert^2)\\
&\quad \quad \quad +\theta\lVert\mathbf{u}_s^{(1)} - \mathbf{u}_s^{0}\rVert^2.
\end{split}
\end{aligned}
\end{equation*}
Let the optimized function be $\mathcal{P}_{1}(\mathbf{u}_s^{(1)})$. We set the derivative of it
w.r.t. $\mathbf{u}_s^{(1)}$ to 0, and solve for it as follows:
\begin{equation*}
\begin{aligned}
\begin{split}
\frac{\partial \mathcal{P}_{1}(\mathbf{u}_s^{(1)})}{\partial \mathbf{u}_s^{(1)}}
&=\sum_{\mathbf{x}_i\in (\mathcal{C}_s\cap D_c)}-2\mathbf{x}_i^{(1)}+2\left|\mathcal{C}_s\right|\mathbf{u}_s^{(1)} +2 \theta\mathbf{u}_s^{(1)} - 2\theta\mathbf{u}_s^{0}\\
&=0,
\end{split}
\end{aligned}
\end{equation*}
where $\left|\mathcal{C}_s\right|$ denotes the number of instances in the set $\mathcal{C}_s$.
After simplification, we can obtain $\mathbf{u}_s^{(1)}= \frac{\sum_{\mathbf{x}_i \in (\mathcal{C}_s\cap D_c)}\mathbf{x}_i^{(1)}+ \theta\mathbf{u}_s^{0}}{\left|\mathcal{C}_s\right|+\theta}.$

\textbf{Updating $\mathbf{u}_s^{(2)}(s=1, \ldots, k)$ with fixed  $\mathbf{u}_s^{(1)}$}: 
Similarly, for each cluster center $\mathbf{u}_s^{(2)}$, we can  transform the objective function (\ref{modelresue_Ob2}) into a sub-problem for each of the $k$ clusters:
\begin{equation*}
\label{modelresue_new centerupdate}
\begin{aligned}
\begin{split}
&\min _{\mathbf{U}_4^{(2)} \in \mathcal{H}^k}\sum_{\mathbf{x}_i\in (\mathcal{C}_s\cap D_c)}\lVert\mathbf{x}_i^{(2)}-\mathbf{u}_s^{(2)} \rVert^2.
\end{split}
\end{aligned}
\end{equation*}
Let the optimized function be $\mathcal{P}_{2}(\mathbf{u}_s^{(2)})$.
We set the derivative of it w.r.t. $\mathbf{u}_s^{(2)}$ to 0, that is
\begin{equation*}
\begin{aligned}
\begin{split}
\frac{\partial \mathcal{P}_{2}(\mathbf{u}_s^{(2)})}{\partial \mathbf{u}_s^{(2)}}
=\sum_{\mathbf{x}_i \in (\mathcal{C}_s\cap D_c)}-2(\mathbf{x}_i^{(2)}-\mathbf{u}_s^{(2)} )=0.
\end{split}
\end{aligned}
\end{equation*}
Hence, we can derive $	\mathbf{u}_s^{(2)}=\frac{\sum_{\mathbf{x}_i \in (\mathcal{C}_s\cap D_c)}(\mathbf{x}_i^{(2)})}{\left|\mathcal{C}_s\right|}$.
In summary, although using $\mathbf{X}_2$ merely for training a clustering model during the current stage is a feasible approach, the available data in the current stage is limited, as $n_2$ is often smaller than $n_1$. Relying only on these instances for training may lead to reduced prediction accuracy. Thus, rather than learning the cluster center directly by only applying the data adaptation algorithm, the importance of reusing the cluster center $\mathbf{U}^0$ from a previous stage needs to be emphasized. 

\section{Generalization Error bound Analysis}
\label{Generalization Error bound Analyses Section}
Expanding the theoretical framework of static clustering, this section derives generalization error bounds for the proposed FIC algorithms in dynamic environments and discusses their quantitative implications.

According to the widespread application and established generalization properties of $k$-means algorithm \cite{antos2005improved, tang2016lloyd, zhang2023imbalanced}, we establish the groundwork by introducing fundamental definitions. We denote a family of $k$-\emph{valued} functions as
\begin{equation}
\label{G_index}
\mathcal{G}_{\mathbf{U}}=\left\{g_{\mathbf{U}}=(g_{\mathbf{u}_{1}}, \ldots, g_{\mathbf{u}_{k}}): \mathbf{U} \in \mathcal{H}^{k}\right\},
\end{equation}
where $g_{\mathbf{u}_{s}}(\mathbf{x})=\lVert\mathbf{x}-\mathbf{u}_{s} \rVert^2$ with $s = 1,\ldots k$.
Let  $\psi :\mathbb{R}^{k}\rightarrow \mathbb{R}$ be a function that computes the minimum value:
$$	\psi (\boldsymbol{\alpha})=\min(\alpha_{1},\ldots, \alpha_{k}), \text{for any} \ \boldsymbol{\alpha} =(\alpha_{1},\ldots, \alpha_{k}) \in \mathbb{R}^{k}. $$
Denote a ``minimum'' family of functions $\mathcal{G}_{\mathbf{U}}$ as $\mathcal{J}_{\mathbf{U}}$:
\begin{equation}
\label{J_U}
\mathcal{J}_{\mathbf{U}}:=\left\{h_{\mathbf{U}}=\psi \circ g_{\mathbf{U}} \mid g_{\mathbf{U}} \in \mathcal{G}_{\mathbf{U}}, h_{\mathbf{U}}(\mathbf{x})=\psi(g_{\mathbf{U}}(\mathbf{x}))\right\}.
\end{equation}
Obviously, $\psi(g_{\mathbf{U}}(\mathbf{x}))= \min(g_{\mathbf{u}_{1}}(\mathbf{x}), \ldots, g_{\mathbf{u}_{k}}(\mathbf{x}))$, we then rewrite the empirical and expected clustering risk as	
$$
\hat{\mathcal{L}}_{D}(\mathbf{U}):=\frac{1}{n}\sum_{i =1}^{n}\psi(g_{\mathbf{U}}(\mathbf{x}_i));
$$
$$
\mathcal{L}_{\mathbb{Q}}(\mathbf{U}):=\mathbb{E}_{\mathbf{x}\sim \mathbb{Q}}\left[\psi(g_{\mathbf{U}}(\mathbf{x})) \right].
$$
\begin{definition}[\textbf{Clustering Rademacher Complexity (CRC)}]
Given a family of functions $\mathcal{J}_{\mathbf{U}}$ denoted by (\ref{J_U}), let a sample set $D=\{\mathbf{x}_{1},\dots,\mathbf{x}_{n}\}$ be drawn from the input space $\mathcal{X}$. Define the clustering empirical Rademacher complexity of $\mathcal{J}_{\mathbf{U}}$ w.r.t. $D$ as
\begin{equation}
\label{R_valueF}
\mathfrak{R}_{n}(\mathcal{J}_{\mathbf{U}})=\mathbb{E}_{\boldsymbol{\sigma}}[\sup _{h_{\mathbf{U}} \in \mathcal{J}_{\mathbf{U}}}\lvert\frac{1}{n}\sum_{i=1}^{n} \sigma_{i} h_{\mathbf{U}}(\mathbf{x}_{i})\rvert],
\end{equation}
where $\{\sigma_i\}_{i=1}^n$ are independent Rademacher variables. The expected CRC is defined as $\mathfrak{R}(\mathcal{J}_{\mathbf{U}}) = \mathbb{E}[\mathfrak{R}_n(\mathcal{J}_{\mathbf{U}})]$.
\end{definition}

First, we propose a generalization error bound for $k$-means algorithm in Theorem\,\ref{k-means  generalization error bound}, serving as a foundational analysis for FIC. This novel generalization bound is derived using a new tool, previously unexplored in $k$-means studies.
\begin{theorem} 
\label{k-means  generalization error bound}
Consider training the $k$-means clustering model in (\ref{k-meansOb}) using any sample set $D=\left\{\mathbf{x}_{1}, \ldots, \mathbf{x}_{n}\right\}$ of $n$ data points. Let $\mathbf{U}\in \mathcal{H}^{k}$ be all possible cluster centers. If $\lVert\mathbf{x}\rVert\leq \gamma$ for $ \forall\mathbf{x}\in \mathcal{X}$,
then for any $\delta \in(0,1)$ and $\forall \mathbf{U}\in \mathcal{H}^{k}$, there exists a constant $V$ such that, with probability at least $1-\delta$, we have
\begin{equation*}
\label{final generalization error inequality}
\begin{aligned}
\mathcal{L}_{\mathbb{Q}}(\mathbf{U})
\leq &\hat{\mathcal{L}}_{D}(\mathbf{U}) + \frac{\alpha+(2\sqrt{2}+3\beta)\sqrt{\log (1 / \delta)}}{\sqrt{n}} +\frac{6\gamma\log (1 / \delta)}{n},
\end{aligned}
\end{equation*}
where {\small $\beta = \sqrt{(\mathbb{E}_{\mathbf{x}\sim \mathbb{Q}} \left[\sup _{g_{\mathbf{U}}\in\mathcal{G}_{\mathbf{U}}}\hat{\mathcal{L}}_{D}(\mathbf{U}) \right]+4\alpha+8\sqrt{2\log (1 / \delta)})\gamma}$}
and $\alpha = 8\gamma^2V \sqrt {k}\log ^{2}(2\sqrt{2\gamma n}).$\\
\end{theorem}
 \noindent\textbf{See the proof of Theorem\,\ref{k-means  generalization error bound}} in Section \ref{k-means  generalization error bound proof}.
\begin{remark}
{\rm To enhance the presentation of the bound, we mainly focus on the leading terms, resulting in
$\mathcal{L}_{\mathbb{Q}}(\mathbf{U})
\leq \hat{\mathcal{L}}_{D}(\mathbf{U}) +\mathcal{O}(\sqrt{\frac{k}{n}}+\frac{1}{n}).$
Theorem\,\ref{k-means  generalization error bound} demonstrates that the generalization error of $k$-means can be bounded at an approximate order of $\tilde{\mathcal{O}}(\sqrt{\frac{k}{n}}+\frac{1}{n})$. Later, we provide the generalization bounds for the four proposed algorithms. }
\end{remark}

\textbf{For FIC-FT}, we suppose that the instances in the previous stage obey an unknown distribution $\mathbb{Q}_{p}$. Let $\mathbf{U}_1\in \mathcal{H}^{k}$ be all possible cluster centers, and denote $\mathcal{J}_{\mathbf{U}_1}$  as a ``minimum"  family. Then, a generalization bound of FIC-FT can be derived in Theorem\,\ref{Feature Tailoring theorem} based on the clustering model (\ref{Feature Tailoring_Ob}).
\begin{theorem}	[\textbf{FIC-FT}] 
\label{Feature Tailoring theorem}
Consider training the clustering model in (\ref{Feature Tailoring_Ob}) using the sample set $D_p \bigcup D_{c^p}$, which comprises $n_1+n_2$ instances. If $\lVert \mathbf{x}\rVert\leq \gamma$ for $ \forall\mathbf{x}\in \mathcal{X}_c$, then for any $\delta \in(0,1)$ and $\forall \mathbf{U}_1\in \mathcal{H}^{k}$, there exists a constant $V_1$ such that, with probability at least $1-\delta$, we obtain
\begin{equation*}
\begin{aligned}
\mathcal{L}_{\mathbb{Q}_{p}}(\mathbf{U}_{1})\leq &\hat{\mathcal{L}}_{D_p \bigcup D_{c^p}}(\mathbf{U}_{1}) 
+ \frac{\alpha_1+(2\sqrt{2}+3\beta_1)\sqrt{\log (1 / \delta)}}{\sqrt{n_1+n_2}}
\\&
+\frac{6\gamma\log (1 / \delta)}{n_1+n_2},
\end{aligned}
\end{equation*}
where $\alpha_1 = 8\gamma^2V_1 \sqrt {k}\log ^{2}(2(n_1+n_2)\sqrt{2\gamma (n_1+n_2)})$ and\\ 
{\small $\beta_1 = \sqrt{(\mathbb{E}_{\mathbf{x}\sim \mathbb{Q}_p} [\sup _{g_{\mathbf{U}_1}\in\mathcal{G}_{\mathbf{U}_1}}\hat{\mathcal{L}}_{D}(\mathbf{U}_1) ]+4\alpha_1+8\sqrt{2\log (1 / \delta)})\gamma}.$ }
\end{theorem}

\begin{remark}
{\rm The generalization error bound in Theorem\,\ref{Feature Tailoring theorem} is established by Theorem\,\ref{k-means  generalization error bound}. In FIC-FT, the training data $D_p \cup D_{c^p}$ consists exclusively of old features. Thus, the generalization bound is primarily determined by their clustering performance. Furthermore, if the representational capacity of these features is constrained, clustering improvement remains limited even with access to data from both stages.}
\end{remark}

\textbf{For FIC-DR}, the reconstructed data $\tilde{D}_p$ from the previous stage follows a distribution $\tilde{\mathbb{Q}}_{p}$. The current data $D_c$ is drawn from distribution $\mathbb{Q}_c$. Pioneering studies have proposed various measures to quantify distributional discrepancies \cite{kifer2004detecting,ben2010theory,mohri2012new,zhang2020learning}. Building upon the $\mathcal{Y}$-discrepancy \cite{mohri2012new}, we define a version tailored for FIC-DR.
\begin{definition}[\textbf{$\mathcal{Y}$-discrepancy}]
Let $\tilde{\mathbb{Q}}_p, \mathbb{Q}_{c}$ denote two given distributions. 
Given cluster centers $\mathbf{U} \in \mathcal{H}^{k}$, the $\mathcal{Y}$-discrepancy between $\tilde{\mathbb{Q}}_p$ and $\mathbb{Q}_c $ is defined as:
\begin{equation}
\operatorname{disc}_\mathcal{Y}(\tilde{\mathbb{Q}}_p, \mathbb{Q}_c)=\sup _{\mathbf{U} \in \mathcal{H}^{k}}\lvert	\mathcal{L}_{\tilde{\mathbb{Q}}_p}(\mathbf{U})-\mathcal{L}_{\mathbb{Q}_{c}}(\mathbf{U})\rvert.
\end{equation}
\end{definition}
Since only empirical data are available, we introduce weights $\boldsymbol{\omega}$ to the data $\tilde{D}_p$ sampled from distribution $\tilde{\mathbb{Q}}_p$. The resulting weighted empirical $\mathcal{Y}$-discrepancy is defined as:
\begin{equation}
\label{disc definition}
\operatorname{disc}_\mathcal{Y}(\tilde{D}_{p_{\boldsymbol{\omega}}}, D_c)=\sup _{\mathbf{U} \in \mathcal{H}^{k}}\lvert\hat{\mathcal{L}}_{\tilde{D}_{p_{\boldsymbol{\omega}}}}(\mathbf{U})-\hat{\mathcal{L}}_{D_c}(\mathbf{U})\rvert.
\end{equation}

By incorporating the weighted empirical $\mathcal{Y}$-discrepancy, Theorem\,\ref{Data Reconstruction generalization error} establish a generalization error bound for FIC-DR on distribution $\mathbb{Q}_{c}$. 
\begin{theorem}	[\textbf{FIC-DR}]
\label{Data Reconstruction generalization error}
Consider training the clustering model in (\ref{Data Reconstruction_Ob}) with the sample set $\tilde{D}_p \bigcup D_{c}$ consisting of $n_1+n_2$ instances. If $\lVert\mathbf{x}\rVert\leq \gamma$ for $ \forall\mathbf{x}\in \mathcal{X}_c$, then for any $\delta \in(0,1)$ and $\forall \mathbf{U}_2\in \mathcal{H}^{k}$, there exists a constant $V_2$ such that, with probability at least $1-\delta$, we have
\begin{equation*}
\label{Data Reconstruction g_error inequility}
\begin{aligned}
\mathcal{L}_{\mathbb{Q}_c}(\mathbf{U}_{2}) \leq &	 \hat{\mathcal{L}}_{\tilde{D}_{p_{\boldsymbol{\omega}}}}(\mathbf{U}_{2})+ \operatorname{disc}_\mathcal{Y}(\tilde{D}_{p_{\boldsymbol{\omega}}}, D_c) +\frac{6\gamma\log (1 / \delta)}{n_1+n_2}
\\&	
+ \frac{\alpha_2+(2\sqrt{2}+3\beta_2)\sqrt{\log (1 / \delta)}}{\sqrt{n_1+n_2}},
\end{aligned}
\end{equation*}
where $\alpha_2 = 8\gamma^2V_2\sqrt {k}\log ^{2}(2(n_1+n_2)\sqrt{2\gamma (n_1+n_2)})$ 	 and\\
{\small $\beta_2 = \sqrt{(\mathbb{E}_{\mathbf{x}\sim \mathbb{Q}_c} [\sup _{g_{\mathbf{U}_2}\in\mathcal{G}_{\mathbf{U}_2}}\hat{\mathcal{L}}_{D}(\mathbf{U}_2) ]+4\alpha_2+8\sqrt{2\log (1 / \delta)})\gamma}.$ }
\end{theorem}
 \noindent\textbf{See the proof of Theorem\,\ref{Data Reconstruction generalization error} in Section \ref{Data Reconstruction generalization error proof}}.

\begin{remark}{\rm This bound is related to the clustering error on the reconstructed distribution $\tilde{D}_{p_{\boldsymbol{\omega}}}$ and the weighted empirical $\mathcal{Y}$-discrepancy between $\tilde{D}_{p_{\boldsymbol{\omega}}}$ and $D_c$. The weights $\boldsymbol{\omega}$ capture distribution shifts, with $\boldsymbol{\omega} = \mathbf{1}$ representing the case of an  invariant distribution. This theoretical framework suggests that generalization can be optimized by refining the reconstruction function upon which $\tilde{D}_{p_{\boldsymbol{\omega}}}$ depends. Consequently, an effective reconstruction process minimizes the discrepancy between the reconstructed and observed distributions, thereby directly enhancing the generalization capacity of FIC-DR.}
\end{remark}

\textbf{For FIC-DA,} recall that the set $D_c$ consists of $n_2$ instances generated from the distribution $\mathbb{Q}_c$ in the current stage. Note that the test data follows the distribution $\mathbb{Q}_c$. 
By training only with the data $D_c$, we can derive the following generalization error bound of FIC-DA as stated in Theorem\,\ref{Data Adaptation  theorem}.
\begin{theorem}	[\textbf{FIC-DA}]
\label{Data Adaptation  theorem}
Consider training the clustering model in (\ref{Data Adaptation_Ob}) with a sample set $D_c $ consisting of  $n_2$ instances. If $\lVert\mathbf{x}\rVert\leq \gamma$ for $ \forall\mathbf{x}\in \mathcal{X}_c$, then for any $\delta \in(0,1)$ and $\forall \mathbf{U}_3\in \mathcal{H}^{k}$, there exists a constant $V_3$ such that, with probability at least $1-\delta$, we have
\begin{equation}
\begin{aligned}
\mathcal{L}_{\mathbb{Q}_{c}}(\mathbf{U}_{3})\leq &\hat{\mathcal{L}}_{D_c}(\mathbf{U}_{3}) 
+ \frac{\alpha_3+(2\sqrt{2}+3\beta_3)\sqrt{\log (1 / \delta)}}{\sqrt{n_2}}
\\&	+\frac{6\gamma\log (1 / \delta)}{n_2},
\end{aligned}
\end{equation}
where we denote $\alpha_3 = 8\gamma^2V_3 \sqrt {k}\log ^{2}(2n_2\sqrt{2\gamma n_2})$ and
$\beta_3 = \sqrt{(\mathbb{E}_{\mathbf{x}\sim \mathbb{Q}_c} [\sup _{g_{\mathbf{U}_3}\in\mathcal{G}_{\mathbf{U}_3}}\hat{\mathcal{L}}_{D}(\mathbf{U}_3) ]+4\alpha_3+8\sqrt{2\log (1 / \delta)})\gamma}.$ 
\end{theorem}

As Theorem\,\ref{k-means  generalization error bound} provides the foundation for feature incremental clustering, the generalization error bound in Theorem\,\ref{Data Adaptation  theorem} is established using a similar proof technique.
\begin{remark}
{\rm For FIC-DA, the generalization bound is primarily constrained by the limited sample size of the current stage. However, if new features offer performance gains, the bound can be further tightened by incorporating additional data to leverage this increased dimensionality.}
\end{remark}

\textbf{For FIC-MR}, let $\tilde{\mathbf{U}}_{4} =(\tilde{\mathbf{U}}_{4} ^{(1)}, \tilde{\mathbf{U}}_{4} ^{(2)})$ be the empirical risk minimizer of Eq.\,(\ref{modelresue_Ob}) according to Definition \ref{emprical center define}. We define the hybrid center as $\hat{\mathbf{U}}_{4}=(\mathbf{U}^{0},\tilde{\mathbf{U}}_{4}^{(2)}),$ ntegrates the optimized cluster center from the previous stage ($\mathbf{U}^{0}$) with the newly optimized components ($\tilde{\mathbf{U}}_{4}^{(2)}$). The generalization error bound for FIC-MR is then presented in Theorem\,\ref{Model Resue theorem}.
\begin{theorem}	 [\textbf{FIC-MR}]
\label{Model Resue theorem}
Consider training the clustering model in (\ref{modelresue_Ob}) with a sample set  $D_c  $ comprising $n_2$ instances, where $\mathbf{U}^{0}$ is a well-trained cluster center from the previous stage.
If $\lVert\mathbf{x}\rVert\leq \gamma$ for $ \forall\mathbf{x}\in \mathcal{X}_c$, then for any $\delta \in(0,1)$ and $\forall \mathbf{U}_4\in \mathcal{H}^{k}$, there exists a constant $V_4$ such that, with probability at least $1-\delta$, we have
\begin{equation}
\label{Model Resue inequality}
\begin{aligned}
\mathcal{L}_{\mathbb{Q}_{c}}(\mathbf{U}_4)\leq &\hat{\mathcal{L}}_{D_c}(\mathbf{U}_4) 
+\frac{\alpha_4+3\beta_4\sqrt{\log (1 / \delta)}}{\sqrt{n_2}}
+\frac{6\gamma\log (1 / \delta)}{n_2},
\end{aligned}
\end{equation}
where $\alpha_4 =4V_4\varepsilon \sqrt {k}\log ^{2}(2n_2\sqrt{(\gamma+\gamma_4) n_2})$,  $\varepsilon= \gamma(\sqrt{\frac{\Gamma(\mathbf{U}^0)  }{\theta}}+ \sqrt{\gamma_4^{2}- (\gamma_0-\sqrt{\frac{\Gamma(\mathbf{U}^0)  }{\theta}})^{2}}+\frac{\gamma }{2})  + \frac{(\gamma_4' +\gamma_0 ) \sqrt{\frac{\Gamma(\mathbf{U}^0)  }{\theta}} +\gamma_4''^2}{2}$. 
The term $\Gamma(\mathbf{U}^0) = \mathbb{E}[\Gamma_{n_2}(\mathbf{U}^0)]$ quantifies the expected clustering risk under the current data distribution, where the expectation is taken with respect to the empirical risk $\Gamma_{n_2}(\mathbf{U}^0)$ computed on the current feature matrix $\mathbf{X}_2^{(1)}$. For $\mathbf{U}_4=[\mathbf{u}_{1},\ldots,\mathbf{u}_{k}]$, assume $\lVert\mathbf{u}_{s}\rVert \leq \gamma_4$,  $\lVert \mathbf{u}_{s}^{(1)}\rVert \leq \gamma_4' \leq\gamma_4 $, $\lVert \mathbf{u}_{s}^{(2)}\rVert \leq \gamma_4''\leq\gamma_4$, for $s=1, \ldots, k$. Furthermore, $\beta_4 = \sqrt{(\mathcal{L}_{Q_c}(\hat{\mathbf{U}}_4)+4\alpha_4)\gamma}$, where $ \mathcal{L}_{Q_c}(\hat{\mathbf{U}}_4)$ denotes the expected clustering risk of the cluster center $\hat{\mathbf{U}}_4$ under the current distribution $\mathbb{Q}_c$.
\end{theorem}
\noindent\textbf{See the proof of Theorem\,\ref{Model Resue theorem} in Section \ref{Model Resue theorem proof}}.
\begin{remark}
{\rm In Theorem\,\ref{Model Resue theorem}, FIC-MR achieves a generalization error bound with a convergence rate of $\tilde{\mathcal{O}}(\sqrt{\frac{k}{n_2}}+\frac{1}{n_2})$, which is influenced by the adaptation term $\Gamma(\mathbf{U}^0)$. A lower $\Gamma(\mathbf{U}^0)$ implies that $\mathbf{U}^0$ lies in a low-risk region of the current hypothesis space.
Notably, when the cluster centers $\mathbf{U}^0$ align well with the current distribution (i.e., $\Gamma(\mathbf{U}^0) \to 0$), the bound in \eqref{Model Resue inequality} tightens to a fast rate of 
$\tilde{\mathcal{O}}(\frac{1}{n_2})$.
}
\end{remark}

\begin{table*}[htb]
	\footnotesize
	\centering
	\caption{A brief description of data set.}
	\vskip -0.1 in
	\label{datasetdescribe}
	\renewcommand{\arraystretch}{0.8}  
	\setlength{\tabcolsep}{1.4mm}{ 
		\begin{tabular}{c c c c c c c c c c c c c}
			\toprule[1.5pt]
			Information      &Dermatology   &Caltech101-7  & Indoor & Event  & Handwritten    &Scene8   & Wiki&KSA   & Mnist \\
			\midrule[0.75pt]
			Sample Size                 &366   &441    &621      &1579     &2000      &2688     &2866   &20000        &100000\\
			Previous Features     &11      &48      &32         &500        &76           &512        &128     &360             &256\\
			Current Features       &22      &411   &512		  &512        &216         &200        &10        &45               &32\\
			Class            				  &6         & 7      &5           &8            &10           &8            &7          & 5                &10\\
			\bottomrule[1.5pt]
		\end{tabular}
	}
\end{table*}

\section{Experiments}  
\label{Experiments Section}
In this section, we validate our theory through experiments on real-world datasets. We compare four proposed incremental clustering algorithms with relevant baselines, analyze the effects of feature increments and the reuse parameter $\theta$, and demonstrate practical utility by an activity recognition task.

\subsection{Dataset Description and Experimental Setup}
\label{}
\subsubsection{Dataset Description}

We use publicly available benchmark datasets: Dermatology, Caltech101, Indoor, Event, Handwritten, Scene8, Wiki, KSA, and MNIST, with 366–100,000 instances and 5–10 classes. Unlike traditional settings with fixed feature space, we consider a feature-incremental clustering scenario where data arrive in two stages with newly introduced features. Specifically, one feature type is treated as previous and another as current; for example, MNIST uses 256-dimensional LBP features initially, followed by 32-dimensional WT features. We offer detailed descriptions of the datasets utilized in this paper, along with a brief overview listed in Table \ref{datasetdescribe}.

\renewcommand
\tabcolsep{7.7pt}   
\begin{table*}[htbp]
	\footnotesize
	\centering
	\caption{The average  performance results of  all methods with different RAs. The best results are highlighted in boldface.}
	\label{one table}
	\vskip -0.08 in
	\begin{adjustbox}{center}  
		\begin{tabular}{c|c|c|cc|cccc}
			\toprule[1.5pt]
			\midrule[0.75pt]
			Dataset & Evaluation Criteria & RA & KM-P1 & KM-C1  & FIC-FT & FIC-DR& FIC-DA & FIC-MR  \\
			\midrule[0.75pt]
			\multirow{13}{*}{Dermatology} 
			& \multirow{5}{*}{ACC}
			&10\%	&0.721(0.058)	&0.652(0.054)	&0.739(0.083)	&\textbf{0.772(0.055)}	&0.717(0.078)	&0.751(0.086)\\  
			&&20\%	&0.747(0.050)	&0.668(0.067)	&0.741(0.068)	&\textbf{0.789(0.076)}	&0.756(0.087)	&0.786(0.057)\\  
			&&30\%	&0.747(0.071)	&0.681(0.075)	&0.765(0.074)	&0.771(0.077)	&0.741(0.070)	&\textbf{0.801(0.045)}\\  
			&&40\%	&0.738(0.069)	&0.706(0.079)	&0.739(0.074)	&0.814(0.104)	&0.793(0.084)	&\textbf{0.817(0.072)}\\  
			\cline{2-9}	
			& \multirow{5}{*}{F-score}
			&10\%	&0.649(0.051)	&0.622(0.077)	&0.666(0.083)	&0.766(0.064)	&0.747(0.088)	&\textbf{0.778(0.085)}\\  
			&&20\%	&0.665(0.047)	&0.634(0.069)	&0.666(0.067)	&\textbf{0.807(0.070)}	&0.758(0.076)	&0.794(0.066)\\  
			&&30\%	&0.673(0.064)	&0.624(0.055)	&0.673(0.073)	&0.799(0.077)	&0.737(0.080)	&\textbf{0.807(0.052)}\\  
			&&40\%	&0.660(0.065)	&0.652(0.076)	&0.679(0.069)	&0.812(0.098)	&0.817(0.079)	&\textbf{0.843(0.066)}\\  
			\cline{2-9}
			& \multirow{5}{*}{NMI}
			&10\%	&0.692(0.039)	&0.655(0.045)	&0.695(0.069)	&0.792(0.047)	&0.789(0.038)	&\textbf{0.805(0.038)}\\  
			&&20\%	&0.699(0.038)	&0.665(0.041)	&0.697(0.053)	&\textbf{0.838(0.041)}	&0.815(0.046)	&0.836(0.047)\\  
			&&30\%	&0.704(0.044)	&0.633(0.043)	&0.694(0.058)	&0.809(0.072)	&0.782(0.066)	&\textbf{0.838(0.042)}\\  
			&&40\%	&0.695(0.036)	&0.686(0.054)	&0.701(0.050)	&0.848(0.060)	&0.859(0.038)	&\textbf{0.878(0.043)}\\  
			\midrule[0.75pt]
			
			\multirow{13}{*}{Caltech101-7} 
			& \multirow{5}{*}{ACC}
			&10\%	&0.459(0.091)	&0.436(0.031)	&0.483(0.070)	&0.491(0.057)	&0.549(0.041)	&\textbf{0.567(0.053)}\\  
			&&20\%	&0.429(0.103)	&0.367(0.087)	&0.444(0.087)	&0.504(0.055)	&\textbf{0.570(0.048)}	&0.562(0.045)\\  
			&&30\%	&0.466(0.070)	&0.402(0.052)	&0.444(0.095)	&0.525(0.068)	&0.568(0.069)	&\textbf{0.597(0.067)}\\  
			&&40\%	&0.489(0.064)	&0.434(0.057)	&0.473(0.068)	&0.522(0.058)	&0.553(0.037)	&\textbf{0.564(0.064)}\\   
			\cline{2-9}	
			& \multirow{5}{*}{F-score}
			&10\%	&0.413(0.061)	&0.393(0.028)	&0.422(0.049)	&\textbf{0.496(0.054)}	&0.481(0.036)	&0.492(0.028)\\  
			&&20\%	&0.392(0.069)	&0.338(0.052)	&0.403(0.061)	&\textbf{0.545(0.050)}	&0.535(0.023)	&0.535(0.028)\\  
			&&30\%	&0.419(0.054)	&0.361(0.048)	&0.403(0.067)	&0.568(0.060)	&0.564(0.045)	&\textbf{0.580(0.049)}\\  
			&&40\%	&0.432(0.047)	&0.390(0.057)	&0.422(0.058)	&0.528(0.064)	&0.551(0.039)	&\textbf{0.551(0.049)}\\  
			\cline{2-9}
			& \multirow{5}{*}{NMI}
			&10\%	&0.417(0.082)	&0.392(0.028)	&0.429(0.060)	&0.504(0.055)	&\textbf{0.535(0.020)}	&0.535(0.014)\\  
			&&20\%	&0.382(0.099)	&0.279(0.066)	&0.403(0.077)	&\textbf{0.574(0.055)}	&0.563(0.029)	&0.567(0.030)\\  
			&&30\%	&0.419(0.075)	&0.332(0.064)	&0.408(0.090)	&0.600(0.070)	&0.614(0.039)	&\textbf{0.637(0.043)}\\  
			&&40\%	&0.438(0.065)	&0.369(0.067)	&0.430(0.078)	&0.568(0.061)	&0.599(0.043)	&\textbf{0.617(0.043)}\\  
			\midrule[0.75pt]

			\multirow{13}{*}{Indoor} 
			& \multirow{5}{*}{ACC}
			&10\%	&0.373(0.025)	&0.377(0.031)	&0.372(0.025)	&0.348(0.060)	&0.398(0.056)	&\textbf{0.455(0.042)}\\  
			&&20\%	&0.367(0.024)	&0.350(0.033)	&0.367(0.026)	&0.324(0.055)	&0.467(0.051)	&\textbf{0.482(0.045)}\\  
			&&30\%	&0.361(0.017)	&0.380(0.015)	&0.371(0.023)	&0.302(0.032)	&0.478(0.053)	&\textbf{0.499(0.048)}\\  
			&&40\%	&0.363(0.024)	&0.360(0.033)	&0.375(0.018)	&0.287(0.030)	&0.477(0.027)	&\textbf{0.485(0.049)}\\  
			\cline{2-9}
			& \multirow{5}{*}{F-score}
			&10\%	&0.302(0.012)	&0.312(0.023)	&0.295(0.015)	&0.358(0.042)	&0.347(0.037)	&\textbf{0.381(0.032)}\\  
			&&20\%	&0.300(0.014)	&0.284(0.024)	&0.293(0.013)	&0.353(0.034)	&0.376(0.051)	&\textbf{0.396(0.039)}\\  
			&&30\%	&0.297(0.013)	&0.301(0.011)	&0.287(0.009)	&0.365(0.041)	&0.386(0.056)	&\textbf{0.400(0.039)}\\  
			&&40\%	&0.293(0.012)	&0.297(0.018)	&0.302(0.014)	&0.372(0.033)	&0.377(0.023)	&\textbf{0.381(0.032)}\\  
			\cline{2-9}
			& \multirow{5}{*}{NMI}
			&10\%	&0.173(0.026)	&0.182(0.038)	&0.174(0.020)	&0.190(0.082)	&0.224(0.059)	&\textbf{0.286(0.045)}\\  
			&&20\%	&0.165(0.028)	&0.137(0.022)	&0.172(0.030)	&0.170(0.088)	&0.272(0.051)	&\textbf{0.294(0.042)}\\  
			&&30\%	&0.155(0.021)	&0.193(0.015)	&0.168(0.026)	&0.172(0.112)	&0.302(0.070)	&\textbf{0.316(0.060)}\\  
			&&40\%	&0.162(0.025)	&0.167(0.030)	&0.181(0.019)	&0.187(0.119)	&0.294(0.032)	&\textbf{0.299(0.037)}\\ 
			\midrule[0.75pt]

			\multirow{13}{*}{Event} 
			& \multirow{5}{*}{ACC}
			&10\%	&0.325(0.027)	&0.296(0.020)	&0.315(0.025)	&0.326(0.028)	&0.348(0.027)	&\textbf{0.351(0.034)}\\  
			&&20\%	&0.304(0.019)	&0.302(0.018)	&0.305(0.020)	&0.339(0.032)	&0.367(0.041)	&\textbf{0.376(0.037)}\\  
			&&30\%	&0.305(0.022)	&0.325(0.013)	&0.313(0.024)	&0.343(0.021)	&0.364(0.026)	&\textbf{0.379(0.041)}\\  
			&&40\%	&0.314(0.021)	&0.322(0.019)	&0.313(0.013)	&0.369(0.033)	&0.383(0.033)	&\textbf{0.384(0.028)}\\   
			\cline{2-9}	
			& \multirow{5}{*}{F-score}
			&10\%	&0.244(0.019)	&0.231(0.016)	&0.240(0.019)	&0.247(0.020)	&0.250(0.022)	&\textbf{0.255(0.017)}\\  
			&&20\%	&0.229(0.011)	&0.230(0.010)	&0.233(0.013)	&0.255(0.018)	&0.273(0.026)	&\textbf{0.275(0.020)}\\  
			&&30\%	&0.234(0.015)	&0.245(0.013)	&0.239(0.018)	&0.261(0.015)	&0.275(0.013)	&\textbf{0.282(0.023)}\\  
			&&40\%	&0.236(0.016)	&0.241(0.013)	&0.235(0.011)	&0.266(0.017)	&\textbf{0.273(0.019)}	&0.273(0.016)\\   
			\cline{2-9}
			& \multirow{5}{*}{NMI}
			&10\%	&0.241(0.031)	&0.205(0.028)	&0.233(0.032)	&0.237(0.025)	&0.238(0.028)	&\textbf{0.246(0.028)}\\  
			&&20\%	&0.212(0.022)	&0.224(0.025)	&0.220(0.024)	&0.245(0.024)	&0.265(0.032)	&\textbf{0.269(0.027)}\\  
			&&30\%	&0.224(0.025)	&0.248(0.019)	&0.236(0.033)	&0.251(0.021)	&0.266(0.015)	&\textbf{0.272(0.029)}\\  
			&&40\%	&0.225(0.029)	&0.249(0.018)	&0.229(0.021)	&0.261(0.024)	&\textbf{0.273(0.023)}	&0.272(0.020)\\  
			\midrule[0.75pt]

			\multirow{13}{*}{Handwritten} 
			& \multirow{5}{*}{ACC}
			&10\%	&0.533(0.052)	&0.478(0.022)	&0.515(0.051)	&0.572(0.063)	&0.621(0.069)	&\textbf{0.651(0.042)}\\  
			&&20\%	&0.506(0.057)	&0.571(0.058)	&0.497(0.039)	&0.634(0.058)	&0.662(0.076)	&\textbf{0.690(0.062)}\\  
			&&30\%	&0.519(0.039)	&0.529(0.043)	&0.547(0.052)	&0.606(0.090)	&0.694(0.069)	&\textbf{0.739(0.070)}\\  
			&&40\%	&0.507(0.061)	&0.620(0.052)	&0.544(0.040)	&0.675(0.077)	&0.732(0.063)	&\textbf{0.740(0.062)}\\ 
			\cline{2-9}	
			& \multirow{5}{*}{F-score}
			&10\%	&0.449(0.033)	&0.391(0.022)	&0.443(0.042)	&0.537(0.046)	&0.544(0.054)	&\textbf{0.571(0.035)}\\  
			&&20\%	&0.437(0.028)	&0.450(0.026)	&0.426(0.028)	&0.586(0.040)	&0.611(0.059)	&\textbf{0.628(0.041)}\\  
			&&30\%	&0.444(0.027)	&0.440(0.026)	&0.453(0.031)	&0.586(0.054)	&0.635(0.053)	&\textbf{0.662(0.048)}\\  
			&&40\%	&0.437(0.040)	&0.505(0.034)	&0.463(0.023)	&0.634(0.054)	&0.669(0.049)	&\textbf{0.681(0.042)}\\  
			\cline{2-9}
			& \multirow{5}{*}{NMI}
			&10\%	&0.555(0.034)	&0.478(0.024)	&0.544(0.038)	&0.654(0.037)	&0.655(0.046)	&\textbf{0.684(0.027)}\\  
			&&20\%	&0.539(0.027)	&0.555(0.025)	&0.534(0.019)	&0.694(0.037)	&0.712(0.046)	&\textbf{0.724(0.031)}\\  
			&&30\%	&0.542(0.024)	&0.521(0.024)	&0.545(0.029)	&0.692(0.039)	&0.728(0.039)	&\textbf{0.747(0.031)}\\  
			&&40\%	&0.540(0.031)	&0.600(0.027)	&0.564(0.019)	&0.728(0.040)	&0.754(0.034)	&\textbf{0.764(0.029)}\\  
			\bottomrule[1.5pt]
		\end{tabular}
	\end{adjustbox}
\end{table*}

\renewcommand
\tabcolsep{9pt}   
\begin{table*}[!]
	\footnotesize
	\renewcommand
	\arraystretch{0.98}
	\caption{The average performance results of all methods with different RAs. The best results are highlighted in boldface.}
	\label{two table}
	\centering
	\vskip -0.08 in
	\begin{adjustbox}{center}
		\begin{tabular}{c|c|c|cc|cccc}
			\toprule[1.5pt]
			\midrule[0.75pt]
			Dataset & Evaluation Criteria & RA & KM-P1 & KM-C1  & FIC-FT & FIC-DR& FIC-DA & FIC-MR  \\
			\midrule[0.75pt]
			
			\multirow{13}{*}{Scene8} 
			& \multirow{5}{*}{ACC}
			&10\%	&0.558(0.036)	&0.490(0.051)	&0.568(0.042)	&0.572(0.055)	&0.511(0.050)	&\textbf{0.584(0.040)}\\  
			&&20\%	&0.578(0.037)	&0.494(0.026)	&0.564(0.034)	&0.579(0.046)	&0.533(0.043)	&\textbf{0.598(0.044)}\\  
			&&30\%	&0.564(0.034)	&0.512(0.039)	&0.555(0.036)	&0.556(0.048)	&0.548(0.039)	&\textbf{0.586(0.028)}\\  
			&&40\%	&0.569(0.040)	&0.521(0.031)	&0.549(0.037)	&0.573(0.039)	&0.577(0.029)	&\textbf{0.590(0.034)}\\  
			\cline{2-9}	
			& \multirow{5}{*}{F-score}
			&10\%	&0.435(0.020)	&0.382(0.031)	&0.440(0.023)	&0.458(0.030)	&0.415(0.034)	&\textbf{0.473(0.019)}\\  
			&&20\%	&0.442(0.021)	&0.386(0.015)	&0.440(0.021)	&0.464(0.022)	&0.436(0.032)	&\textbf{0.479(0.028)}\\  
			&&30\%	&0.434(0.019)	&0.407(0.027)	&0.430(0.021)	&0.453(0.031)	&0.440(0.026)	&\textbf{0.473(0.015)}\\  
			&&40\%	&0.442(0.023)	&0.403(0.020)	&0.430(0.016)	&0.464(0.018)	&0.457(0.020)	&\textbf{0.474(0.021)}\\ 
			\cline{2-9}
			& \multirow{5}{*}{NMI}
			&10\%	&0.472(0.018)	&0.409(0.031)	&0.472(0.020)	&0.499(0.029)	&0.454(0.038)	&\textbf{0.517(0.017)}\\  
			&&20\%	&0.472(0.020)	&0.409(0.020)	&0.471(0.021)	&0.503(0.021)	&0.479(0.034)	&\textbf{0.515(0.032)}\\  
			&&30\%	&0.465(0.016)	&0.431(0.027)	&0.466(0.016)	&0.500(0.030)	&0.482(0.024)	&\textbf{0.512(0.014)}\\  
			&&40\%	&0.473(0.019)	&0.428(0.019)	&0.461(0.018)	&0.502(0.024)	&0.493(0.025)	&\textbf{0.517(0.024)}\\ 
			\midrule[0.75pt]
			
			\multirow{13}{*}{Wiki} 
			& \multirow{5}{*}{ACC}
			&10\%	&0.188(0.007)	&0.177(0.010)	&0.189(0.005)	&0.299(0.066)	&0.511(0.068)	&\textbf{0.543(0.048)}\\  
			&&20\%	&0.186(0.005)	&0.192(0.007)	&0.187(0.010)	&0.308(0.069)	&0.547(0.045)	&\textbf{0.552(0.041)}\\  
			&&30\%	&0.188(0.006)	&0.186(0.008)	&0.186(0.007)	&0.366(0.054)	&0.532(0.057)	&\textbf{0.560(0.051)}\\  
			&&40\%	&0.187(0.006)	&0.196(0.008)	&0.189(0.006)	&0.409(0.039)	&0.530(0.074)	&\textbf{0.560(0.048)}\\  
			\cline{2-9}	
			& \multirow{5}{*}{F-score}
			&10\%	&0.140(0.006)	&0.152(0.007)	&0.142(0.005)	&0.372(0.052)	&0.441(0.031)	&\textbf{0.445(0.032)}\\  
			&&20\%	&0.140(0.006)	&0.136(0.005)	&0.139(0.006)	&0.390(0.053)	&\textbf{0.473(0.021)}	&0.471(0.022)\\  
			&&30\%	&0.139(0.006)	&0.132(0.006)	&0.139(0.007)	&0.395(0.034)	&0.450(0.017)	&\textbf{0.460(0.019)}\\  
			&&40\%	&0.140(0.005)	&0.137(0.005)	&0.141(0.005)	&0.446(0.049)	&0.461(0.031)	&\textbf{0.462(0.019)}\\ 
			\cline{2-9}
			& \multirow{5}{*}{NMI}
			&10\%	&0.073(0.006)	&0.062(0.007)	&0.072(0.005)	&0.490(0.036)	&0.529(0.019)	&\textbf{0.539(0.010)}\\  
			&&20\%	&0.073(0.003)	&0.070(0.004)	&0.073(0.006)	&0.503(0.042)	&0.558(0.014)	&\textbf{0.559(0.009)}\\  
			&&30\%	&0.074(0.005)	&0.068(0.006)	&0.071(0.004)	&0.519(0.029)	&0.536(0.012)	&\textbf{0.541(0.015)}\\  
			&&40\%	&0.074(0.004)	&0.076(0.005)	&0.073(0.006)	&0.542(0.028)	&0.546(0.023)	&\textbf{0.548(0.008)}\\ 
			\midrule[0.75pt]

			\multirow{13}{*}{KSA} 
			& \multirow{5}{*}{ACC}
			&10\%	&0.628(0.053)	&0.602(0.033)	&0.609(0.041)	&0.612(0.044)	&0.603(0.037)	&\textbf{0.646(0.051)}\\  
			&&20\%	&0.625(0.043)	&0.620(0.043)	&0.631(0.052)	&0.636(0.042)	&0.617(0.038)	&\textbf{0.638(0.037)}\\  
			&&30\%	&0.622(0.045)	&0.613(0.029)	&0.627(0.037)	&0.626(0.041)	&0.625(0.025)	&\textbf{0.637(0.050)}\\  
			&&40\%	&0.622(0.031)	&0.595(0.024)	&0.634(0.035)	&0.614(0.043)	&0.604(0.029)	&\textbf{0.636(0.038)}\\  
			\cline{2-9}
			
			& \multirow{5}{*}{F-score}
			&10\%	&0.563(0.030)	&0.532(0.032)	&\textbf{0.567(0.027)}	&0.540(0.033)	&0.528(0.023)	&0.560(0.038)\\  
			&&20\%	&\textbf{0.564(0.032)}	&0.534(0.037)	&0.555(0.027)	&0.560(0.029)	&0.542(0.036)	&0.556(0.033)\\  
			&&30\%	&\textbf{0.555(0.034)}	&0.537(0.028)	&0.551(0.028)	&0.549(0.033)	&0.543(0.031)	&0.545(0.043)\\  
			&&40\%	&\textbf{0.563(0.032)}	&0.521(0.037)	&0.557(0.024)	&0.543(0.039)	&0.516(0.033)	&0.548(0.042)\\  
			\cline{2-9}
			
			& \multirow{5}{*}{NMI}
			&10\%	&0.564(0.036)	&0.521(0.034)	&\textbf{0.574(0.037)}	&0.546(0.033)	&0.514(0.027)	&0.552(0.036)\\  
			&&20\%	&\textbf{0.567(0.036)}	&0.527(0.031)	&0.553(0.025)	&0.555(0.028)	&0.530(0.036)	&0.544(0.029)\\  
			&&30\%	&\textbf{0.565(0.037)}	&0.522(0.028)	&0.553(0.027)	&0.554(0.032)	&0.529(0.029)	&0.538(0.037)\\  
			&&40\%	&\textbf{0.570(0.037)}	&0.511(0.032)	&0.554(0.019)	&0.548(0.030)	&0.506(0.027)	&0.543(0.041)\\  
			\midrule[0.75pt]

			\multirow{13}{*}{Mnist} 
			& \multirow{5}{*}{ACC}
			&10\%	&0.280(0.035)	&0.313(0.011)	&0.284(0.018)	&0.327(0.048)	&\textbf{0.534(0.033)}	&0.532(0.021)\\  
			&&20\%	&0.278(0.024)	&0.320(0.015)	&0.277(0.019)	&0.453(0.040)	&0.549(0.031)	&\textbf{0.554(0.028)}\\  
			&&30\%	&0.274(0.021)	&0.332(0.011)	&0.295(0.025)	&0.405(0.064)	&0.512(0.018)	&\textbf{0.561(0.032)}\\  
			&&40\%	&0.288(0.021)	&0.320(0.012)	&0.289(0.013)	&0.466(0.061)	&0.538(0.045)	&\textbf{0.563(0.027)}\\  
			\cline{2-9}
			
			& \multirow{5}{*}{F-score}
			&10\%	&0.228(0.005)	&0.227(0.010)	&0.230(0.006)	&0.249(0.033)	&\textbf{0.405(0.019)}	&0.398(0.020)\\  
			&&20\%	&0.231(0.003)	&0.227(0.007)	&0.231(0.006)	&0.328(0.028)	&0.418(0.025)	&\textbf{0.419(0.023)}\\  
			&&30\%	&0.229(0.004)	&0.232(0.007)	&0.229(0.003)	&0.305(0.048)	&0.388(0.010)	&\textbf{0.419(0.025)}\\  
			&&40\%	&0.231(0.004)	&0.230(0.006)	&0.228(0.004)	&0.349(0.041)	&0.410(0.032)	&\textbf{0.430(0.023)}\\  
			\cline{2-9}
			
			& \multirow{5}{*}{NMI}
			&10\%	&0.206(0.027)	&0.241(0.009)	&0.210(0.014)	&0.249(0.046)	&\textbf{0.452(0.015)}	&0.447(0.017)\\  
			&&20\%	&0.204(0.021)	&0.247(0.006)	&0.203(0.018)	&0.352(0.037)	&0.462(0.022)	&\textbf{0.464(0.019)}\\  
			&&30\%	&0.199(0.017)	&0.254(0.008)	&0.214(0.020)	&0.324(0.062)	&0.438(0.006)	&\textbf{0.461(0.018)}\\  
			&&40\%	&0.213(0.017)	&0.248(0.007)	&0.213(0.011)	&0.374(0.052)	&0.455(0.028)	&\textbf{0.475(0.019)}\\  
			\midrule[0.75pt]
			\bottomrule[1.5pt]
		\end{tabular}
	\end{adjustbox}
\end{table*}

\begin{itemize}
	\item[$\bullet$] 
	\textbf{Dermatology}\footnote{https://archive.ics.uci.edu/ml/datasets/dermatology} is a disease dataset on erythemato-squamous, consisting of 366 instances distributed across 6 classes. The Dermatology dataset's 11 comprehensive clinical attributes are utilized as the previous features, while 22 histopathological attributes are employed as the current features. 
	\item[$\bullet$] 
	\textbf{Caltech101}\footnote{http://www.vision.caltech.edu/Image Datasets/Caltech101/} is an object recognition dataset consisting of 8677 images across 101 categories. A downsized version of this dataset is commonly employed in experiments. In our study, we concentrate on 7 widely utilized categories \cite{dueck2007non}: Dolla-Bill, Faces, Motorbikes, Stop-Sign, Garfield, Windsor-Chair, and Snoopy. Then, a smaller dataset is generated, containing a total of 441 images across these 7 classes, referred to as Caltech101-7. Each image in this dataset is characterized by 6 distinct types of extracted features. We take one type of feature with 48 attributes as the previous features, while another type with 411 attributes serves as the current features.

	\item[$\bullet$] 
	
	\textbf{Indoor}\footnote{http://web.mit.edu/torralba/www/indoor.html} is a scene recognition dataset containing 67 indoor categories. As reported in \cite{tao2018multiview}, 5 specific categories (auditorium, elevator, classroom, buffet, and cloister) have been selected for analysis. The newly generated dataset consists of 621 instances. Two visual features, namely 32 WT and 512 GIST, are utilized for the previous and current features, respectively.
	
	\item[$\bullet$] 
	\textbf{Event}\footnote{vision.stanford.edu/lijiali/event dataset/} is a sports event dataset
	containing 1579 images categorized into 8 distinct categories. We set 500 SIFT \cite{SIFT} as the previous features and 512 GIST as the current features.
	
	\item[$\bullet$]   
	\textbf{Handwritten}\footnote{https://archive.ics.uci.edu/ml/datasets/Multiple+Features} contains 2000 data points distributed among ten-digit classes. Two types of features, namely 76 Fourier coefficients of the character shapes (FOU) and 216 profile correlations (FAC), are employed as the previous and current features, respectively.
	
	\item[$\bullet$] 
	\textbf{Scene8}\footnote{https://figshare.com/articles/15-Scene\_Image\_Dataset/7007177} is composed of 2688 images categorized into 8 outdoor scene classes. We utilize two types of features: 512 GIST and 200 SURF features \cite{SURF}, designated as the previous and current features, respectively.
	
	\item[$\bullet$] 
	\textbf{Wiki}\footnote{http://www.svcl.ucsd.edu/projects/crossmodal/} is an image-text dataset comprising 2,866 samples distributed among 10 categories. These samples were sourced from articles acknowledged as featured content on Wikipedia. Two types of features are set as the previous and current features, respectively.
	
	\item[$\bullet$] 
	\textbf{KSA}\footnote{http://www.cs.cmu.edu/ kevinma/} is a Kinect Skeleton Action dataset, comprising 20,000 data points captured from four subjects performing five actions. Each action is treated as a distinct category, with 5000 video frames recorded for each subject. Two pose features, namely 360 $f_{JJ\_d}$ and 45 $f_{J\_c}$, are used for the previous and current features, respectively.

	\item[$\bullet$] 
	\textbf{Mnist}\footnote{https://www.csie.ntu.edu.tw/cjlin/libsvmtools/datasets
		/multiclass.html} is a dataset of handwritten digits, consisting of 810,000 images across 10 classes. In this study, we randomly select 10,000 data points from each class to create a subset containing 100,000 data points. For the previous features, we utilize 256 LBP features, while 32 WT features are employed as the current features.
\end{itemize}

\subsubsection{Experimental Setup}

We evaluate the proposed incremental algorithms with theoretical guarantees (\textbf{FIC-FT/DR/DA/MR}, introduced in Section\,\ref{Feature Incremental Clustering Section}). In practical scenarios, adding new features generally improves model performance if they contain substantial predictive signals and are not affected by noise or redundancy. To investigate the impact of integrating new features , we compare them against two baselines in an ablation study: \textbf{KM-P1} (using previous data $\mathbf{X}_1^{(1)}$) and \textbf{KM-C1} (using current feature subset $\mathbf{X}_2^{(1)}$).

Each dataset is randomly partitioned into three parts: 50\% for previous-stage training, 20\% for testing, and the remainder for the current stage. To evaluate the impact of data volume, we subsample the current-stage training set using various ratios (RA), specifically setting RA to 10\%, 20\%, 30\%, and 40\%.
For all compared algorithms, we only need to tune a single parameter $\theta$ in \textbf{FIC-MR}, searching over the grid $\{10^0, 10^1, 10^2, 10^3\}$. Experimental results are reported as the mean (standard deviation) over 10 independent runs, evaluated by Accuracy (ACC), F-score, and Normalized Mutual Information (NMI).
\begin{figure*}[tbhp]
	\begin{center}	
		\subfigure[Dermatology]{
			\includegraphics[width=5.5cm,height=4.3cm]{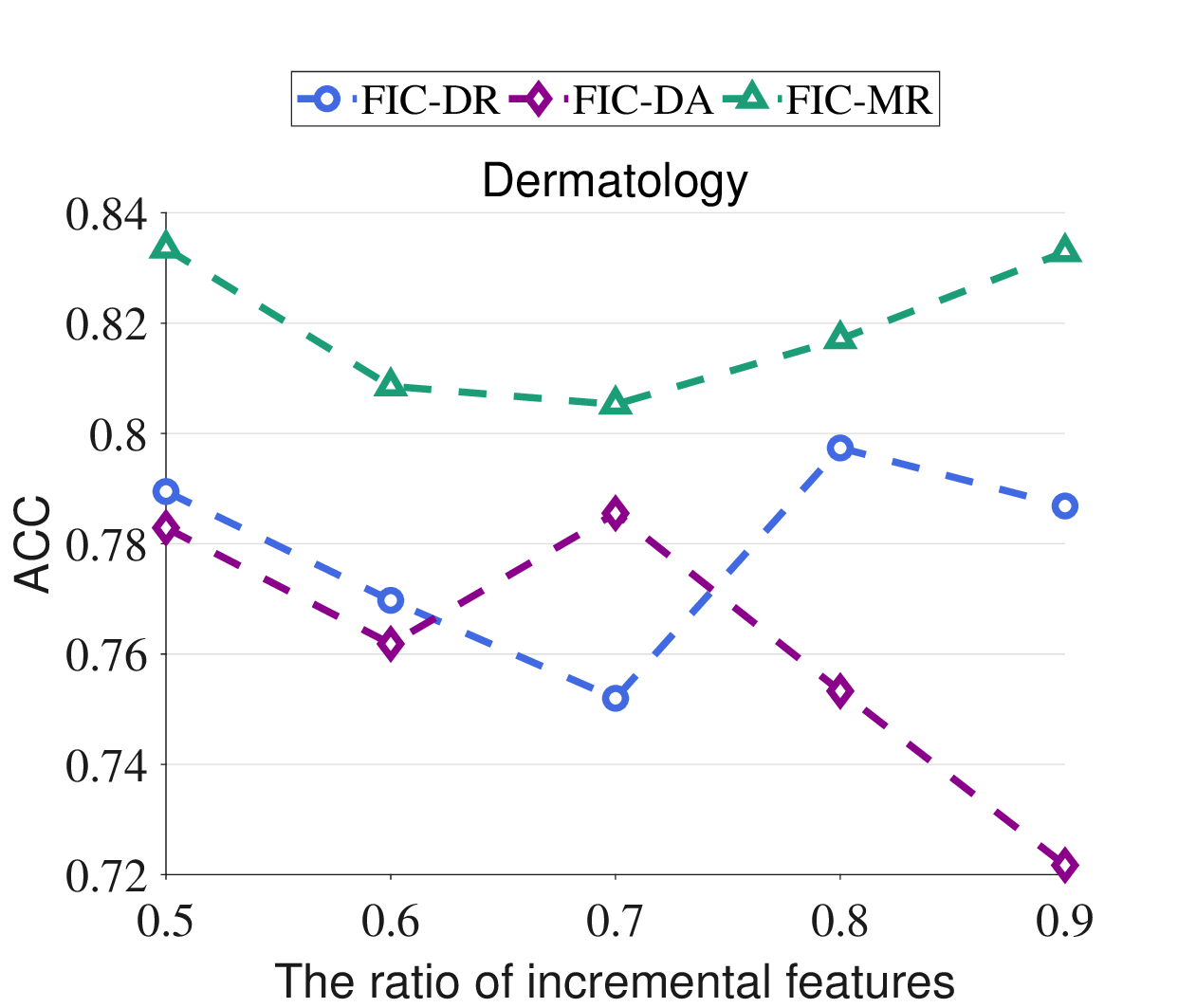}}
		\subfigure[Caltech101-7]{
			\includegraphics[width=5.5cm,height=4.3cm]{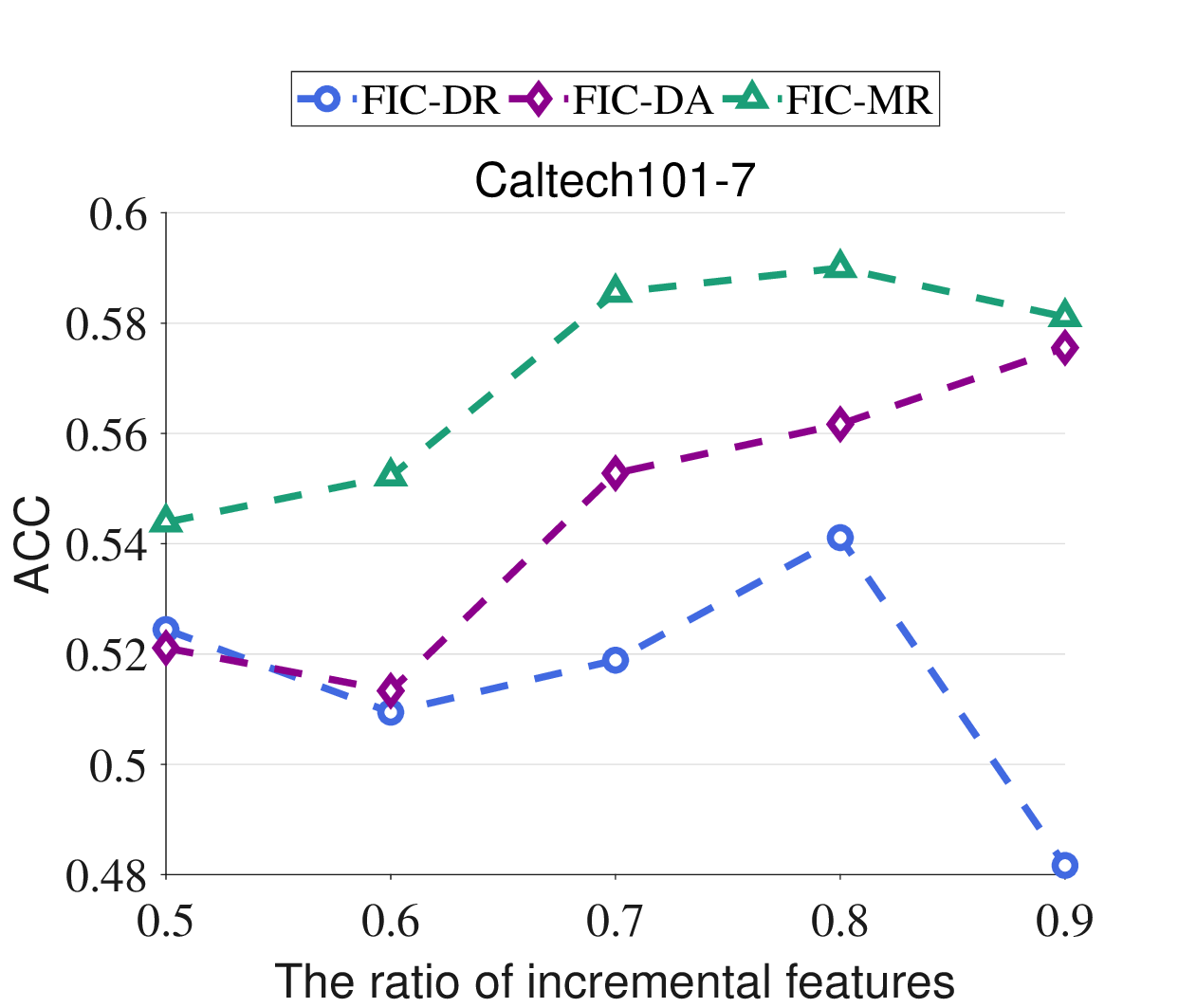}}
		\label{}
		\subfigure[Event]{
			\includegraphics[width=5.5cm,height=4.3cm]{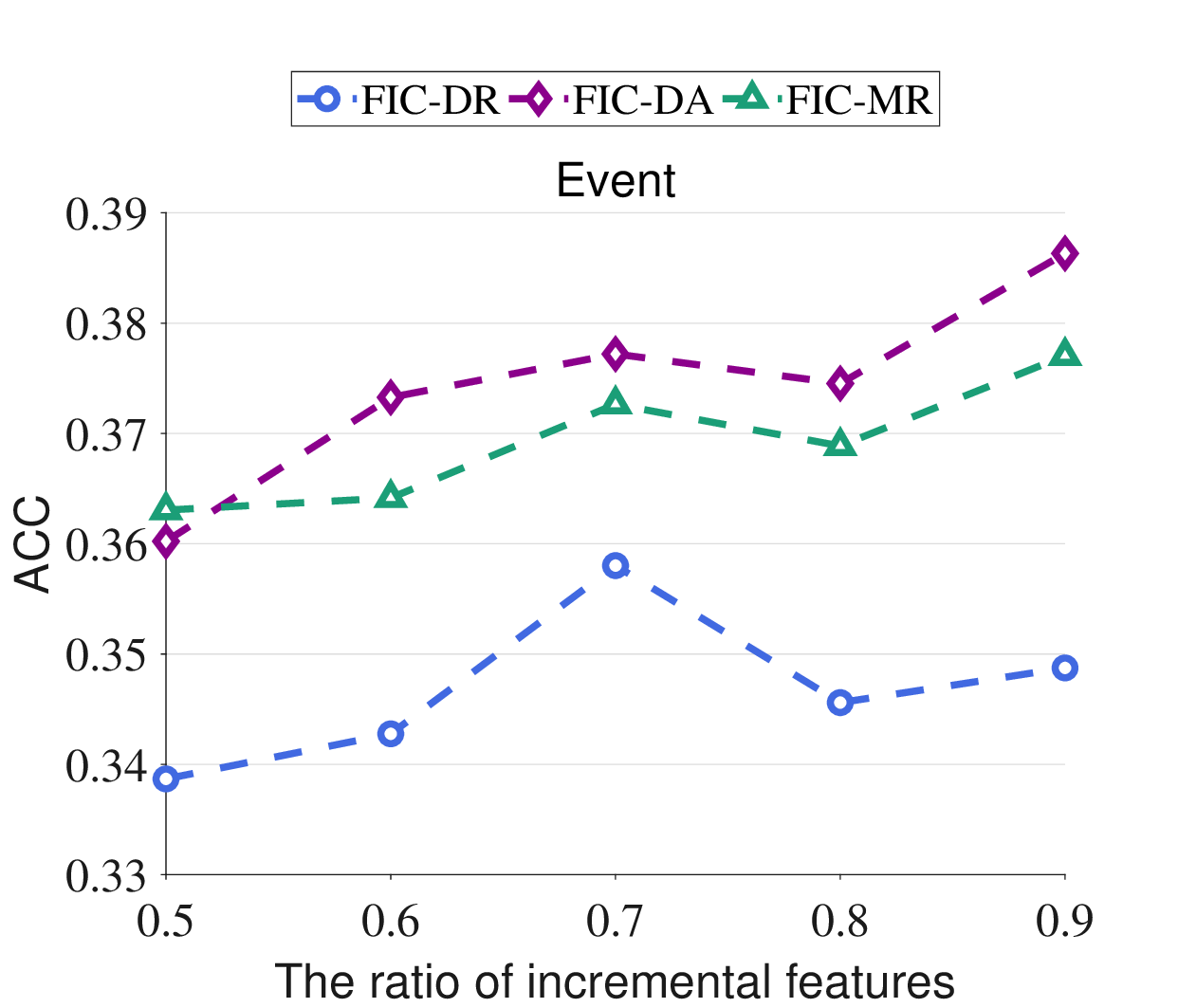}}
		\label{}
		
		\subfigure[Handwritten]{
			\includegraphics[width=5.5cm,height=4.3cm]{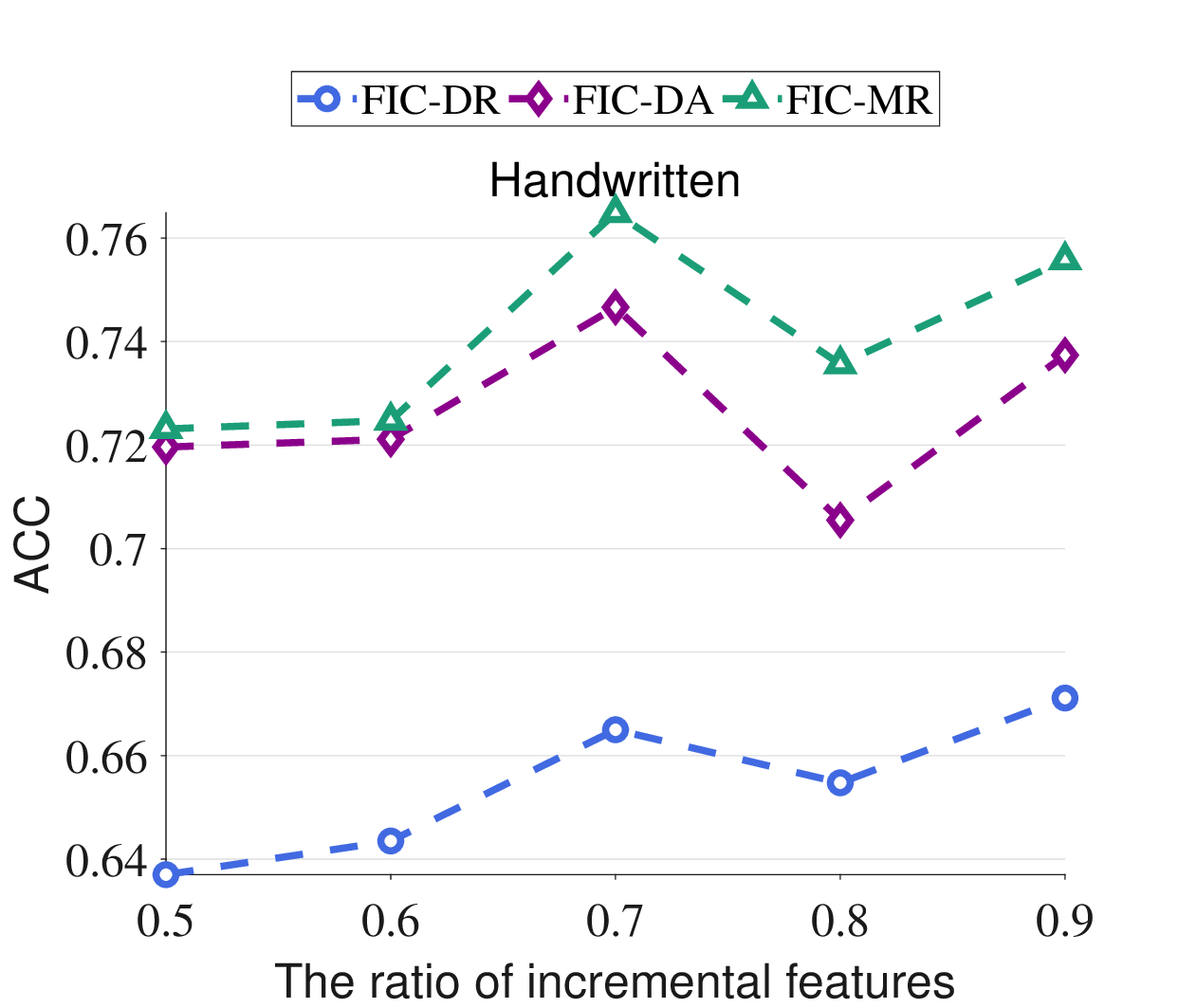}}
		\label{}
		\subfigure[Wiki]{
			\includegraphics[width=5.5cm,height=4.3cm]{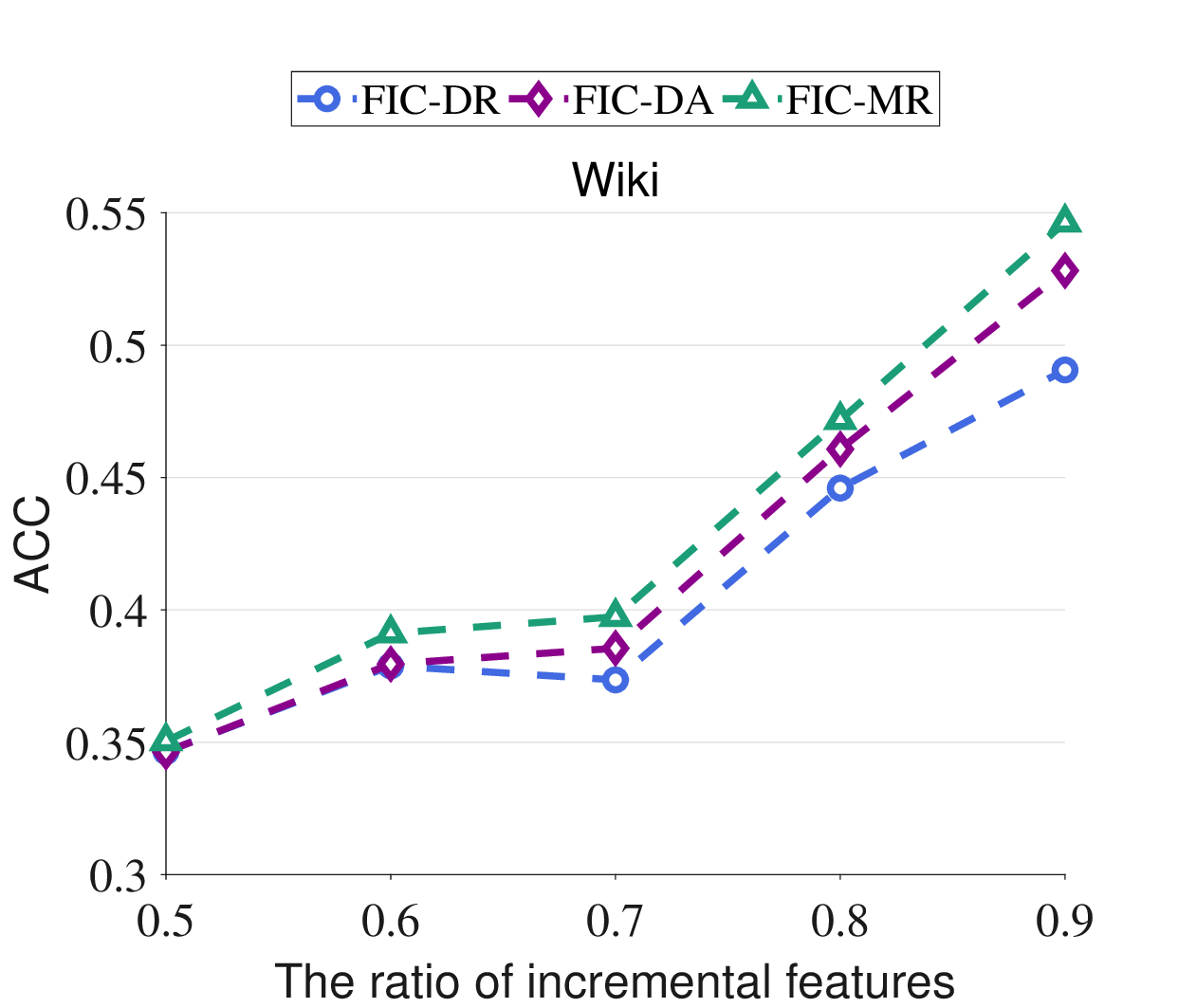}}
		\label{}
		\subfigure[Mnist]{
			\includegraphics[width=5.5cm,height=4.3cm]{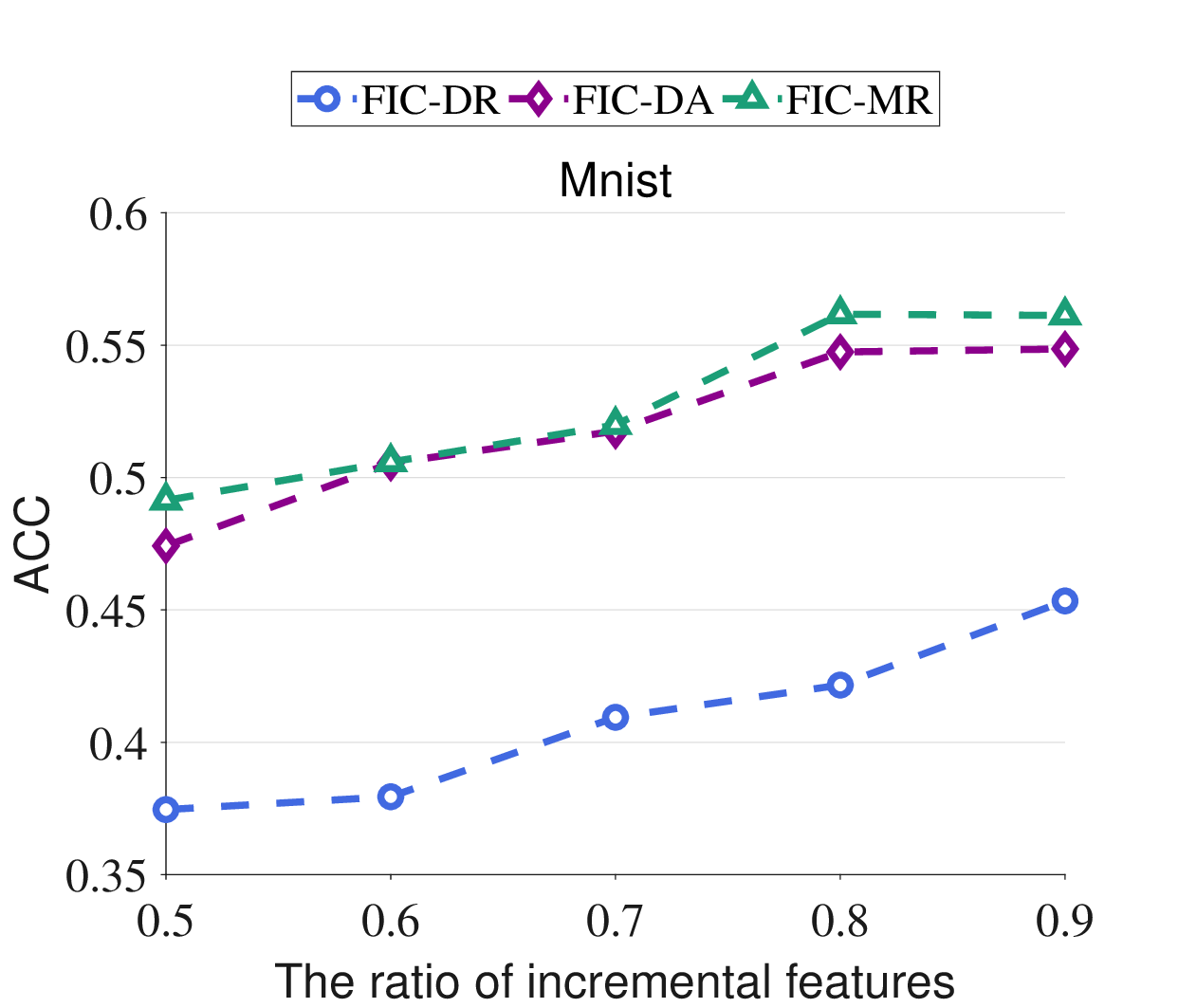}}
		\label{}
	\end{center}
	\vskip -0.2 in
	\caption{The influence of incremental features in the current stage with FIC-DR, FIC-DA, and FIC-MR.}
	\label{incremental feature change}
\end{figure*}

\begin{figure*}[tbhp]
	\begin{center}	
		\subfigure[Caltech101-7]{
			\includegraphics[width=5.8cm,height=4cm]{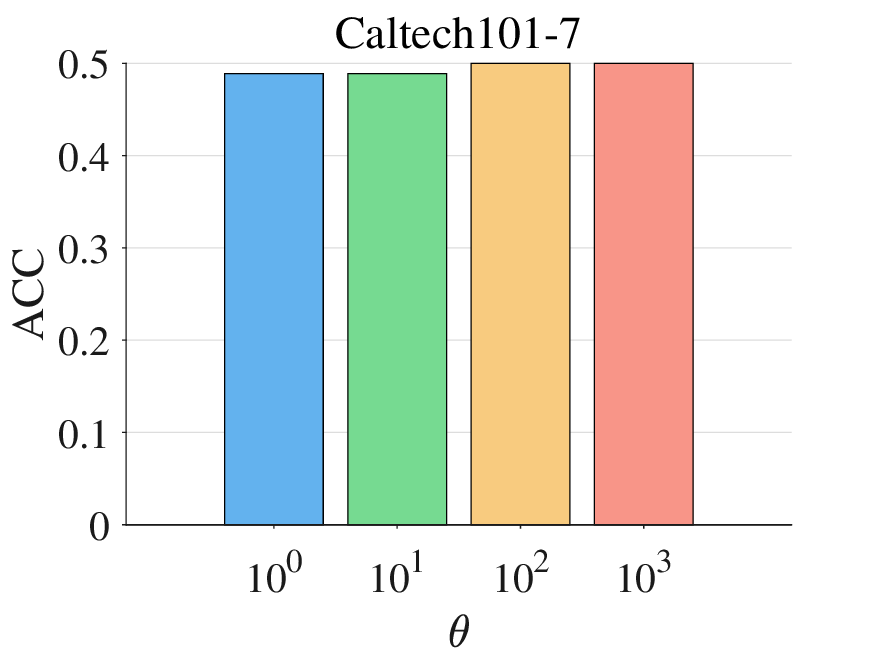}}
		\label{}
		\subfigure[Indoor]{
			\includegraphics[width=5.8cm,height=4cm]{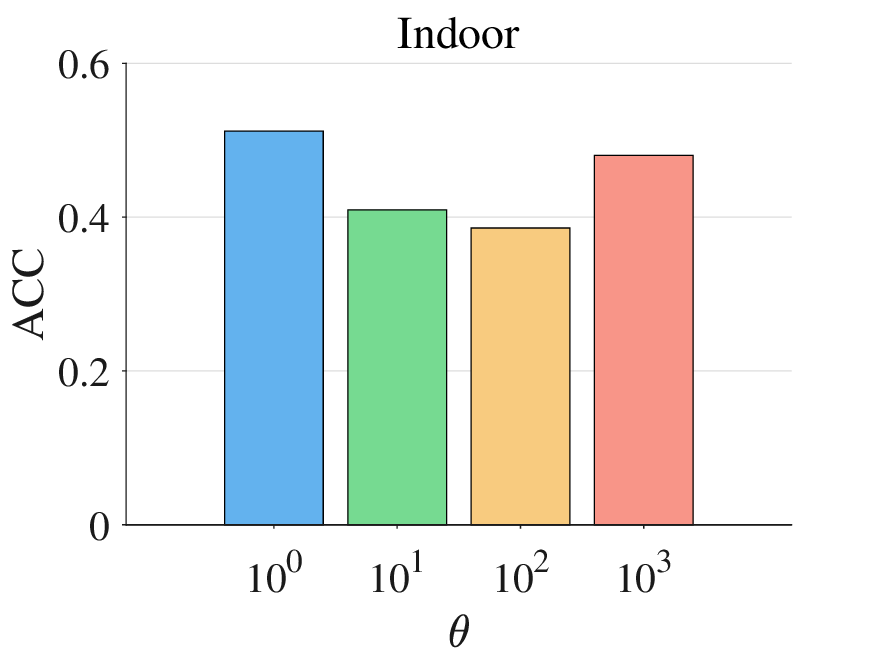}}
		\label{}
		\subfigure[Event]{
			\includegraphics[width=5.8cm,height=4cm]{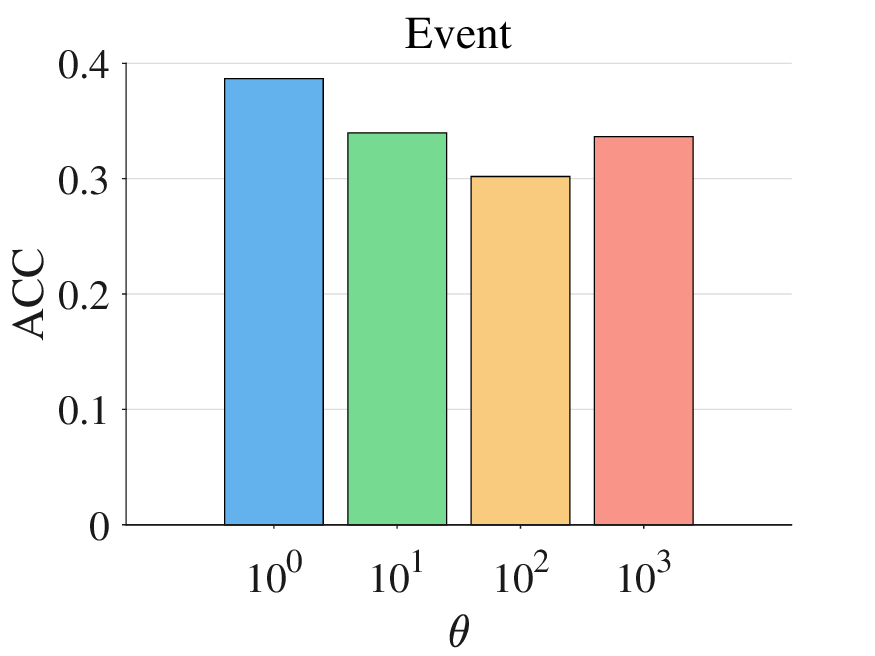}}
		\label{}
	\end{center}
	\vskip -0.15 in
	\caption{Parameter sensitivity of FIC-MR on ACC results on six datasets.}
	\label{parameterfigure}
\end{figure*}
\begin{figure*}[tbhp]
	\begin{center}	
		\subfigure[]{
			\includegraphics[width=5.8cm,height=5cm]{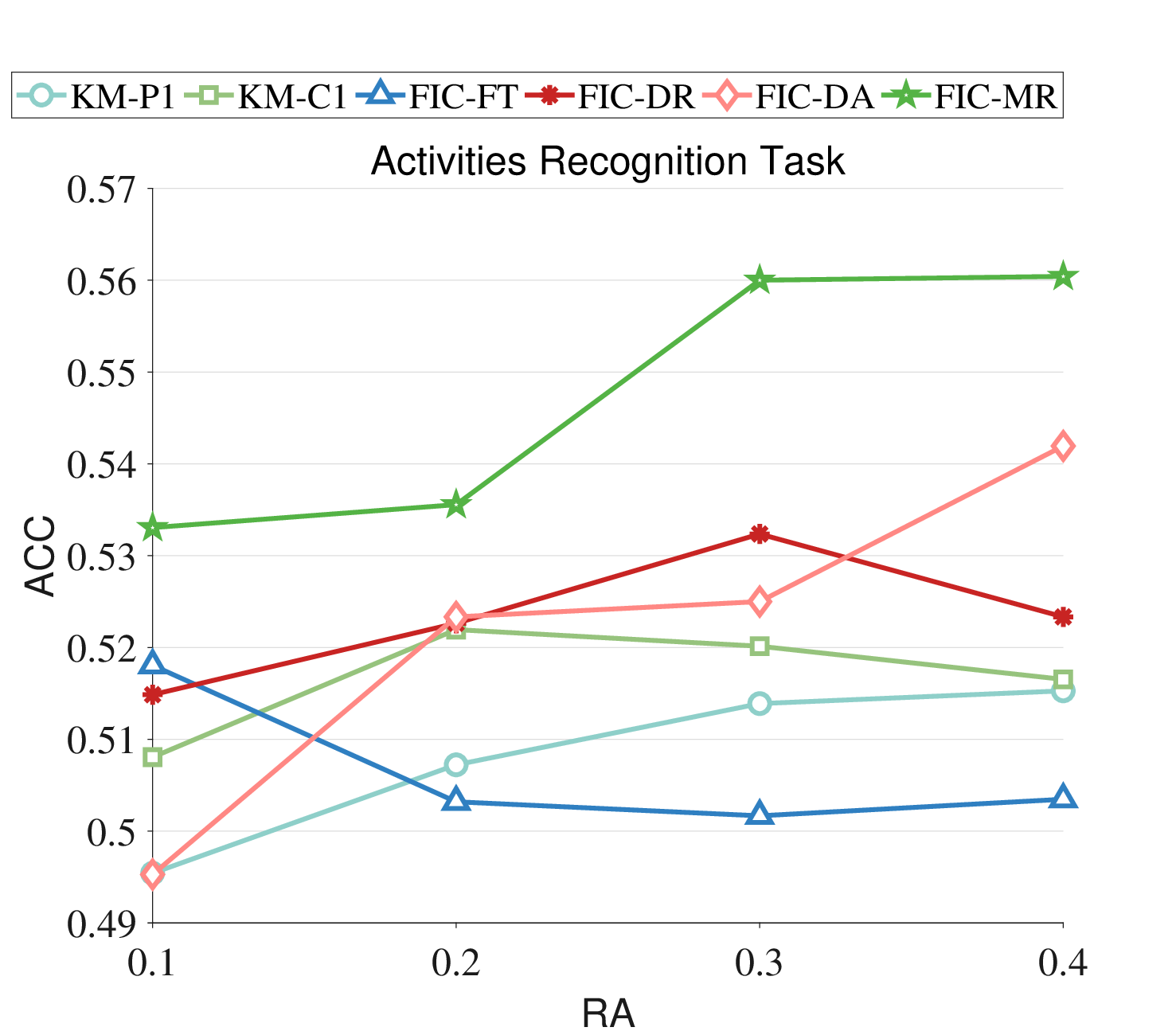}}
		\subfigure[]{
			\includegraphics[width=5.8cm,height=5cm]{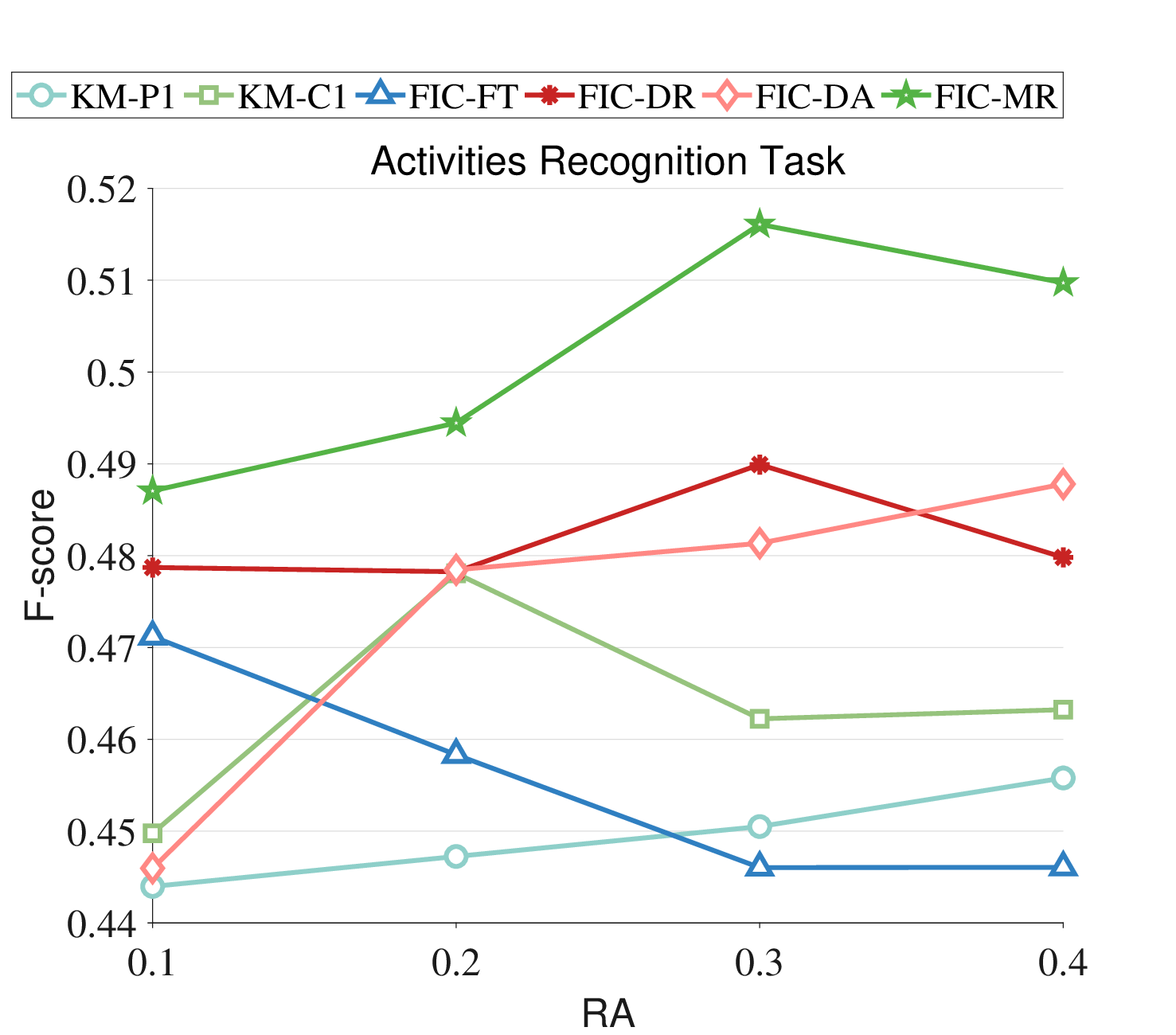}}
		\label{}
		\subfigure[]{
			\includegraphics[width=5.8cm,height=5cm]{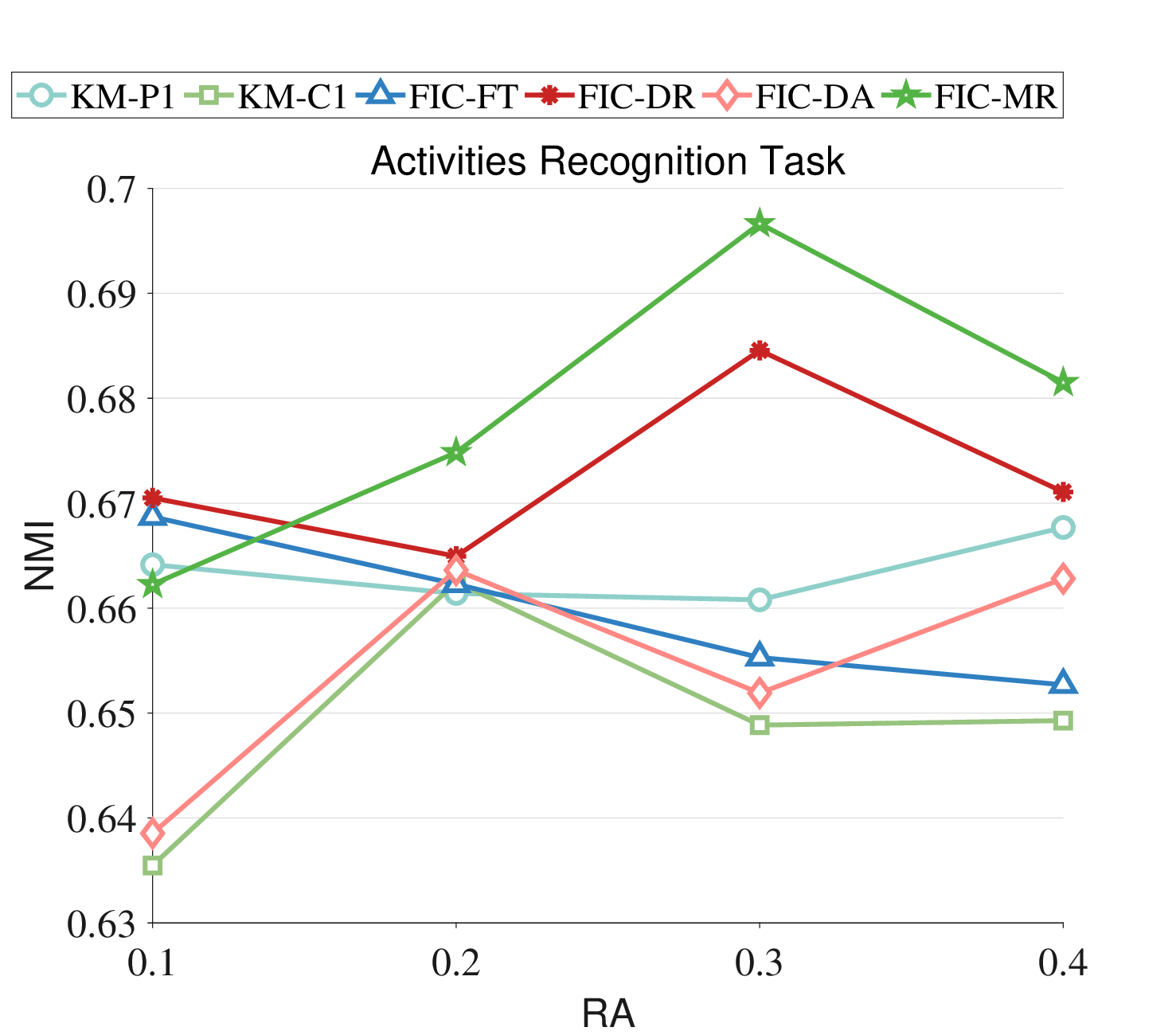}}
		\label{}
	\end{center}
	\vskip -0.2 in
	\caption{ The ACC, F-score, and NMI results comparison in the activity recognition clustering task.}
	\label{application}
\end{figure*}
\subsection{Comparison Results}
Tables \ref{one table} and \ref{two table} summarize the performance of all algorithms across nine real-world datasets, leading to the following observations:

\begin{itemize}
\item 	With access to previous-stage data, FIC-FT, KM-P1, and KM-C1 use only old features, while FIC-DA integrates both sets.  FIC-DA consistently outperforms these baselines in most cases. This suggests that leveraging a single feature type is often insufficient. 
Notably, strategic use of incremental features can improve clustering performance even without previous-stage data, despite limited current-stage availability.

\item FIC-DR completes new features from previous-stage data based on current-stage data, then merges them with the current data for model training. However, its performance is often unstable and generally inferior to FIC-MR. This instability likely stems from the lack of prior distribution information, compromising the reconstruction function's ability to minimize the discrepancy between reconstructed and observed data. Besides, FIC-DR performs well on datasets like Dermatology and Caltech101-7, where this distribution discrepancy is relatively small. These empirical findings align with the generalization theory for FIC-DR, validating that minimizing distribution divergence is crucial for enhancing generalization capability.

\item Despite neither accessing previous-stage data, FIC-MR outperforms FIC-DA by reusing prior cluster centers rather than relying only on current-stage data. This demonstrates that effectively reusing prior knowledge can enhance performance, aligning with our generalization theory for FIC-MR.

\item FIC-MR consistently outperforms all compared algorithms in most scenarios. For instance, on the Indoor dataset with RA=10\%, it achieves over a $5\%$ ACC improvement over the best baseline. This gain stems from the effective inheritance of knowledge from the previous stage.

Results show performance generally improves with RA as more current-stage data becomes available. Notably,  the advantages of FIC-MR are more pronounced at smaller RAs, highlighting the effectiveness of model reuse in feature-incremental scenarios.

\end{itemize}

\subsection{Influence of Incremental Features}
We vary the proportion of new features in the current stage from 0.5 to 0.9. For the incremental feature-based algorithms (FIC-DR, FIC-DA, and FIC-MR), we investigate how performance varies with the number of incremental features used.
Fig.\,\ref{incremental feature change} illustrates the experimental results across six datasets, due to space constraints.
 Generally, FIC-DA and FIC-MR performance improve with incremental features, despite minor fluctuations in parameter sensitivity or feature interactions in some cases.  Furthermore, FIC-DR’s performance declines on certain datasets (e.g., Caltech101-7 and Event) as new feature increase. This degradation likely stems from the difficulty of accurately reconstructing larger feature sets, causing greater distribution deviations. Such instability in FIC-DR aligns with both intuitive expectations and our generalization theory.

\subsection{Parameter Sensitivity}

Among the algorithms, only FIC-MR requires tuning a single parameter within a small range. We evaluate its performance impact by fixing the current-stage RA at 10\% and using all incremental features.

Due to space constraints, we show the experimental results across three datasets in Fig.\,\ref{parameterfigure}. Across the tested range $\{10^0, 10^1, 10^2, 10^3\}$,  a suitable parameter can always yield good performance, with no further refinement needed.

\subsection{ Application to Activity Recognition}
We evaluate our four algorithms on the activity recognition task \cite{banos2014dealing} using the RealDisp dataset \footnote{http://archive.ics.uci.edu/ml/datasets/REALDISP+Activity+Recognition+\\ Dataset}. This setup follows the motivating example in Fig.\,\ref{scenariofig}, where new features become available during the current stage. Specifically, people usually perform warm-up exercises before engaging in formal physical activities. During the warm-up phase (previous stage), 5 sensors collect 65-dimensional data. In the formal exercise phase (current stage), 4 additional sensors are introduced to enhance monitoring, expanding the feature space to 117 dimensions. 
The goal is to cluster these observations into 12 activity categories. By simulating this two-stage evolution of feature space. This setup provides a realistic testbed for evaluating the adaptability of our algorithms in dynamic, sensor-based environments.

We adopt the earlier data split in our main experiments (50\% previous, 20\% test, 30\% current) and test various current-stage data volumes by setting the RA across 10\%--40\%.
The results, illustrated in Fig.\,\ref{application}, lead to several key observations in the activity recognition clustering task:
\begin{itemize}
\item FIC-DA generally outperforms FIC-FT, suggesting that incorporating both new and old features from the current stage is more effective than combining cross-stage old features, provided the new features contribute positively.

\item FIC-DR exhibits unstable and inferior performance compared to FIC-MR. This degradation likely stems from noise and distribution mismatch during feature reconstruction, empirically validating our theoretical stability analysis of FIC-DR. FIC-MR performs well in most cases, corroborating our findings on other datasets.

\end{itemize}
Overall, these findings validate the practical utility of our framework for real-world feature incremental clustering. Beyond activity recognition, the proposed algorithms can also be applied to various domains. For example, in ecosystem sensor networks, initial sensors collect baseline environmental features, while newly deployed sensors provide new features. By aligning with these feature incremental settings, our framework offers a solution for diverse real-world clustering tasks.

\section{Conclusion}
\label{ConclusionSection}
In this paper, we have initially addressed the feature incremental clustering learning problem, which is important yet rarely studied. 
We first propose four clustering algorithms tailored for incremental feature scenarios. Next, to further understand these algorithms, we establish generalization error bounds for all four algorithms. This theoretical insight not only guides the development of new algorithms but also reveals key factors influencing their performance through a comparative analysis.
Finally, extensive experiments show the effectiveness of our proposed algorithms.

Future improvements could integrate deep clustering base models to improve performance. Furthermore, since the current framework assumes informative features, robustness against redundancy and high variance could be improved by: (i) filtering redundant inputs during pre-processing \cite{yang2024say}; (ii) incorporating latent variance control to enhance robustness against noisy or high-variance features \cite{mavroeidis2014feature}.

\section{Proofs}
\label{proofs}

\subsection{proof of Theorem \ref{k-means  generalization error bound}}
\label{k-means  generalization error bound proof}
In proving the validity of Theorem \ref{k-means  generalization error bound}, we propose several lemmas.
\begin{lemma}
	\label{Rademacher complexity0}
	Consider any sample set $D=\left\{\mathbf{x}_{1}, \ldots, \mathbf{x}_{n}\right\}$ drawn from the input space $\mathcal{X}$.
	Assuming that $\lVert \mathbf{x}\rVert \leq \gamma$ for $ \forall\mathbf{x}\in \mathcal{X}$,  there exists a positive constant $V$, such that
	\begin{equation}
		\mathfrak{R}_{n}(\mathcal{J}_{\mathbf{U}}) \leq 4\gamma^2V\sqrt {\frac{k}{n}}\log ^{2}(2n\sqrt{2\gamma n}),
	\end{equation}
\end{lemma}
\noindent We provide the proof of Lemma \ref{Rademacher complexity0} in Section \ref{Rademacher complexity0proof}.

\begin{lemma} 
	\label{Generalization error 0}
	Denote $\mathcal{G}_{U}$ as the hypothesis set and let $D$ represent a sample set comprising $n$ instances generated i.i.d. from the input space $\mathcal{X}$. 
	Let $\mathcal{J}_{\mathbf{U}}$ be  the minimum family of the functions $\mathcal{G}_{\mathbf{U}}$. If $\lVert \mathbf{x}\rVert \leq \gamma$ for $ \forall\mathbf{x}\in \mathcal{X}$,
	then for $\forall g_{\mathbf{U}}\in \mathcal{G}_{
		\mathbf{U}}$ and for any $\delta>0$, with probability at least $1-\delta$, we have
	\begin{equation}
		\label{generalization error inequality}
		\begin{aligned}
			\mathcal{L}_{\mathbb{Q}}(\mathbf{U})
			\leq &\hat{\mathcal{L}}_{D}(\mathbf{U}) +2 \mathfrak{R}(\mathcal{J}_{\mathbf{U}})+\frac{6\gamma\log (1 / \delta)}{n} \\&+ 3 \sqrt{\frac{(\lambda+8  \mathfrak{R}(\mathcal{J}_{\mathbf{U}})) \gamma \log (1 / \delta)}{n}},
		\end{aligned}
	\end{equation}
	where $\lambda=\mathbb{E}_{\mathbf{x}\sim \mathbb{Q}} [\sup _{g_{\mathbf{U}}\in\mathcal{G}_{\mathbf{U}}}\hat{\mathcal{L}}_{D}(\mathbf{U}) ]$. 
\end{lemma}
\noindent We provide the proof of Lemma \ref{Generalization error 0} in Section \ref{Generalization error 0proof}.

\begin{proof}[\textbf{Proof of Theorem \ref{k-means  generalization error bound}}]
	Next, by using the lemmas presented above, we prove Theorem \ref{k-means  generalization error bound}. According to the result in \cite{oneto2015local}, there is a probability of $1-\delta$ that the following inequality remains valid:
	\begin{equation}
		\mathfrak{R}(\mathcal{J}_{\mathbf{U}}) \leq \mathfrak{R}_{n}(\mathcal{J}_{\mathbf{U}})+\sqrt{\frac{2  \log (1/\delta)}{n}}.
	\end{equation}
	Then we can use Lemma \ref{Rademacher complexity0} to get
	\begin{equation}
		\label{R_value_inequality_final0}
		\begin{aligned}
			\mathfrak{R}(\mathcal{J}_{\mathbf{U}}) \leq & 4\gamma^2 V \sqrt {\frac{k}{n}}\log ^{2}(2n\sqrt{2\gamma n})+\sqrt{\frac{2  \log (1/\delta)}{n}}.
		\end{aligned}
	\end{equation}
	By substituting Inequality (\ref{R_value_inequality_final0}) into the corresponding term in Inequality (\ref{generalization error inequality}) from Lemma \ref{Generalization error 0}, we can thus conclude the proof for Theorem \ref{k-means  generalization error bound}.	
\end{proof}

\subsection{Proof of Lemma \ref{Rademacher complexity0}}
\label{Rademacher complexity0proof}
To prove Lemma \ref{Rademacher complexity0}, we introduce Lemma \ref{Rademacher complexity_property0} and Lemma \ref{R_vallue_max0}.
\begin{lemma}
	\label{Rademacher complexity_property0}
	{\rm (Theorem 1, provided in \cite{foster2019vector})}. 
	Let  $\psi: \mathbb{R}^{k} \rightarrow \mathbb{R}$ be a function that satisfies the property of being L-Lipschitz w.r.t. the $L_{\infty}$ norm. This property is defined as  $\lVert \psi(\mathbf{q})-\psi(\mathbf{q}^{\prime})\rVert _{\infty} \leq L \cdot\lVert \mathbf{q}-\mathbf{q}^{\prime}\rVert _{\infty}$, for all $\mathbf{q}, \mathbf{q}^{\prime} \in \mathbb{R}^{k}$. 
	Let $\mathcal{G}\subseteq\left\{g: \mathcal{X} \rightarrow \mathbb{R}^{k}\right\}$.  
	If $\max \left\{|\psi(g(\mathbf{x}))|,\lVert g(\mathbf{x})\rVert_{\infty}\right\} \leq b$ for any positive value $r$,  and there exists a constant $V>0$, then the following inequality holds
	$$
	\mathfrak{R}_{n}(\psi \circ \mathcal{G}) \leq V \cdot L \sqrt{k} \max _{s} \widetilde{\mathfrak{R}}_{n}(\mathcal{G}_{s}) \log ^{\frac{3}{2}+r}(\frac{b n}{\max _{s} \widetilde{\mathfrak{R}}_{n}(\mathcal{G}_{s})}),
	$$
	where $\widetilde{\mathfrak{R}}_{n}(\mathcal{G}_{s})=\sup _{D \in \mathcal{X}^{n}} \mathfrak{R}_{n}(\mathcal{G}_{s})$ and
	$\mathfrak{R}_{n}(\psi\circ \mathcal{G})=\mathbb{E}_{\boldsymbol{\sigma}}[\sup _{g \in \mathcal{G}}\lvert \sum_{i=1}^{n} \sigma_{i} \psi(g(\mathbf{x}_{i}))\rvert]$. Here $\mathcal{G}_{s}$ denotes the family of functions  corresponding to the $s$-th output coordinate of $\mathcal{G}$.
\end{lemma}

\begin{lemma}
	\label{R_vallue_max0}
	Assuming that $\lVert \mathbf{x}\rVert \leq \gamma$ for $ \forall\mathbf{x}\in \mathcal{X}$, we consider any sample set $D=\left\{\mathbf{x}_{1}, \ldots, \mathbf{x}_{n}\right\}$ drawn from the input space $\mathcal{X}$, and $\mathbf{U}\in \mathcal{H}^{k}$. 
	Let $\widetilde{\mathfrak{R}}_{n}(\mathcal{G}_{\mathbf{U}^{(s)}})=\sup _{D \in \mathcal{X}} \mathfrak{R}_{n}(\mathcal{G}_{\mathbf{U}^{(s)}})$, where $\mathcal{G}_{{\mathbf{U}^{(s)}}}$ refers to a function family  for the output coordinate $s$ of $\mathcal{G}_{\mathbf{U}}$.
	Then, we can obtain
	\begin{equation}
		\label{max_R_unper bound0}
		\begin{aligned}
			&\max_{s}{\widetilde{\mathfrak{R}}}_{n}(\mathcal{G}_{\mathbf{U}^{(s)}})  \leq
			\frac{4\gamma^2}{\sqrt{n}}. \\
		\end{aligned}
	\end{equation}
\end{lemma}
\noindent We provide the proof of Lemma \ref{R_vallue_max0} in Section \ref{R_vallue_max0proof}.

\begin{proof} [\textbf{Proof of Lemma \ref{Rademacher complexity0}}]
	Subsequently, leveraging Lemma \ref{Rademacher complexity_property0} and Lemma \ref{R_vallue_max0}, we proceed to provide a comprehensive proof of Lemma \ref{Rademacher complexity0}. First, it is easy to verify that the function $\psi(\mathbf{q})$ is 1-Lipschitz continuous w.r.t. the $L_{\infty}$-norm, i.e.,
	$\forall \mathbf{q}, \mathbf{q}^{\prime} \in \mathbb{R}^{k},\lvert \psi(\mathbf{q})-\psi(\mathbf{q}^{\prime})\rvert \leq\lVert \mathbf{q}-\mathbf{q}^{\prime}\rVert _{\infty}$. Specifically,  assuming $\psi(\mathbf{q}) \geq \psi(\mathbf{q}^{\prime})$, let $\psi(\mathbf{q}^{\prime})=q_{j}^{\prime}$ within the definition of the function $\psi$, where $j= \underset{i=1, \ldots, k}{\arg \min } \ q_{i}^{\prime}.$ Hence, there is
	\begin{equation}
		\label{psi_value_max}
		\begin{aligned}
			&\lvert \psi(\mathbf{q})-\psi(\mathbf{q}^{\prime})\rvert=\psi(\mathbf{q})-q_{j}^{\prime}  \leq q_{j}-q_{j}^{\prime}  \leq\lVert \mathbf{q}-\mathbf{q}^{\prime}\rVert _{\infty}.
		\end{aligned}
	\end{equation}
	Our next step involves verifying the another condition:
	$\max \left\{\lvert \psi(g_{\mathbf{U}}(\mathbf{x}))\rvert,\lVert g_{\mathbf{U}}(\mathbf{x})\rVert _{\infty}\right\}$  bounded by a constant. 
	We know 
	$g_{\mathbf{U}}(\mathbf{x})=(g_{\mathbf{u}_{1}}(\mathbf{x}),\ldots,g_{\mathbf{u}_{k}}(\mathbf{x}))$ with $g_{\mathbf{u}_{s}}(\mathbf{x})=\lVert  \mathbf{x}-\mathbf{u}_{s} \rVert ^2, \lVert \mathbf{x}\rVert  \leq \gamma$. 
	Obviously, in $k$-means algorithm, we have $\lVert \mathbf{u}_{s}\rVert  \leq \gamma,  \text{for}\  s=1,\ldots,k$  and 
	\begin{equation}
		\label{gx_sup0}
		\begin{aligned}
			&g_{\mathbf{u}_{s}}(\mathbf{x})=\lVert  \mathbf{x}-\mathbf{u}_{s} \rVert ^2\leq 2( \lVert \mathbf{x}\rVert +\lVert \mathbf{u}_{s}\rVert ) \leq 4\gamma.
		\end{aligned}
	\end{equation}
	Hence, it is easy to get
	\begin{equation}
		\label{psi_valuebound}
		\lvert \psi(g_{\mathbf{U}}(\mathbf{x}))\rvert=
		\lvert \min(g_{\mathbf{u}_{1}},\ldots,g_{\mathbf{u}_{k}})\rvert \leq 4\gamma
	\end{equation}
	and 
	\begin{equation*}
		\lVert g_{\mathbf{U}}(\mathbf{x})\rVert _{\infty}=\max _{s=1,\ldots,k}\lvert g_{\mathbf{u}_{s}}(\mathbf{x})\rvert \leq 4\gamma.
	\end{equation*}
	According to the above analysis, one can see that the function $\psi(\mathbf{q})$ is 1-Lipschitz continuous w.r.t. the $L_{\infty}$-norm and $\max \left\{\lvert \psi(g_{\mathbf{U}}(\mathbf{x}))\rvert,\lVert g_{\mathbf{U}}(\mathbf{x})\rVert _{\infty}\right\} \leq 4\gamma$. Therefore, leveraging  Lemma \ref{Rademacher complexity_property0} with $L=1$,  $b= 4\gamma$, and $r=1 / 2$ within our objective function allows us to achieve
	\begin{equation}
		\label{R_value_inequality0}
		\mathfrak{R}_{n}(\mathcal{F}_{\mathbf{U}}) \leq V \sqrt{k} \max _{s} \widetilde{\mathfrak{R}}_{n}(\mathcal{G}_{\mathbf{U}^{(s)}}) \log ^{2}(\frac{4\gamma n}{\max _{s} \widetilde{\mathfrak{R}}_{n}(\mathcal{G}_{\mathbf{U}^{(s)}})}),
	\end{equation} 
	where $\mathcal{G}_{\mathbf{U}^{(s)}}$ is a family of the output coordinate $s$ of $\mathcal{G}_{\mathbf{U}}$,
	$\widetilde{\mathfrak{R}}_{n}(\mathcal{G}_{\mathbf{U}^{(s)}})=\sup _{D \in \mathcal{X}^{n}} \mathfrak{R}_{n}(\mathcal{G}_{\mathbf{U}^{(s)}})$.
	On the other hand, we can obtain
	\begin{equation}
		\label{R_value0}
		\begin{aligned}
			&\widetilde{\mathfrak{R}}_{n}(\mathcal{G}_{\mathbf{U}^{(s)}}) 
			=\sup _{D \in \mathcal{X}^{n}} \mathfrak{R}_{n}(\mathcal{G}_{\mathbf{U}^{(s)}})\\
			&=\sup _{D \in \mathcal{X}^{n}} \mathbb{E}_{\boldsymbol{\sigma}}[\sup _{g_{\mathbf{u}} \in \mathcal{G}_{\mathbf{U}^{(s)}}}\lvert\frac{1}{n}\sum_{j=1}^{n} \sigma_{j} g_{\mathbf{u}}(\mathbf{x}_{j})\rvert] \\
			& \geq \sup _{\mathbf{x} \in \mathcal{X}} \mathbb{E}_{\boldsymbol{\sigma}}[\sup _{g_{\mathbf{u}} \in \mathcal{G}_{\mathbf{U}^{(s)}}}\lvert \frac{1}{n}\sum_{j=1}^{n} \sigma_{j} g_{\mathbf{u}}(\mathbf{x})\rvert] \\
			& \geq \sup _{\mathbf{x} \in \mathcal{X}, g_{\mathbf{u}} \in \mathcal{G}_{\mathbf{U}^{(s)}}} \mathbb{E}_{\boldsymbol{\sigma}}[\lvert \frac{1}{n}\sum_{j=1}^{n} \sigma_{j} g_{\mathbf{u}}(\mathbf{x})\rvert]\text { (via Jensen's inequality) }. 
		\end{aligned}
	\end{equation}
	As outlined in Lemma 24 (a) of \cite{lei2019data} with $p=2$, assuming $\sigma_{1},\ldots,\sigma_{n}$ are independent Rademacher variables, and $\beta_{1}, \ldots, \beta_{n} $ are drawn from the a Hilbert space $\mathcal{H}$ with the norm denoted by $\lVert \cdot\rVert $, it is established that 
	\begin{equation}
		\label{In_10}
		\mathbb{E}_{\boldsymbol{\sigma}}[\lVert \sum_{i=1}^{n} \sigma_{i} \beta_{i} \rVert ^{2}]\leq   \sum_{i=1}^{n}\lVert  \beta_{i} \rVert ^2
	\end{equation}
	and
	\begin{equation}
		\label{In_20}
		\mathbb{E}_{\boldsymbol{\sigma}}[\lvert \sum_{i=1}^{n} \sigma_{i} \beta_{i}\rvert]\geq  \frac{\sqrt{2}}{2} \sqrt{\sum_{i=1}^{n}\lVert  \beta_{i} \rVert ^2}.
	\end{equation}
	By (\ref{In_20}), there is
	\begin{equation*}
		\begin{aligned}
			\widetilde{\mathfrak{R}}_{n}(\mathcal{G}_{\mathbf{U}^{(s)}}) &\geq \sup _{\mathbf{x} \in \mathcal{X}, g_{\mathbf{u}} \in \mathcal{G}_{\mathbf{U}^{(s)}}} \mathbb{E}_{\boldsymbol{\sigma}}[\lvert \frac{1}{n}\sum_{j=1}^{n} \sigma_{j} g_{\mathbf{u}}(\mathbf{x})\rvert] \\
			&\geq  
			\frac{\sqrt{2}}{2\sqrt{n}}\sup _{\mathbf{x} \in \mathcal{X}, g_{\mathbf{u}} \in \mathcal{G}_{\mathbf{U}^{(s)}}}( \lvert g_{\mathbf{u}}(\mathbf{x})\rvert)^{1/2}.
		\end{aligned}
	\end{equation*}
	We give a definition of 
	$u_{s}=\sup _{\mathbf{x} \in \mathcal{X}, g_{\mathbf{u}} \in \mathcal{G}_{\mathbf{U}^{(s)}}}\lvert g_{\mathbf{u}}(\mathbf{x})\rvert;$ $u=\max \left\{u_{s}, s=1, \ldots, k\right\}.$
	Obviously, one has
	$
	\widetilde{\mathfrak{R}}_{n}(\mathcal{G}_{\mathbf{U}^{(s)}}) \geq\frac{\sqrt{2u_{s}}}{2\sqrt{n}}.
	$
	Then, there is  $	\max _{s} \widetilde{\mathfrak{R}}_{n}(\mathcal{G}_{\mathbf{U}^{(s)}})\geq \frac{\sqrt{2u}}{2\sqrt{n}}$.
	Using Eq.\,(\ref{gx_sup0}), we know $\max _{s}\lvert g_{\mathbf{u}_s}(\mathbf{x})\rvert \leq 4\gamma$. Thus, it is easy to obtain $u\leq 4\gamma$. 
	So, we derive
	\begin{equation*}
		\max _{s} \widetilde{\mathfrak{R}}_{n}(\mathcal{G}_{\mathbf{U}^{(s)}})\geq\frac{\sqrt{2\gamma}}{\sqrt{n}}.
	\end{equation*}
	Then, the following inequality holds
	\begin{equation}
		\label{max_R_lowerbound0}
		\frac{4\gamma n}{\max _{s} \widetilde{\mathfrak{R}}_{n}(\mathcal{G}_{\mathbf{U}^{(s)}})} \leq 
		2n\sqrt{2\gamma n}.
	\end{equation}
	We substitute  (\ref{max_R_unper bound0})  and (\ref{max_R_lowerbound0}) into (\ref{R_value_inequality0}) to conclude the proof of Lemma \ref{Rademacher complexity0}.
\end{proof}

\subsection{Proof of Lemma \ref{Generalization error 0} }
\label{Generalization error 0proof}
To achieve a fast generalization rate, we rely on a Bernstein-type concentration inequality provided in the paper \cite{bousquet2002bennett}. The conclusion corresponding to this is detailed in Lemma \ref{Bernstein-type concentration}.
\begin{lemma} 
	\label{Bernstein-type concentration}({\rm Theorem 2.11, presented in \cite{bousquet2002bennett}}) Suppose the samples $\mathbf{x}_1, \ldots, \mathbf{x}_n$  follow i.i.d. according to $\rho$.
	Denote $\mathcal{F}$ as a set of functions mapping from $\mathcal{X}$ to $\mathbb{R}$. Suppose that all functions $f\in \mathcal{F}$ are $\rho$-measurable, square-integrable and meet $\mathbb{E}[f]=0$. 
	If $\sup _{f\in \mathcal{F}} f\leq 1$, then let
	$
	G=\sup _{f\in \mathcal{F}} \sum_{i=1}^n f(\mathbf{x}_i).
	$
	If $\sup _{f \in \mathcal{F}}\|f\|_{\infty} \leq 1$, $G$ can be defined as either as described above or as
	$
	G=\sup _{f\in \mathcal{F}}\lvert \sum_{i=1}^n f(\mathbf{x}_i)\rvert.
	$
	Assuming $\eta$ is a positive real number satisfying $\eta^2 \geq \sup _{f\in \mathcal{F}} \operatorname{Var}[f(\mathbf{x}_1)]$ almost surely, for all $c\geq 0$, we obtain
	\begin{equation*}
		\operatorname{Pr}[G \geq \mathbb{E}[G]+c] \leq \exp \left\{-t v(\frac{c}{t})\right\},
	\end{equation*}
	where $v(x)=(1+x) \log (1+x)-x$ and $t=2 \mathbb{E}[G] + n \eta^2$, also
	$$
	\operatorname{Pr}[G\geq \mathbb{E}[G]+\sqrt{2 ct}+\frac{c}{3}] \leq e^{-c} .
	$$
\end{lemma}	

\begin{proof}[\textbf{Proof of Lemma \ref{Generalization error 0}}]
	We now provide the proof of Lemma \ref{Generalization error 0}, relying on the application of the functional generalization of Bennett's inequality as presented in Lemma \ref{Bernstein-type concentration}. Considering any sample $D=\left\{\mathbf{x}_1, \ldots, \mathbf{x}_n\right\}$ and any hypothesis $g_{\mathbf{U}} \in  \mathcal{G}_{\mathbf{U}}$, our focus shifts to examining the uniform upper bound of 
	$\mathcal{L}_{\mathbb{Q}}(\mathbf{U})-\hat{\mathcal{L}}_{D}(\mathbf{U}) $.
	The argument involves utilizing the functional generalization of Bennett's inequality, i.e., Lemma \ref{Bernstein-type concentration}, to the function $\Psi_D$ defined as
	\begin{equation*}
		\Psi_D=\frac{n}{8\gamma} \sup _{g_{\mathbf{U}}\in \mathcal{G}_{
				\mathbf{U}}}(\mathcal{L}_{\mathbb{Q}}(\mathbf{U})-\hat{\mathcal{L}}_{D}(\mathbf{U})),
	\end{equation*}
	which meets the condition provided in Lemma \ref{Bernstein-type concentration}. 
	We start by verifying this claim. Using the symbols provided in Lemma \ref{Bernstein-type concentration}, denote $f(\mathbf{x}_i)=\frac{1}{8\gamma}(\mathbb{E}[\psi(g_{\mathbf{U}}(\mathbf{x}))]-\psi(g_{\mathbf{U}}(\mathbf{x}_i)) )$. Obviously, we have $\mathbb{E}[f]=0$. Furthermore, since the minimum function $\psi$ is bounded by $4\gamma$ by (\ref{psi_valuebound}), one can obtain
	$\sup _{f \in \mathcal{F}} f=\sup _{g_{\mathbf{U}} \in  \mathcal{G}_{\mathbf{U}}} \frac{1}{8\gamma}(\mathbb{E}[\psi(g_{\mathbf{U}}(\mathbf{x}))]-\psi(g_{\mathbf{U}}(\mathbf{x}_i)) )\leq 1$. Hence, the $\Psi_D$ essentially corresponds to $G$ in Lemma \ref{Bernstein-type concentration}, then we can verify
	\begin{equation*}
		\begin{aligned} 
			G & =\sup _{g\in \mathcal{G}} \sum_{i=1}^n g(X_i)\\
			&=\sup _{g_{\mathbf{U}} \in \mathcal{G}_{
					\mathbf{U}}} \sum_{i=1}^n \frac{1}{8\gamma}(\mathbb{E}[\psi(g_{\mathbf{U}}(\mathbf{x}))]-\psi(g_{\mathbf{U}}(\mathbf{x}_i)) )\\
			& =\frac{n}{8\gamma} \sup _{g_{\mathbf{U}} \in \mathcal{G}_{\mathbf{U}}}	(\mathcal{L}_{\mathbb{Q}}(\mathbf{U})-\hat{\mathcal{L}}_{D}(\mathbf{U}))\\
			&=\Psi_D.
		\end{aligned}
	\end{equation*}
	So, based on Lemma \ref{Bernstein-type concentration}, for  $ \forall  c \geq 0$, we obtain
	\begin{equation}
		\label{Pr_value}
		\operatorname{Pr}[\Psi_D \geq \mathbb{E}_D[\Psi_D]+c] \leq \exp \left\{-t v(\frac{c}{t})\right\},
	\end{equation}
	where $v(x)=(1+x) \log (1+x)-x$, 
	$t=2 \mathbb{E}_D[\Psi_D]+n \eta^2$, 
	and
	$\eta^2=\sup _{g_{\mathbf{U}}\in\mathcal{G}_{\mathbf{U}}} \operatorname{Var}[\frac{1}{8\gamma} \mathbb{E}_{\mathbf{x}^{\prime}\sim \mathbb{Q}}[\psi(g_{\mathbf{U}}(\mathbf{x}^{\prime}))]\\
	-\psi(g_{\mathbf{U}}(\mathbf{x})) ].$
	Reversing Inequality  (\ref{Pr_value}) allows us to conclude that for 
	$ \forall g_{\mathbf{U}}\in\mathcal{G}_{\mathbf{U}}$ and for any $\delta>0$, with probability at least $1-\delta$, we have
	$$
	\Psi_D\leq \mathbb{E}_D[\Psi_D]+\frac{3 \log (1 / \delta)}{4}+\frac{3}{2} \sqrt{t\log (1 / \delta)}.
	$$
	Then, we need to bound both $\eta^2$ and $\mathbb{E}_D[\Psi_D]$. Initially, we bound $\eta^2$ as follows 
	\begin{equation*}
		\begin{aligned}
			\eta^2 & =\sup _{g_{\mathbf{U}}\in\mathcal{G}_{\mathbf{U}}} \operatorname{Var}[\frac{1}{8\gamma} \mathbb{E}_{\mathbf{x}^{\prime}\sim \mathbb{Q}}[\psi(g_{\mathbf{U}}(\mathbf{x}^{\prime}))]-\psi(g_{\mathbf{U}}(\mathbf{x})) ] \\
			& =\sup _{g_{\mathbf{U}}\in\mathcal{G}_{\mathbf{U}}}\frac{1}{64\gamma^2} 
			\mathbb{E}_{\mathbf{x}\sim \mathbb{Q}}[\psi(g_{\mathbf{U}}(\mathbf{x})-\mathbb{E}_{\mathbf{x}^{\prime}\sim \mathbb{Q}}[\psi(g_{\mathbf{U}}(\mathbf{x}))])^2] \\
			& \leq \sup _{g_{\mathbf{U}}\in\mathcal{G}_{\mathbf{U}}} \frac{1}{64\gamma^2} \mathbb{E}_{\mathbf{x}\sim \mathbb{Q}}[\psi(g_{\mathbf{U}}(\mathbf{x}))^2] \\
			& \leq \sup _{g_{\mathbf{U}}\in\mathcal{G}_{\mathbf{U}}} \frac{1}{16\gamma} \mathbb{E}_{\mathbf{x}\sim \mathbb{Q}}[\psi(g_{\mathbf{U}}(\mathbf{x}))]\\
			& \leq \sup _{g_{\mathbf{U}}\in\mathcal{G}_{\mathbf{U}}} \frac{1}{16\gamma} \mathbb{E}_{\mathbf{x}\sim \mathbb{Q}}[\psi(g_{\mathbf{U}}(\mathbf{x}))]\\
			& =\frac{1}{16\gamma}\sup _{g_{\mathbf{U}}\in\mathcal{G}_{\mathbf{U}}} \mathbb{E}_{\mathbf{x}\sim \mathbb{Q}} [\hat{\mathcal{L}}_{D}(\mathbf{U}) ]\\
			&\leq \frac{1}{16\gamma} \mathbb{E}_{\mathbf{x}\sim \mathbb{Q}} [\sup _{g_{\mathbf{U}}\in\mathcal{G}_{\mathbf{U}}}\hat{\mathcal{L}}_{D}(\mathbf{U}) ]\\
			&= \frac{\lambda}{16\gamma},
		\end{aligned}
	\end{equation*}
	where $\lambda=\mathbb{E}_{\mathbf{x}\sim \mathbb{Q}} [\sup _{g_{\mathbf{U}}\in\mathcal{G}_{\mathbf{U}}}\hat{\mathcal{L}}_{D}(\mathbf{U}) ]$.
	Then, leveraging the standard symmetrization technique, we determine the bound of $\mathbb{E}_D[\Psi_D]$ by
	\begin{equation*}
		\label{upper_bound_expected_Psi}
		\begin{aligned}
			&\mathbb{E}_D[\Psi_D]  =\mathbb{E}_D[\frac{n}{8\gamma} \sup _{g_{\mathbf{U}}\in\mathcal{G}_{\mathbf{U}}}(\mathcal{L}_{\mathbb{Q}}(\mathbf{U})-\hat{\mathcal{L}}_{D}(\mathbf{U}))] \\
			&=\frac{n}{8\gamma} \mathbb{E}_D[\sup _{g_{\mathbf{U}}\in\mathcal{G}_{\mathbf{U}}}(\mathbb{E}_{D^{\prime}}[\frac{1}{n} \sum_{i=1}^n \psi(g_{\mathbf{U}}(\mathbf{x}^{\prime}_i)) 
			-\frac{1}{n} \sum_{i=1}^n \psi(g_{\mathbf{U}}(\mathbf{x}_i))])] \\
			& \overset{\circled{1}}{\leq } \frac{n}{8\gamma} \mathbb{E}_{D, D^{\prime}}[\sup _{g_{\mathbf{U}}\in\mathcal{G}_{\mathbf{U}}}(\frac{1}{n} \sum_{i=1}^n \psi(g_{\mathbf{U}}(\mathbf{x}^{\prime}_i))-\frac{1}{n} \sum_{i=1}^n\psi(g_{\mathbf{U}}(\mathbf{x}_i)))] \\
			& =\frac{n}{8\gamma} \mathbb{E}_{D, D^{\prime}}[\sup _{g_{\mathbf{U}}\in\mathcal{G}_{\mathbf{U}}} \frac{1}{n} \sum_{i=1}^n(\psi(g_{\mathbf{U}}(\mathbf{x}^{\prime}_i))-\psi(g_{\mathbf{U}}(\mathbf{x}_i)))] \\
			&\overset{\circled{2}}{= }\frac{n}{8\gamma} \mathbb{E}_{\sigma, D, D^{\prime}}[\sup _{g_{\mathbf{U}}\in\mathcal{G}_{\mathbf{U}}} \frac{1}{n} \sum_{i=1}^n \sigma_i(\psi(g_{\mathbf{U}}(\mathbf{x}^{\prime}_i))-\psi(g_{\mathbf{U}}(\mathbf{x}_i)))] \\
			& \overset{\circled{3}}{\leq } \frac{n}{4\gamma}  \mathbb{E}_{\sigma, D}[\sup _{g_{\mathbf{U}}\in\mathcal{G}_{\mathbf{U}}} \frac{1}{n} \sum_{i=1}^n \sigma_i\psi(g_{\mathbf{U}}(\mathbf{x}_i))]\\
			& \leq \frac{n}{4\gamma}  \mathbb{E}_{\sigma, D}[\sup _{g_{\mathbf{U}}\in\mathcal{G}_{\mathbf{U}}} \lvert \frac{1}{n} \sum_{i=1}^n \sigma_i\psi(g_{\mathbf{U}}(\mathbf{x}_i))\rvert]\\
			&=\frac{n}{4\gamma} \mathfrak{R}(\mathcal{J}_{\mathbf{U}}).
		\end{aligned}
	\end{equation*}
	In the analysis above, Inequality $\circled{1}$ is valid owing to the application of Jensen's inequality, which stems from the convexity of the supremum function.
	For the term $\circled{2}$, since Rademacher random variables $\sigma_i \mathrm{~s}$ are independent random variables uniformly distributed as $\{-1,+1\}$,  their introduction does not modify the expectation. Inequality $\circled{3}$ is valid as a result of the sub-additivity of the supremum function. 
	Therefore, we can derive the bound for $t$,
	\begin{equation*}
		t=n \eta^2+2 \mathbb{E}_D[\Psi_D] \leq \frac{n\lambda}{16\gamma}+\frac{ n}{2\gamma} \mathfrak{R}(\mathcal{J}_{\mathbf{U}}).
	\end{equation*}
	From the constraints on  $t$ and $\mathbb{E}_D[\Psi_D]$, we can obtain
	\begin{equation*}	
		\begin{aligned}
			&	\Psi_D  =\frac{n}{8\gamma} \sup _{g_{\mathbf{U}}\in \mathcal{G}_{
					\mathbf{U}}}(\mathcal{L}_{\mathbb{Q}}(\mathbf{U})-\hat{\mathcal{L}}_{D}(\mathbf{U}))\\
			&\leq \mathbb{E}_D[\Psi_D]+\frac{3\log (1 / \delta)}{4}+\frac{3\sqrt{t \log (1 / \delta)} }{2}\\
			&	\leq \frac{n}{4\gamma}   \mathfrak{R}(\mathcal{J}_{\mathbf{U}})+\frac{3\log (1 / \delta)}{4}
			+\frac{3}{2}\sqrt{(\frac{n\lambda }{16\gamma}+\frac{n}{2\gamma} \mathfrak{R}(\mathcal{J}_{\mathbf{U}})) \log (1 / \delta)}.
		\end{aligned}
	\end{equation*}
	Therefore, we deduce that for any  $g_{\mathbf{U}}\in\mathcal{G}_{\mathbf{U}}$,
	\begin{equation*}
		\begin{aligned}
			&	\mathcal{L}_{\mathbb{Q}}(\mathbf{U})-\hat{\mathcal{L}}_{D}(\mathbf{U}) 	\\
			&	\leq \sup_{g_{\mathbf{U}}\in\mathcal{G}_{\mathbf{U}}}(\mathcal{L}_{\mathbb{Q}}(\mathbf{U})-\hat{\mathcal{L}}_{D}(\mathbf{U}))\\
			&=\frac{8\gamma}{n}\Psi_D \\
			&\leq 2 \mathfrak{R}(\mathcal{\mathcal{J}_{\mathbf{U}}})+\frac{6\gamma\log (1 / \delta)}{n}+3 \sqrt{\frac{(\lambda+8  \mathfrak{R}(\mathcal{J}_{\mathbf{U}})) \gamma \log (1 / \delta)}{n}}.
		\end{aligned}
	\end{equation*}
	Thus, we finish the proof of Lemma \ref{Generalization error 0}.
\end{proof}

\subsection{Proof of Lemma \ref{R_vallue_max0}}
\label{R_vallue_max0proof}
\begin{proof}
	It is evident that for any given sample set $D$, along with $\mathbf{U}\in \mathcal{H}^{k}$ and $s\in \{1,\ldots,k\}$, the following holds
	\begin{equation}
		\label{R_ineqaility0}
		\begin{aligned}
			\mathfrak{R}_{n}(\mathcal{G}_{\mathbf{U}^{(s)}}) &=\mathbb{E}_{\mathbf{\sigma}} [\sup_{g_{\mathbf{u}\in\mathcal{G}_{\mathbf{U}^{(s)}}}} \lvert \frac{1}{n}\sum_{j=1}^{n}\mathbf{\sigma}_{j}g_{\mathbf{u}}(\mathbf{x}_{j})\rvert]
			\\
			& = \mathbb{E}_{\mathbf{\sigma}}[ \sup_{\mathbf{u}\in\mathcal{H}} \lvert \frac{1}{n}\sum_{j=1}^{n}\mathbf{\sigma}_{j}\lVert  \mathbf{x}_{j}-\mathbf{u} \rVert ^2\rvert]\\
			& =  \mathbb{E}_{\mathbf{\sigma}} [ \sup_{\mathbf{u}\in\mathcal{H}}\lvert \frac{1}{n}\sum_{j=1}^{n}\mathbf{\sigma}_{j}	(
			\lVert  \mathbf{x}_{j}\rVert ^2-2\left \langle \mathbf{x}_{j},\mathbf{u}\right \rangle +	\lVert  \mathbf{u}\rVert ^2)\rvert ] \\
			& \leq  \underbrace{\mathbb{E}_{\mathbf{\sigma}} [\lvert  \frac{1}{n}\sum_{j=1}^{n} \mathbf{\sigma}_{j}\lVert   \mathbf{x}_{j}\rVert ^2\rvert]}_{\text {Term-\circled{1}}}
			\\&+\underbrace{2\mathbb{E}_{\mathbf{\sigma}} [\sup_{\mathbf{u}\in\mathcal{H}} \lvert \frac{1}{n}\sum_{j=1}^{n}\mathbf{\sigma}_{j}\left \langle \mathbf{x}_{j},\mathbf{u}\right \rangle \rvert ]}_{\text {Term-\circled{2}}} \\
			&+\underbrace{\mathbb{E}_{\mathbf{\sigma}} [\sup_{\mathbf{u}\in\mathcal{H}}\lvert \frac{1}{n}\sum_{j=1}^{n} \mathbf{\sigma}_{j}\lVert  \mathbf{u}\rVert ^2\rvert]}_{\text {Term-\circled{3}}}.\\
		\end{aligned}
	\end{equation}
	For the $\text {Term-\circled{1}}$, note that $\lVert \mathbf{x}\rVert \leq \gamma$, we have
	\begin{equation*}
		\begin{aligned}
			\mathbb{E}_{\mathbf{\sigma}} [\lvert  \frac{1}{n}\sum_{j=1}^{n} \mathbf{\sigma}_{j}\lVert   \mathbf{x}_{j}\rVert ^2\rvert]
			& \leq \frac{ \gamma^2 }{n}[\mathbb{E}_{\mathbf{\sigma}} \lvert  \sum_{j=1}^{n} \mathbf{\sigma}_{j}\rvert].\\
		\end{aligned}
	\end{equation*}
	Here, using (\ref{In_10}), we obtain
	\begin{equation}
		\label{Term10}
		\text {Term-\circled{1}} \leq \frac{\gamma^2 }{ \sqrt{n}}.
	\end{equation}
	For the $\text {Term-\circled{2}}$,  it becomes evident that
	\begin{equation*}
		\begin{aligned}
			&\mathbb{E}_{\mathbf{\sigma}} [\sup_{\mathbf{u}\in\mathcal{H}} \lvert  \frac{1}{n}\sum_{j=1}^{n}\mathbf{\sigma}_{j}\left \langle \mathbf{x}_{j},\mathbf{u}\right \rangle \rvert] \\&=  \frac{1}{n}\mathbb{E}_{\mathbf{\sigma}} [\sup_{\mathbf{u}\in\mathcal{H}}\lvert \left \langle \sum_{j=1}^{n}\mathbf{\sigma}_{j}	 \mathbf{x}_{j},\mathbf{u}\right \rangle \rvert ]
			\\&\leq  \frac{1}{n}\mathbb{E}_{\mathbf{\sigma}} [\sup_{\mathbf{u}\in\mathcal{H}} \lVert \sum_{j=1}^{n}\mathbf{\sigma}_{j}	 \mathbf{x}_{j}\rVert  \lVert \mathbf{u}\rVert ].
		\end{aligned}
	\end{equation*}
	Here, considering $\lVert \mathbf{x}\rVert \leq \gamma$, we can easily obtain $\lVert \mathbf{u}\rVert \leq \gamma$ in the $k$ means algorithm. Referring to Inequality \,(\ref{In_10}), we have
	\begin{equation}
		\label{Term20}
		\begin{aligned}
			\text {Term-\circled{2}}&\leq  \frac{2}{n}\mathbb{E}_{\mathbf{\sigma}} [\lVert  \sum_{j=1}^{n}\mathbf{\sigma}_{j}\mathbf{x}_{j}  \rVert  \lVert \mathbf{u}\rVert ]
			\\&\leq   \frac{2}{n}\gamma  [\mathbb{E}_{\mathbf{\sigma}} [ \lVert  \sum_{j=1}^{n}\mathbf{\sigma}_{j}	 \mathbf{x}_{j} \rVert ^{2} ]]^{1/2}
			\\&\leq   \frac{2}{n}\gamma [  \sum_{j=1}^{n}
			\lVert \mathbf{x}_{j}\rVert ^{2} ]^{1/2}
			\\&\leq   \frac{2\gamma^2 }{\sqrt{n}}.
		\end{aligned}
	\end{equation}
	For  $\text {Term-\circled{3}}$, it is similar to the proof process of  $\text {Term-\circled{1}}$, then  it can be seen that 
	\begin{equation}
		\label{Term30}
		\begin{aligned}
			\text {Term-\circled{3}}&\leq  \frac{ \gamma^2 }{n} \mathbb{E}_{\mathbf{\sigma}}[\left\vert\sum_{j=1}^{n} \mathbf{\sigma}_{j}\right \vert] 
			\\&\leq  \frac{ \gamma^2 }{n}[\mathbb{E}_{\mathbf{\sigma}}[\left\vert\sum_{j=1}^{n} \mathbf{\sigma}_{j}\right \vert^{2}]]^{1/2} 
			\\&	\leq  \frac{\gamma^2}{\sqrt{n}} \ \  (\text{By Eq. (\ref{In_10})} ).
		\end{aligned}  
	\end{equation}		
	Then, according to the results from(\ref{Term10}), (\ref{Term20}) and (\ref{Term30}), we reformulate Inequality\,(\ref{R_ineqaility0}) as
	\begin{equation*}
		\begin{aligned}
			\mathfrak{R}_{n}(\mathcal{G}_{\mathbf{U}^{(s)}) } \leq \frac{4\gamma^2}{\sqrt{n}}.
		\end{aligned}
	\end{equation*}
	Finally, since $\widetilde{\mathfrak{R}}_{n}(\mathcal{G}_{\mathbf{U}^{(s)}})=\sup _{D \in \mathcal{X}^{n}} \mathfrak{R}_{n}(\mathcal{G}_{\mathbf{U}^{(s)}})$,
	we get
	\begin{equation*}
		\begin{aligned}
			\max_{s}{\widetilde{\mathfrak{R}}}_{n}(\mathcal{G}_{\mathbf{U}^{(s)}})  \leq  \frac{4\gamma^2}{\sqrt{n}}.
		\end{aligned}
	\end{equation*}
	This proves the statements in Lemma \ref{R_vallue_max0}.	
\end{proof}

\subsection{Proof of Theorem \ref{Data Reconstruction generalization error}}
\label{Data Reconstruction generalization error proof}

\begin{proof}
	For the FIC-DR algorithm, we apply the conclusion in Theorem \ref{k-means  generalization error bound} to derive:
	\begin{equation}
		\label{errorinequality}
		\begin{aligned}
			\mathcal{L}_{\mathbb{Q}_c}(\mathbf{U}_{2}) &\leq  \hat{\mathcal{L}}_{D_{c}}(\mathbf{U}_{2})+ \frac{\alpha_2+2\sqrt{2\log (1 / \delta)}}{\sqrt{n_1+n_2}}
			\\ & \quad +\frac{3\beta_2\sqrt{\log (1 / \delta)}}{\sqrt{n_1+n_2}}+\frac{6\gamma\log (1 / \delta)}{n_1+n_2}.
		\end{aligned}
	\end{equation}
	To simplify the proof expression, let 
	$$A = \frac{\alpha_2+(2\sqrt{2}+3\beta_2)\sqrt{\log (1 / \delta)}}{\sqrt{n_1+n_2}}
	+\frac{6\gamma\log (1 / \delta)}{n_1+n_2},$$ then
	Inequality (\ref{errorinequality}) can be restated as
	\begin{align}
		\mathcal{L}_{\mathbb{Q}_c}(\mathbf{U}_{2})& \leq  \hat{\mathcal{L}}_{D_{c}}(\mathbf{U}_{2})+A \nonumber \\
		&= \hat{\mathcal{L}}_{\tilde{D}_{p_{\boldsymbol{\omega}}}}(\mathbf{U}_{2})+\hat{\mathcal{L}}_{D_{c}}(\mathbf{U}_{2})-\hat{\mathcal{L}}_{\tilde{D}_{p_{\boldsymbol{\omega}}}}(\mathbf{U}_{2})+A \label{weighted empirical clustering risk}\\
		&\leq \hat{\mathcal{L}}_{\tilde{D}_{p_{\boldsymbol{\omega}}}}(\mathbf{U}_{2})+\operatorname{disc}_\mathcal{Y}(\tilde{D}_{p_{\boldsymbol{\omega}}}, D_c)+A.
		\label{disc inequality}
	\end{align}
	In Eq.\,(\ref{weighted empirical clustering risk}), we incorporate the weighted empirical clustering risk. In  Inequality (\ref{disc inequality}), considering  any $\tilde{D}_{
		p_{\boldsymbol{\omega}}}$, we utilize the definition given in Eq. (23) of the main paper to derive  the inequality: $$\hat{\mathcal{L}}_{D_{c}}(\mathbf{U}_{2})-\hat{\mathcal{L}}_{\tilde{D}_{p_{\boldsymbol{\omega}}}}(\mathbf{U}_{2})\leq \operatorname{disc}_ \mathcal{Y}(\tilde{D}_{p_{\boldsymbol{\omega}}}, D_c).$$ Then, by substituting $A$ into (\ref{disc inequality}), the proof of Theorem \ref{Data Reconstruction generalization error} is completed.
\end{proof}

\subsection{Proof of Theorem \ref{Model Resue theorem}}
\label{Model Resue theorem proof}
We first recall that the objective function of the FIC-MR algorithm is expressed as
\begin{equation}
	\label{modelresue_Ob11}
	\min _{\mathbf{U}_4 \in \mathcal{H}^k}(\underbrace{\frac{1}{n_{2}} \sum_{i=1}^n \min _{s=1, \ldots, k}\lVert \mathbf{x}_i-\mathbf{u}_s\rVert ^2}_{\text{$\hat{\mathcal{L}}_{D_c}(\mathbf{U}_4)$}}+ \theta\lVert \mathbf{U}_4^{(1)} - \mathbf{U}^{0}\rVert ^2).
\end{equation}
Noting that $\mathbf{U}_4=(\mathbf{U}_4^{(1)}, \mathbf{U}_4^{(2)})$, let us establish an upper bound for the domain of the cluster center $\mathbf{U}_4^{(1)}$. 
We define the empirical risk minimizer of Eq.\,(\ref{modelresue_Ob11}) by 
$\tilde{\mathbf{U}}_{4} =(\tilde{\mathbf{U}}_{4} ^{(1)}, \tilde{\mathbf{U}}_{4} ^{(2)})$. 
Let $\hat{\mathbf{U}}_{4}=(\mathbf{U}^{0},\tilde{\mathbf{U}}_{4}^{(2)}), 
$
incorporating the optimized cluster center $\mathbf{U}^{0}$ from the previous stage and $\tilde{\mathbf{U}}_{4}^{(2)}$. Given the optimality of the empirical objective at $\tilde{\mathbf{U}}_{4}$, we can derive 
\begin{equation*}
	\begin{aligned}
		&\hat{\mathcal{L}}_{D_c}(\tilde{\mathbf{U}}_{4})+\theta\lVert \tilde{\mathbf{U}}_{4}^{(1)}-\mathbf{U}^{0}\rVert _F^2 
		\\&\leq \hat{\mathcal{L}}_{D_c}(\hat{\mathbf{U}}_{4})+\theta\lVert \mathbf{U}^{0}-\mathbf{U}^{0}\rVert _F^2
		\\&= \hat{\mathcal{L}}_{D_c}(\hat{\mathbf{U}}_{4}).
	\end{aligned}
\end{equation*}
For the decomposition of the empirical risk on both sides of $\hat{\mathcal{L}}_{D_c}(\tilde{\mathbf{U}}_{4})+\theta\lVert \tilde{\mathbf{U}}_{4}^{(1)}-\mathbf{U}^{0}\rVert _F^2 \leq \hat{\mathcal{L}}_{D_c}(\hat{\mathbf{U}}_{4})$, we get
\begin{equation}
	\label{objective2}
	\begin{aligned}
		&\Gamma_{n_2}(\tilde{\mathbf{U}}_{4}^{(1)})+\Gamma_{n_2}(\tilde{\mathbf{U}}_{4}^{(2)})+\theta\lVert \tilde{\mathbf{U}}_{4}^{(1)}-\mathbf{U}^{0}\rVert _F^2 
		\\& \leq \Gamma_{n_2}(\mathbf{U}^{0})+\Gamma_{n_2}(\tilde{\mathbf{U}}_{4}^{(2)}),
	\end{aligned}
\end{equation}
where $\Gamma_{n_2}(\tilde{\mathbf{U}}_{4}^{(1)})$ and $\Gamma_{n_2}(\mathbf{U}^{0})$ are empirical clustering risks based on the feature matrix $\mathbf{X}_2^{(1)}$ corresponding to the cluster center $\tilde{\mathbf{U}}_{4}^{(1)}$ and $\mathbf{U}^{0}$, respectively. $\Gamma_{n_2}(\tilde{\mathbf{U}}_{4}^{(2)})$ is the empirical clustering risk based on the feature matrix $\mathbf{X}_2^{(2)}$ corresponding to the cluster center $\tilde{\mathbf{U}}_{4}^{(2)}$. By (\ref{objective2}), it is easy to 
get $\lVert \tilde{\mathbf{U}}_{4}^{(1)}-\mathbf{U}^{0}\rVert _F\leq \sqrt{\frac{\Gamma_{n_2}(\mathbf{U}^0)}{\theta}}$ and $\Gamma_{n_2}(\tilde{\mathbf{U}}_{4}^{(1)})\leq \Gamma_{n_2}(\mathbf{U}^{0})$.
Furthermore, we also have
\begin{equation*}
	\begin{aligned}
		\lVert \tilde{\mathbf{U}}_{4} ^{(1)}\rVert _F & =\lVert \tilde{\mathbf{U}}_{4} ^{(1)}-\mathbf{U}^{0}+\mathbf{U}^{0}\rVert _F \\
		& \leq\lVert \tilde{\mathbf{U}}_{4} ^{(1)}-\mathbf{U}^{0}\rVert _F+\lVert \mathbf{U}^{0}\rVert _F \\
		&\leq \sqrt{\frac{\Gamma_{n_2}(\mathbf{U}^{0})}{\theta}}+\lVert \mathbf{U}^{0}\rVert _F.
	\end{aligned}
\end{equation*}
Therefore, the $\tilde{\mathbf{U}}_{4} ^{(1)}$ resides within a bounded domain
\begin{equation}
	\label{U_bound}
	\begin{aligned}
		&\tilde{\mathbf{U}}_{4} ^{(1)}\in \mathcal{C}=\left\{\mathbf{U}_4^{(1)}\in \mathbb{R}^{d_1 \times k},\lVert \mathbf{U}_{4}^{(1)}-\mathbf{U}^{0}\rVert _F\leq \sqrt{\frac{\Gamma_{n_2}(\mathbf{U}^0)}{\theta}},\right.\\& \left.\Gamma_{n_2}(\mathbf{U}_4^{(1)}) \leq\Gamma_{n_2}(\mathbf{U}^0)\right\}.
	\end{aligned}
\end{equation}

We utilize similar proof tools for the generalization error bound of the $k$-means algorithm
to establish the proof of Theorem \ref{Model Resue theorem}. The distinction arises from the fact that the objective function considered for the FIC-MR algorithm includes additional regularization terms, leading to a different clustering Rademacher complexity. 
We begin by analyzing the clustering Rademacher complexity in the following Lemma \ref{Rademacher complexity_model resue}.
\begin{lemma}
	\label{Rademacher complexity_model resue}
	Assuming that $\lVert \mathbf{x}\rVert \leq \gamma$ for $ \forall\mathbf{x}\in \mathbb{Q}_c$, consider any sample set $D_c=\left\{\mathbf{x}_{1}, \ldots, \mathbf{x}_{n_2}\right\}$ drawn from the input space $\mathcal{X}_{c}$, there exists a positive constant $V_4$, it holds that
	$$
	\mathfrak{R}_{n_2}(\mathcal{J}_{\mathbf{U}_4}) \leq 2V_4 \varepsilon \sqrt {\frac{k}{n_2}}\log ^{2}(2n_2\sqrt{(\gamma+\gamma_4)n_2}),
	$$
	where
	$\varepsilon= \gamma(\Delta+ \sqrt{\gamma_4^{2}- (\gamma_0-\Delta)^{2}}+\frac{\gamma }{2})+\frac{ (\gamma_4' +\gamma_0 ) \Delta +\gamma_4''^2}{2}$. In the current stage, recall that any  cluster center $\mathbf{u}$ can decomposed into $\mathbf{u}=[\mathbf{u}^{(1)},\mathbf{u}^{(2)}]$. Here, we suppose 
	$ \sup_{\mathbf{u}\in\mathcal{H}}\lVert  \mathbf{u}^{(1)}-\mathbf{u}^{0}\rVert \leq \Delta$, where $\mathbf{u}^{0}$ is the pre-trained cluster center in the previous stage.
\end{lemma}

\noindent The proof of Lemma  \ref{Rademacher complexity_model resue} is presented in Section \ref{Rademacher complexity_model resue section}.

\begin{proof}[\textbf{Proof of Theorem \ref{Model Resue theorem} }]
	We utilize $\mathfrak{R}_{n_2}(\mathcal{J}_{\mathbf{U}_4}) $ to bound the expected CRC as follows
	\begin{equation}
		\label{R_expecte_dvalue_bound}
		\begin{aligned}
			\mathfrak{R}(\mathcal{J}_{\mathbf{U}_4})&=\mathbb{E}[\mathfrak{R}_{n_2}(\mathcal{J}_{\mathbf{U}_4})]\\
			&\leq \mathbb{E}[2V \varepsilon \sqrt {\frac{k}{n_2}}\log ^{2}(2 n_2\sqrt{(\gamma+\gamma_4)n_2})]\\
			& = 2V\sqrt {\frac{k}{n_2}}\log ^{2}(2 n_2\sqrt{(\gamma+\gamma_4)n_2})\mathbb{E} [ \varepsilon ],
		\end{aligned}
	\end{equation}
	where $\varepsilon= \gamma(\Delta+ \sqrt{\gamma_4^{2}- (\gamma_0-\Delta)^{2}}+\frac{\gamma }{2})+\frac{ (\gamma_4' +\gamma_0 ) \Delta +\gamma_4''^2}{2}$.
	For  $\mathbb{E} [ \varepsilon ]$, let
	$\mathbb{E} [ \varepsilon ] = B_1+B_2$, where 
	$ B_1 =\mathbb{E} [ \gamma(\Delta+ \sqrt{\gamma_4^{2}- (\gamma_0-\Delta)^{2}})+\frac{\gamma}{2}] $ and 
	$ B_2 = \mathbb{E} [ \frac{ (\gamma_4' +\gamma_0 ) \Delta +\gamma_4''^2}{2}]. $
	It is easy to obtain
	\begin{equation*}
		\begin{aligned}
			B_1 &= \mathbb{E} [ \frac{ (\gamma_4' +\gamma_0 ) \Delta +\gamma_4''^2}{2}] \\
			&=\gamma( \mathbb{E}[ \Delta]+ \mathbb{E}[ \sqrt{\gamma_4^{2}- (\gamma_0-\Delta)^{2}}])+\frac{\gamma}{2}.
		\end{aligned}
	\end{equation*}
	Notice that the square-root function is concave, by applying the Jensen’s inequality, we have 
	\begin{equation}
		\begin{aligned}
			&\mathbb{E}[ \sqrt{\gamma_4^{2}- (\gamma_0-\Delta)^{2}}] \\
			&\leq \sqrt{\gamma_4^{2}-\mathbb{E}[ (\gamma_0-\Delta)^{2}]}  \\
			& \leq \sqrt{\gamma_4^{2}-(\mathbb{E}[ \gamma_0-\Delta])^{2}}  \\
			&= \sqrt{\gamma_4^{2}-(\gamma_0-\mathbb{E}[\Delta])^{2}}.
		\end{aligned}
	\end{equation}
	Thus, it holds
	\begin{equation*}
		\begin{aligned}
			&B_1 \leq  \gamma( \mathbb{E}[ \Delta]+ \sqrt{\gamma_4^{2}-(\gamma_0-\mathbb{E}[\Delta])^{2}}  )+\frac{\gamma}{2}.
		\end{aligned}
	\end{equation*}
	Besides, we have
	\begin{equation*}
		\begin{aligned}
			B_2 & = \mathbb{E} [ \frac{ (\gamma_4' +\gamma_0 ) \Delta +\gamma_4''^2}{2}]\\
			& = \frac{\gamma_4' +\gamma_0}{2}\mathbb{E} [ \Delta ]+\frac{\gamma_4''^2}{2}.
		\end{aligned}
	\end{equation*}		
	According to $\lVert \mathbf{U}_{4}^{(1)}-\mathbf{U}^{0}\rVert _F\leq \sqrt{\frac{\Gamma_{n_2}(\mathbf{U}^0)}{\theta}}$  in (\ref{U_bound})
	and 
	$ \sup_{\mathbf{u}\in\mathcal{H}}\lVert  \mathbf{u}^{(1)}-\mathbf{u}^{0}\rVert \leq \Delta$, there is
	\begin{equation*}
		\label{expect_delta}
		\begin{aligned}
			\mathbb{E}[ \Delta] \leq& \mathbb{E}[  \sqrt{\frac{\Gamma_{n_2}(\mathbf{U}^0)}{\theta}}] 
			\leq \sqrt{ \frac{\mathbb{E}[ \Gamma_{n_2}(\mathbf{U}^0)] }{\theta}}
			=  \sqrt{\frac{\Gamma(\mathbf{U}^0)  }{\theta}}.
		\end{aligned}
	\end{equation*}
	where $ \Gamma(\mathbf{U}^0) = \mathbb{E}[ \Gamma_{n_2}(\mathbf{U}^0)] $.
	Based on the above analysis, 
	we can conclude that
	\begin{equation}
		\label{R_exp_inequality1}
		\begin{aligned}
			\mathfrak{R}(\mathcal{J}_{\mathbf{U}}) &\leq B_1+B_2
			\leq C[\gamma (\sqrt{\varepsilon_1} + 	\sqrt{\varepsilon_2}+\frac{\gamma}{2})+\varepsilon_3] ,
		\end{aligned}
	\end{equation}
	where $\varepsilon_1= \sqrt{\frac{\Gamma(\mathbf{U}^0) }{\theta}}$,
	$\varepsilon_2 = \gamma_4^{2}-(\gamma_0-\varepsilon_1)^{2}$, $\varepsilon_3 = \frac{\gamma_4' +\gamma_0}{2}\varepsilon_1+\frac{\gamma_4''^2}{2}$ and  $C = 2V_4\sqrt {\frac{k}{n_2}}\log ^{2}(2n_2 \sqrt{(\gamma+\gamma_4)n_2})$.
	Next, by using Lemma \ref{Generalization error 0}, we substitute Inequality  (\ref{R_exp_inequality1}) into Inequality (\ref{generalization error inequality}) of the model reuse framework. Consequently, Inequality (25) of the main paper can be written as
	\begin{equation*}
		\begin{aligned}
			\mathcal{L}_{\mathbb{Q}_{c}}(\mathbf{U}_4)\leq &\hat{\mathcal{L}}_{D_c}(\mathbf{U}_4) 
			+\frac{\alpha_4+3\beta_4\sqrt{\log (1 / \delta)}}{\sqrt{n_2}}+\frac{6\gamma\log (1 / \delta)}{n_2},
		\end{aligned}
	\end{equation*}
	where $\alpha_4 =4V_4\varepsilon \sqrt {k}\log ^{2}(2n_2\sqrt{(\gamma+\gamma_4) n_2})$, $\varepsilon= \gamma(\sqrt{\frac{\Gamma(\mathbf{U}^0)  }{\theta}}+ \sqrt{\gamma_4^{2}- (\gamma_0-\sqrt{\frac{\Gamma(\mathbf{U}^0)  }{\theta}})^{2}}+\frac{\gamma }{2})+\frac{ (\gamma_4' +\gamma_0 ) \sqrt{\frac{\Gamma(\mathbf{U}^0)  }{\theta}} +\gamma_4''^2}{2}$, and
	$\beta_4 = \sqrt{(\lambda_4+4\alpha_4)\gamma}.$ 
	Similarly, we can derive the definition of  $\lambda_4=\mathbb{E}_{\mathbf{x}\sim \mathbb{Q}_c} [\sup _{g_{\mathbf{U}_4}\in\mathcal{G}_{\mathbf{U}_4}}\hat{\mathcal{L}}_{D_c}(\mathbf{U}_4) ]$ using Lemma \ref{Generalization error 0}  within the model reuse framework.
	We further have
	\begin{equation*}
		\begin{aligned}
			\lambda_4
			& \leq  \mathbb{E}_{\mathbf{x}\sim \mathbb{Q}}[\hat{\mathcal{L}}_{D_c}(\hat{\mathbf{U}}_4) ] = \mathcal{L}_{Q_c}(\hat{\mathbf{U}}_4),
		\end{aligned}
	\end{equation*}
	where $\hat{\mathbf{U}}_{4}=(\mathbf{U}^{0},\tilde{\mathbf{U}}_{4}^{(2)})$.
	Thus, we can replace $\lambda_4$ with $\mathcal{L}_{Q_c}(\hat{\mathbf{U}}_4)$ to obtain $\beta_4 = \sqrt{(\mathcal{L}_{Q_c}(\hat{\mathbf{U}}_4)+4\alpha_4)\gamma}$.
	Then,  the proof of Theorem \ref{Model Resue theorem} is finished.
\end{proof}

\subsection{Proof of Lemma \ref{Rademacher complexity_model resue}}
\label{Rademacher complexity_model resue section}
To prove Lemma \ref{Rademacher complexity_model resue}, we give the following Lemma \ref{R_vallue_max1}. 
\begin{lemma}
	\label{R_vallue_max1}
	Assuming that $\lVert \mathbf{x}\rVert \leq \gamma$ for $ \forall\mathbf{x}\in \mathcal{X}_c$, consider any sample set $D_c=\left\{\mathbf{x}_{1}, \ldots, \mathbf{x}_{n_2}\right\}$ drawn from the input space $\mathcal{X}_c$ and $\mathbf{U}_4\in \mathcal{H}^{k}$. We let $\widetilde{\mathfrak{R}}_{n}(\mathcal{G}_{\mathbf{U}_{4}^{(s)}})=\sup _{D \in \mathcal{X}} \mathfrak{R}_{n}(\mathcal{G}_{\mathbf{U}_{4}^{(s)}})$,
	Then, we obtain
	\begin{equation}
		\label{max_R_unper bound1}
		\begin{aligned}
			&\max_{s}{\widetilde{\mathfrak{R}}}_{n_2}(\mathcal{G}_{\mathbf{U}_{4}^{(s)}})  \leq
			\frac{2\varepsilon}{\sqrt{n_2}},\\
		\end{aligned}
	\end{equation}
	where $\varepsilon= \gamma(\Delta+ \sqrt{\gamma_4^{2}- (\gamma_0-\Delta)^{2}}+\frac{\gamma }{2})+\frac{ (\gamma_4' +\gamma_0 ) \Delta +\gamma_4''^2}{2}$ and  $\mathbf{U}_{4}^{(s)}$ is the output coordinate $s$ of $\mathbf{U}_4$.  
\end{lemma}
\noindent We provide the proof of Lemma \ref{R_vallue_max1} in \ref{SectionR_vallue_max1proof}.	

\noindent  
\begin{proof} [\textbf{Proof of Lemma \ref{Rademacher complexity_model resue}}]
	Next, using Lemma \ref{Rademacher complexity_property0} and Lemma \ref{R_vallue_max1}, we proceed to provide a comprehensive proof of Lemma \ref{Rademacher complexity_model resue}. Let $\psi_4(\mathbf{q})$  be the minimum function in the framework of FIC-MR algorithm. By utilizing a similar proof process as in (\ref{psi_value_max}), we can prove that
	the function $\psi_4(\mathbf{q})$ is 1-Lipschitz continuous w.r.t. the $L_{\infty}$-norm, i.e.,
	$\forall \mathbf{q}, \mathbf{q}^{\prime} \in \mathbb{R}^{k},\lvert \psi_4(\mathbf{q})-\psi_4(\mathbf{q}^{\prime})\rvert \leq\lVert \mathbf{q}-\mathbf{q}^{\prime}\rVert _{\infty}$. 
	Similarly, note that 	$g_{\mathbf{U}_4}(\mathbf{x})=(g_{\mathbf{u}_{1}}(\mathbf{x}),\ldots,g_{\mathbf{u}_{k}}(\mathbf{x}))$ with $g_{\mathbf{u}_{s}}(\mathbf{x})=\lVert  \mathbf{x}-\mathbf{u}_{s} \rVert ^2, \lVert \mathbf{x}\rVert  \leq \gamma$. Let $\lVert \mathbf{u}_{s}\rVert  \leq \gamma_4 \ \text{with}\  s=1,\ldots,k$  in the FIC-MR algorithm, then we have
	\begin{equation}
		\label{gx_sup}
		\begin{aligned}
			&g_{\mathbf{u}_{s}}(\mathbf{x})=\lVert  \mathbf{x}-\mathbf{u}_{s} \rVert ^2\leq 2( \lVert \mathbf{x}\rVert +\lVert \mathbf{u}_{s}\rVert ) \leq 2(\gamma+\gamma_4).
		\end{aligned}
	\end{equation}
	Hence, it is easy to get
	\begin{equation*}
		\lvert \psi_4(g_{\mathbf{U}}(\mathbf{x}))\rvert=
		\lvert \min(g_{\mathbf{u}_{1}},\ldots,g_{\mathbf{u}_{k}})\rvert \leq 2(\gamma+\gamma_4)
	\end{equation*}
	and 
	\begin{equation*}
		\lVert g_{\mathbf{U}_4}(\mathbf{x})\rVert _{\infty}=\max _{s=1,\ldots,k}\lvert g_{\mathbf{u}_{s}}(\mathbf{x})\rvert \leq 2(\gamma+\gamma_4).
	\end{equation*}
	According to the above analysis, one can see that the function $\psi_4(\mathbf{q})$  is 1-Lipschitz continuous w.r.t. the $L_{\infty}$-norm and $\max \left\{\lvert \psi_4(g_{\mathbf{U}_4}(\mathbf{x}))\rvert,\lVert g_{\mathbf{U}_4}(\mathbf{x})\rVert _{\infty}\right\} \leq 2(\gamma+\gamma_4)$. Therefore, leveraging  Lemma \ref{Rademacher complexity_property0} with $L=1$,  $b= 2(\gamma+\gamma_4)$ , and $r=1 / 2$, there exists a constant $V_4$, then we can achieve
	\begin{equation}
		\label{R_value_inequality}
		\mathfrak{R}_{n_2}(\mathcal{J}_{\mathbf{U}_4}) \leq V_4 \sqrt{k} \max _{s} \widetilde{\mathfrak{R}}_{n_2}(\mathcal{G}_{\mathbf{U}_{4}^{(s)}}) \log ^{2}(\frac{2n_2(\gamma+\gamma_4)}{\max _{s} \widetilde{\mathfrak{R}}_{n_2}(\mathcal{G}_{\mathbf{U}_{4}^{(s)}})}),
	\end{equation} 
	where $\mathcal{G}_{\mathbf{U}_{4}^{(s)}}$ is a family of the output coordinate $s$ of $\mathcal{G}_{\mathbf{U}_4}$,
	$\widetilde{\mathfrak{R}}_{n_2}(\mathcal{G}_{\mathbf{U}_{4}^{(s)}})=\sup _{D_c \in \mathcal{X}_{c}^{n_2}} \mathfrak{R}_{n}(\mathcal{G}_{\mathbf{U}_{4}^{(s)}})$.
	
	\noindent By utilizing a similar proof process as in (\ref{R_value0}), we have
	\begin{equation*}
		\begin{aligned}
			\widetilde{\mathfrak{R}}_{n_2}(\mathcal{G}_{\mathbf{U}_{4}^{(s)}}) 
			&	=\sup _{D_c \in \mathcal{X}_{c}^{n_2}} \mathfrak{R}_{n_2}(\mathcal{G}_{\mathbf{U}_{4}^{(s)}})\\
			&	\geq  \frac{\sqrt{2}}{2\sqrt{n_2}}\sup _{\mathbf{x} \in \mathcal{X}, g_{\mathbf{u}} \in \mathcal{G}_{\mathbf{U}_{4}^{(s)}}}( \lvert g_{\mathbf{u}}(\mathbf{x})\rvert)^{1/2}.
		\end{aligned}
	\end{equation*}
	Denote
	$u_{s}=\sup _{\mathbf{x} \in \mathcal{X}, g_{\mathbf{u}} \in \mathcal{G}_{\mathbf{U}_{4}^{(s)}}}\lvert g_{\mathbf{u}}(\mathbf{x})\rvert;$ $u=\max \left\{u_{s}, s=1, \ldots, k\right\}.$
	It is easy to obtain
	$
	\widetilde{\mathfrak{R}}_{n}(\mathcal{G}_{\mathbf{U}_{4}^{(s)}}) \geq\frac{\sqrt{2u_{s}}}{2\sqrt{n_2}}.
	$
	Then, there is  $	\max _{s} \widetilde{\mathfrak{R}}_{n_2}(\mathcal{G}_{\mathbf{U}_{4}^{(s)}})\geq \frac{\sqrt{2u}}{2\sqrt{n_2}}$.
	Using Eq.\,(\ref{gx_sup}), we know $\max _{s}\lvert g_{\mathbf{u}_{s}}(\mathbf{x})\rvert \leq 2(\gamma+\gamma_4)$. Hence, one can see that $u\leq 2(\gamma+\gamma_4)$. 
	Then, there is
	\begin{equation*}
		\max _{s} \widetilde{\mathfrak{R}}_{n_2}(\mathcal{G}_{\mathbf{U}_{4}^{(s)}}) \geq\sqrt{\frac{(\gamma+\gamma_4)}{n_2}}.
	\end{equation*}
	So, the following inequality holds
	\begin{equation}
		\label{max_R_lowerbound}
		\frac{2n_2(\gamma+\gamma_4)}{\max _{s} \widetilde{\mathfrak{R}}_{n_2}(\mathcal{G}_{\mathbf{U}_{4}^{(s)}})} \leq 
		2 n_2\sqrt{(\gamma+\gamma_4)n_2}.
	\end{equation}
	We substitute (\ref{max_R_lowerbound}) and (\ref{max_R_unper bound1}) into (\ref{R_value_inequality}) to conclude the proof of Lemma \ref{Rademacher complexity_model resue}.
\end{proof}		

\subsection{Proof of Lemma \ref{R_vallue_max1}}
\label{SectionR_vallue_max1proof}
\begin{proof}
	According to Inequality (\ref{R_ineqaility0}), for any sample set $D_c$ and $\mathbf{U}_4\in \mathcal{H}^{k}$, it is easy to 
	\begin{equation}
		\label{R_ineqaility1}
		\begin{aligned}
			\mathfrak{R}_{n_2}(\mathcal{G}_{\mathbf{U}_{4}^{(s)}})&\leq  \underbrace{\mathbb{E}_{\mathbf{\sigma}} [\lvert \frac{1}{n_2} \sum_{j=1}^{n} \mathbf{\sigma}_{j}\lVert   \mathbf{x}_{j}\rVert ^2\rvert]}_{\text {Term-\circled{1}}}
			\\&\quad+\underbrace{2\mathbb{E}_{\mathbf{\sigma}} [\sup_{\mathbf{u}\in\mathcal{H}} \lvert \frac{1}{n_2}\sum_{j=1}^{n}\mathbf{\sigma}_{j}\left \langle \mathbf{x}_{j},\mathbf{u}\right \rangle \rvert ]}_{\text {Term-\circled{2}}}	
			\\&\quad
			+\underbrace{\mathbb{E}_{\mathbf{\sigma}} [\sup_{\mathbf{u}\in\mathcal{H}}\lvert \frac{1}{n_2}\sum_{j=1}^{n} \mathbf{\sigma}_{j}\lVert  \mathbf{u}\rVert ^2\rvert]}_{\text {Term-\circled{3}}}.\\
		\end{aligned}
	\end{equation}
	For the Term $\circled{1}$, employing the same proof process as in (\ref{Term10}), we also obtain
	\begin{equation}
		\label{Term1}
		\text {Term-\circled{1}} \leq  \frac{\gamma^2} {\sqrt{n_2}}.
	\end{equation}
	For the $\text {Term-\circled{2}}$, as the cluster center $\mathbf{u}$ and the data point $\mathbf{x}_{j} $ can be decomposed into $\mathbf{u}=[\mathbf{u}^{(1)},\mathbf{u}^{(2)}]$ and $\mathbf{x}_{j}=[\mathbf{x}_{j}^{(1)},\mathbf{x}_{j}^{(2)}]$, respectively, considering the $\text {Term-\circled{2}}$ of (\ref{R_ineqaility1}), it becomes evident that
	\begin{equation}
		\label{term11}
		\begin{aligned}
			&\mathbb{E}_{\mathbf{\sigma}} [\sup_{\mathbf{u}\in\mathcal{H}} \lvert \frac{1}{n_2}\sum_{j=1}^{n}\mathbf{\sigma}_{j}\left \langle \mathbf{x}_{j},\mathbf{u}\right \rangle \rvert] \\& = \mathbb{E}_{\mathbf{\sigma}} [\sup_{\mathbf{u}\in\mathcal{H}}\lvert \frac{1}{n_2}\left \langle \sum_{j=1}^{n}\mathbf{\sigma}_{j}	 \mathbf{x}_{j},\mathbf{u}\right \rangle \rvert ]\\
			& = \mathbb{E}_{\mathbf{\sigma}} [\sup_{\mathbf{u}\in\mathcal{H}}
			\lvert \frac{1}{n_2}\left \langle \sum_{j=1}^{n}\mathbf{\sigma}_{j}	 \mathbf{x}_{j}^{(1)},\mathbf{u}^{(1)}\right \rangle \rvert ] \\&\quad +\mathbb{E}_{\mathbf{\sigma}} [\sup_{\mathbf{u}\in\mathcal{H}}
			\lvert \frac{1}{n_2}\left \langle \sum_{j=1}^{n}\mathbf{\sigma}_{j}	 \mathbf{x}_{j}^{(2)},\mathbf{u}^{(2)}\right \rangle\rvert ]\\
			&= \mathbb{E}_{\mathbf{\sigma}} [\sup_{\mathbf{u}\in\mathcal{H}}\lvert \frac{1}{n_2}\left \langle \sum_{j=1}^{n}\mathbf{\sigma}_{j}	 \mathbf{x}_{j}^{(1)},\mathbf{u}^{(1)}-\mathbf{u}^{0} \right \rangle \rvert ]
			\\&\quad +\mathbb{E}_{\mathbf{\sigma}} [\sup_{\mathbf{u}\in\mathcal{H}}
			\lvert \frac{1}{n_2}\left \langle \sum_{j=1}^{n}\mathbf{\sigma}_{j}	 \mathbf{x}_{j}^{(2)},\mathbf{u}^{(2)}\right \rangle \rvert ]\\
		\end{aligned}
	\end{equation}

	\begin{equation}
		\begin{aligned}
			&\leq \mathbb{E}_{\mathbf{\sigma}} [\sup_{\mathbf{u}\in\mathcal{H}} \frac{1}{n_2}\lVert \sum_{j=1}^{n}\mathbf{\sigma}_{j}	 \mathbf{x}_{j}^{(1)}\rVert  \lVert \mathbf{u}^{(1)}-\mathbf{u}^{0} \rVert ]
			\\&\quad +\mathbb{E}_{\mathbf{\sigma}} [\sup_{\mathbf{u}\in\mathcal{H}}\frac{1}{n_2}\lVert  \sum_{j=1}^{n}\mathbf{\sigma}_{j}	 \mathbf{x}_{j}^{(2)}\rVert  \lVert \mathbf{u}^{(2)} \rVert  ].\\
		\end{aligned}
	\end{equation}
	In the FIC-MR algorithm, recall that we let $\lVert \mathbf{u}\rVert \leq \gamma_4$, there is
	\begin{equation*}
		\begin{aligned}
			&\lVert \mathbf{u}\rVert ^{2}	\leq \gamma_4^{2}
			\Rightarrow\lVert \mathbf{u}^{(1)}\rVert ^{2} +	\lVert \mathbf{u}^{(2)}\rVert ^{2}\leq \gamma_4^{2}
			\\&	\Rightarrow\lVert \mathbf{u}^{(2)}\rVert ^{2}\leq \gamma_4^{2}- \lVert \mathbf{u}^{(1)}\rVert ^{2}.\\
		\end{aligned}
	\end{equation*}
	Note that $ \sup_{\mathbf{u}\in\mathcal{H}}\lVert  \mathbf{u}^{(1)}-\mathbf{u}^{0}\rVert \leq \Delta$ and $\lVert  \mathbf{u}^{0}\rVert  = \gamma_0$.  According to the properties of norm, one can see that 
	\begin{equation*}
		\begin{aligned}
			&\lVert  \mathbf{u}^{(0)}\rVert -\lVert  \mathbf{u}^{1}\rVert \leq \lVert \mathbf{u}^{(0)} - \mathbf{u}^{(1)} \rVert  \leq \Delta\\
			&\Rightarrow -\lVert  \mathbf{u}^{1}\rVert \leq \Delta-\lVert  \mathbf{u}^{(0)}\rVert \leq  \Delta-\gamma_0\\
			&\Rightarrow-\lVert  \mathbf{u}^{1}\rVert ^{2}\leq -(\gamma_0-\Delta)^{2}.\\
		\end{aligned}
	\end{equation*}
	Thus, we have
	\begin{equation}
		\label{mu_2bound}
		\begin{aligned}
			&\lVert \mathbf{u}^{(2)}\rVert ^{2}\leq \gamma_4^{2} -(\gamma_0-\Delta)^{2}\Rightarrow \lVert \mathbf{u}^{(2)}\rVert \leq \sqrt{\gamma_4^{2}- (\gamma_0-\Delta)^{2}}.\\			
		\end{aligned}
	\end{equation}
	Referring to Eq.\,(\ref{In_10}), and considering $\lVert \mathbf{x}\rVert \leq \gamma$, we obtain
	\begin{equation}
		\label{sigma and x1 bound}
		\begin{aligned}
			\mathbb{E}_{\mathbf{\sigma}} [\lVert  \sum_{j=1}^{n}\mathbf{\sigma}_{j}\mathbf{x}_{j}^{(1)}  \rVert  ]
			& \leq  [\mathbb{E}_{\mathbf{\sigma}} [ \lVert  \sum_{j=1}^{n}\mathbf{\sigma}_{j}	 \mathbf{x}_{j}^{(1)}  \rVert ^{2} ]]^{1/2}
			\\&\leq [  \sum_{j=1}^{n}
			\lVert \mathbf{x}_{j}^{(1)}\rVert ^{2} ]^{1/2}
			\\&\leq  \gamma \sqrt{n_2}.
		\end{aligned}
	\end{equation}
	Similarly, there is 
	\begin{equation}
		\label{sigma and x2 bound}
		\begin{aligned}
			\mathbb{E}_{\mathbf{\sigma}} [\lVert  \sum_{j=1}^{n}\mathbf{\sigma}_{j}\mathbf{x}_{j}^{(2)}  \rVert  ] \leq \gamma \sqrt{n_2}.
		\end{aligned}
	\end{equation}
	We know that $ \sup_{\mathbf{u}\in\mathcal{H}}\lVert  \mathbf{u}^{(1)}-\mathbf{u}^{0}\rVert \leq \Delta$. By Inequalities.\,(\ref{term11}), (\ref{mu_2bound}), (\ref{sigma and x1 bound}), and (\ref{sigma and x2 bound}), we conclude
	\begin{equation}
		\label{Term2}
		\begin{aligned}
			&\text{Term-\circled{2}}
			\leq  \frac{2\gamma}{\sqrt{n_2}} (\Delta+ \sqrt{\gamma_4^{2}- (\gamma_0-\Delta)^{2}}).
		\end{aligned}
	\end{equation}
	Since $\lVert  \mathbf{u}\rVert  \leq \gamma_4$ and $\mathbf{u}=[\mathbf{u}^{(1)},\mathbf{u}^{(2)}]$, we suppose $\lVert  \mathbf{u}^{(1)}\rVert  \leq \gamma_4' \leq\gamma_4 $ and $\lVert  \mathbf{u}^{(2)}\rVert  \leq \gamma_4''\leq\gamma_4$. 
	For the $\text {Term-\circled{3}}$ of (\ref{R_ineqaility1}), it can be seen that 
	\begin{equation}
		\label{Term3}
		\begin{aligned}
			&\mathbb{E}_{\mathbf{\sigma}} [\sup_{\mathbf{u}\in\mathcal{H}}\lvert \frac{1}{n_2}\sum_{j=1}^{n} \mathbf{\sigma}_{j}\lVert  \mathbf{u}\rVert ^2\rvert]  \\
			&=  \mathbb{E}_{\mathbf{\sigma}} [\sup_{\mathbf{u}\in\mathcal{H}}\lvert \frac{1}{n_2}\sum_{j=1}^{n} \mathbf{\sigma}_{j}(\lVert  \mathbf{u}^{(1)}\rVert ^2+\lVert  \mathbf{u}^{(2)}\rVert ^2)\rvert]\\
			&= \mathbb{E}_{\mathbf{\sigma}} [\sup_{\mathbf{u}\in\mathcal{H}}\lvert \frac{1}{n_2}\sum_{j=1}^{n} \mathbf{\sigma}_{j}(\lVert  \mathbf{u}^{(2)}\rVert ^2- \lVert  \mathbf{u}^{0}\rVert ^2+\lVert  \mathbf{u}^{(2)}\rVert ^2)\rvert]\\
			&=\mathbb{E}_{\mathbf{\sigma}} [\sup_{\mathbf{u}\in\mathcal{H}}\lvert \frac{1}{n_2}\sum_{j=1}^{n} \mathbf{\sigma}_{j}( (\mathbf{u}^{(1)}- \mathbf{u}^{0})(\mathbf{u}^{(1)}+\mathbf{u}^{0})^{T}+\lVert  \mathbf{u}^{(2)}\rVert ^2)\rvert]\\
			&\leq\mathbb{E}_{\mathbf{\sigma}}[\sup_{\mathbf{u}\in\mathcal{H}}\frac{1}{n_2}\left\vert\sum_{j=1}^{n} \mathbf{\sigma}_{j}\right \vert \lVert \mathbf{u}^{(1)}- \mathbf{u}^{0}\rVert  \lVert \mathbf{u}^{(1)}+\mathbf{u}^{0}\rVert ] \\&\quad+\mathbb{E}_{\mathbf{\sigma}}[\sup_{\mathbf{u}\in\mathcal{H}}\frac{1}{n_2}\left\vert\sum_{j=1}^{n} \mathbf{\sigma}_{j}\right \vert \lVert  \mathbf{u}^{(2)}\rVert ^2]\\
			& \leq  \frac{(\gamma_4' +\gamma_0) \Delta +\gamma_4''^2}{\sqrt{n_2}}.
		\end{aligned}  
	\end{equation}
	
	Then, according to the results from (\ref{Term1}), (\ref{Term2}) and (\ref{Term3}), we reformulate Inequality\,(\ref{R_ineqaility1}) as
	\begin{equation*}
		\begin{aligned}
			\mathfrak{R}_{n_2}(\mathcal{G}_{\mathbf{U}_{4}^{(s)}) } \leq 
			\frac{2\varepsilon}{\sqrt{n_2}},\\
		\end{aligned}
	\end{equation*}
	where $\varepsilon= \gamma(\Delta+ \sqrt{\gamma_4^{2}- (\gamma_0-\Delta)^{2}}+\frac{\gamma }{2})+\frac{ (\gamma_4' +\gamma_0 ) \Delta +\gamma_4''^2}{2}$.
	Finally, since $\widetilde{\mathfrak{R}}_{n_2}(\mathcal{G}_{\mathbf{U}_{4}^{(s)}})=\sup _{D_c \in \mathcal{X}_c^{n}} \mathfrak{R}_{n_2}(\mathcal{G}_{\mathbf{U}_{4}^{(s)}})$,
	we get
	\begin{equation*}
		\begin{aligned}
			\max_{s}{\widetilde{\mathfrak{R}}}_{n_2}(\mathcal{G}_{\mathbf{U}_{4}^{(s)}})  \leq
			\frac{2\varepsilon}{\sqrt{n_2}}.
		\end{aligned}
	\end{equation*}
	This proves the findings outlined in Lemma \ref{R_vallue_max1}.	
	
\end{proof}

\bibliographystyle{IEEEtran}
\bibliography{IEEEabrv, Manuscript}

\begin{thebibliography}{10}
\providecommand{\url}[1]{#1}
\csname url@samestyle\endcsname
\providecommand{\newblock}{\relax}
\providecommand{\bibinfo}[2]{#2}
\providecommand{\BIBentrySTDinterwordspacing}{\spaceskip=0pt\relax}
\providecommand{\BIBentryALTinterwordstretchfactor}{4}
\providecommand{\BIBentryALTinterwordspacing}{\spaceskip=\fontdimen2\font plus
\BIBentryALTinterwordstretchfactor\fontdimen3\font minus
  \fontdimen4\font\relax}
\providecommand{\BIBforeignlanguage}[2]{{%
\expandafter\ifx\csname l@#1\endcsname\relax
\typeout{** WARNING: IEEEtran.bst: No hyphenation pattern has been}%
\typeout{** loaded for the language `#1'. Using the pattern for}%
\typeout{** the default language instead.}%
\else
\language=\csname l@#1\endcsname
\fi
#2}}
\providecommand{\BIBdecl}{\relax}
\BIBdecl

\bibitem{shi2000normalized}
J.~Shi and J.~Malik, ``Normalized cuts and image segmentation,'' \emph{{IEEE}
  Trans. Pattern Anal. Mach. Intell.}, vol.~22, no.~8, pp. 888--905, 2000.

\bibitem{li2019clustering}
F.~Li, Y.~Qian, J.~Wang, C.~Dang, and L.~Jing, ``Clustering ensemble based on
  sample's stability,'' \emph{Artif. Intell.}, vol. 273, pp. 37--55, 2019.

\bibitem{Mishro2021}
P.~K. Mishro, S.~Agrawal, R.~Panda, and A.~Abraham, ``A novel type-2 fuzzy
  c-means clustering for brain mr image segmentation,'' \emph{{IEEE} Trans.
  Cybern.}, vol.~51, no.~8, pp. 3901--3912, 2021.

\bibitem{Peng2022}
Z.~Peng, H.~Liu, Y.~Jia, and J.~Hou, ``Adaptive attribute and structure
  subspace clustering network,'' \emph{{IEEE} Trans. Image Process.}, vol.~31,
  pp. 3430--3439, 2022.

\bibitem{Zhou2025}
J.~Zhou, C.~Huang, C.~Gao, Y.~Wang, W.~Pedrycz, and G.~Yuan, ``Reweighted
  subspace clustering guided by local and global structure preservation,''
  \emph{{IEEE} Trans. Cybern.}, vol.~55, no.~3, pp. 1436--1449, 2025.

\bibitem{fahy2018ant}
C.~Fahy, S.~Yang, and M.~Gongora, ``Ant colony stream clustering: A fast
  density clustering algorithm for dynamic data streams,'' \emph{{IEEE} Trans.
  Cybern.}, vol.~49, no.~6, pp. 2215--2228, 2018.

\bibitem{sui2020dynamic}
J.~Sui, Z.~Liu, L.~Liu, A.~Jung, and X.~Li, ``Dynamic sparse subspace
  clustering for evolving high-dimensional data streams,'' \emph{{IEEE} Trans.
  Cybern.}, vol.~52, no.~6, pp. 4173--4186, 2020.

\bibitem{chen2022two}
J.~Chen, Z.~Wang, S.~Yang, and H.~Mao, ``Two-stage sparse representation
  clustering for dynamic data streams,'' \emph{{IEEE} Trans. Cybern.}, vol.~53,
  no.~10, pp. 6408--6420, 2022.

\bibitem{Urio2025}
A.~Urio-Larrea, H.~Camargo, G.~Lucca, T.~Asmus, C.~Marco-Detchart, L.~Schick,
  C.~Lopez-Molina, J.~Andreu-Perez, H.~Bustince, and G.~P. Dimuro, ``Data
  stream clustering: Introducing recursively extendable aggregation functions
  for incremental cluster fusion processes,'' \emph{{IEEE} Trans. Cybern.},
  vol.~55, no.~3, pp. 1421--1435, 2025.

\bibitem{cheng2025tabfsbench}
Z.~Cheng, Z.~Jia, Z.~Zhou, Y.~Li, and L.~Guo, ``Tabfsbench: Tabular benchmark
  for feature shifts in open environments,'' in \emph{in Proc. 42th Int. Conf.
  Mach. Learn.}, vol. 267, Vancouver, BC, Canada, 2025.

\bibitem{yan2015egocentric}
Y.~Yan, E.~Ricci, G.~Liu, and N.~Sebe, ``Egocentric daily activity recognition
  via multitask clustering,'' \emph{{IEEE} Trans. Image Process.}, vol.~24,
  no.~10, pp. 2984--2995, 2015.

\bibitem{ors2020event}
F.~K. {\"O}rs, S.~Yeniterzi, and R.~Yeniterzi, ``Event clustering within news
  articles,'' in \emph{in Proc. Workshop Automated Extraction Socio-political
  Events News (AESPEN)}, Marseille, France, 2020, pp. 63--68.

\bibitem{hou2017learning}
B.~Hou, L.~Zhang, and Z.~Zhou, ``Learning with feature evolvable streams,'' in
  \emph{in Proc. Adv. Neural Inf. Process. Syst. (NeurIPS)}, vol.~30, Long
  Beach, CA, USA, 2017, pp. 1417--1427.

\bibitem{hou2018one}
C.~Hou and Z.~Zhou, ``One-pass learning with incremental and decremental
  features,'' \emph{{IEEE} Trans. Pattern Anal. Mach. Intell.}, vol.~40,
  no.~11, pp. 2776--2792, 2018.

\bibitem{sadreddin2021incremental}
A.~Sadreddin and S.~Sadaoui, ``Incremental feature learning for infinite
  data,'' \emph{arXiv preprint arXiv:2108.02932}, 2021.

\bibitem{hou2019safe}
C.~Hou, L.~Zeng, and D.~Hu, ``Safe classification with augmented features,''
  \emph{{IEEE} Trans. Pattern Anal. Mach. Intell.}, vol.~41, no.~9, pp.
  2176--2192, 2019.

\bibitem{zhang2020learning}
Z.~Zhang, P.~Zhao, Y.~Jiang, and Z.~Zhou, ``Learning with feature and
  distribution evolvable streams,'' in \emph{in Proc. 37th Int. Conf. Mach.
  Learn. (ICML)}, vol. 119, Virtual, 2020, pp. 11\,317--11\,327.

\bibitem{hou2023incremental}
C.~Hou, S.~Gu, C.~Xu, and Y.~Qian, ``Incremental learning for simultaneous
  augmentation of feature and class,'' \emph{{IEEE} Trans. Pattern Anal. Mach.
  Intell.}, vol.~45, no.~12, pp. 14\,789--14\,806, 2023.

\bibitem{shu2023incremental}
W.~Shu, T.~Chen, D.~Cao, and W.~Qian, ``Incremental feature selection based on
  uncertainty measure for dynamic interval-valued data,'' \emph{Int. J. Mach.
  Learn. Cybern.}, vol.~15, pp. 1453--1472, 2024.

\bibitem{antos2005improved}
A.~Antos, ``Improved minimax bounds on the test and training distortion of
  empirically designed vector quantizers,'' \emph{{IEEE} Trans. Inf. Theory},
  vol.~51, no.~11, pp. 4022--4032, 2005.

\bibitem{tang2016lloyd}
C.~Tang and C.~Monteleoni, ``On lloyd’s algorithm: new theoretical insights
  for clustering in practice,'' in \emph{in Proc. 19th Int. Conf. Artif.
  Intell. Statist. (AISTATS)}, vol.~51, Cadiz, Spain, 2016, pp. 1280--1289.

\bibitem{zhang2023imbalanced}
J.~Zhang, H.~Tao, and C.~Hou, ``Imbalanced clustering with theoretical learning
  bounds,'' \emph{{IEEE} Trans. Knowl. Data Eng.}, vol.~35, no.~9, pp.
  9598--9612, 2023.

\bibitem{biau2008performance}
G.~Biau, L.~Devroye, and G.~Lugosi, ``On the performance of clustering in
  hilbert spaces,'' \emph{{IEEE} Trans. Inf. Theory}, vol.~54, no.~2, pp.
  781--790, 2008.

\bibitem{li2021sharper}
S.~Li and Y.~Liu, ``Sharper generalization bounds for clustering,'' in \emph{in
  Proc. 38th Int. Conf. Mach. Learn. (ICML)}, vol. 139, Virtual, 2021, pp.
  6392--6402.

\bibitem{liu2021refined}
Y.~Liu, ``Refined learning bounds for kernel and approximate $k$-means,'' in
  \emph{in Proc. Adv. Neural Inf. Process. Syst. (NeurIPS)}, vol.~34, Virtual,
  2021, pp. 6142--6154.

\bibitem{yin2022randomized}
R.~Yin, Y.~Liu, W.~Wang, and D.~Meng, ``Randomized sketches for clustering:
  Fast and optimal kernel $k$-means,'' in \emph{in Proc. Adv. Neural Inf.
  Process. Syst. (NeurIPS)}, vol.~35, New Orleans, LA, USA, 2022, pp.
  6424--6436.

\bibitem{liang2024consistency}
W.~Liang, C.~Tang, X.~Liu, Y.~Liu, J.~Liu, E.~Zhu, and K.~He, ``On the
  consistency and large-scale extension of multiple kernel clustering,''
  \emph{{IEEE} Trans. Pattern Anal. Mach. Intell.}, vol.~46, no.~10, pp.
  6935--6947, 2024.

\bibitem{li2023understanding}
S.~Li, S.~Ouyang, and Y.~Liu, ``Understanding the generalization performance of
  spectral clustering algorithms,'' in \emph{in Proc. 37th AAAI Conf. Artif.
  Intell. (AAAI)}, Washington, DC, USA, 2023, pp. 8614--8621.

\bibitem{Hou2023}
C.~Hou, S.~Gu, C.~Xu, and Y.~Qian, ``Incremental learning for simultaneous
  augmentation of feature and class,'' \emph{{IEEE} Trans. Pattern Anal. Mach.
  Intell.}, vol.~45, no.~12, pp. 14\,789--14\,806, 2023.

\bibitem{Gu2025}
S.~Gu, C.~Xu, D.~Hu, and C.~Hou, ``Adaptive learning for dynamic features and
  noisy labels,'' \emph{{IEEE} Trans. Pattern Anal. Mach. Intell.}, vol.~47,
  no.~2, pp. 1219--1237, 2025.

\bibitem{pmlr-v119-muzellec20a}
B.~Muzellec, J.~Josse, C.~Boyer, and M.~Cuturi, ``Missing data imputation using
  optimal transport,'' in \emph{in Proc. 37th Int. Conf. Mach. Learn. (ICML)},
  vol. 119, 2020, pp. 7130--7140.

\bibitem{pmlr-v80-ye18c}
H.-J. Ye, D.-C. Zhan, Y.~Jiang, and Z.-H. Zhou, ``Rectify heterogeneous models
  with semantic mapping,'' in \emph{in Proc. 35th Int. Conf. Mach. Learn.
  (ICML)}, vol.~80, 2018, pp. 5630--5639.

\bibitem{kifer2004detecting}
D.~Kifer, S.~Ben-David, and J.~Gehrke, ``Detecting change in data streams,'' in
  \emph{in Proc. 30th Int. Conf. Very Large Data Bases (VLDB)}, vol.~4,
  Toronto, Canada, 2004, pp. 180--191.

\bibitem{ben2010theory}
S.~Ben-David, J.~Blitzer, K.~Crammer, A.~Kulesza, F.~Pereira, and J.~W.
  Vaughan, ``A theory of learning from different domains,'' \emph{Mach.
  Learn.}, vol.~79, pp. 151--175, 2010.

\bibitem{mohri2012new}
M.~Mohri and A.~Mu{\~n}oz~Medina, ``New analysis and algorithm for learning
  with drifting distributions,'' in \emph{in Proc. 23rd Int. Conf. Algorithmic
  Learn. Theory (ALT)}, vol. 7568, Lyon, France, 2012, pp. 124--138.

\bibitem{dueck2007non}
D.~Dueck and B.~J. Frey, ``Non-metric affinity propagation for unsupervised
  image categorization,'' in \emph{in Proc. 11th IEEE Int. Conf. Comput. Vis.
  (ICCV)}, Rio de Janeiro, Brazil, 2007, pp. 1--8.

\bibitem{tao2018multiview}
H.~Tao, C.~Hou, D.~Yi, and J.~Zhu, ``Multiview classification with cohesion and
  diversity,'' \emph{{IEEE} Trans. Cybern.}, vol.~50, no.~5, pp. 2124--2137,
  2018.

\bibitem{SIFT}
D.~G. Lowe, ``Distinctive image features from scale-invariant keypoints,''
  \emph{Int. J. Comput. Vis.}, vol.~60, no.~2, pp. 91--110, 2004.

\bibitem{SURF}
H.~Bay, T.~Tuytelaars, and L.~V. Gool, ``{SURF:} speeded up robust features,''
  in \emph{in Proc. 9th Eur. Conf. Comput. Vis. (ECCV)}, vol. 3951, Graz,
  Austria, 2006, pp. 404--417.

\bibitem{banos2014dealing}
O.~Banos, M.~A. Toth, M.~Damas, H.~Pomares, and I.~Rojas, ``Dealing with the
  effects of sensor displacement in wearable activity recognition,''
  \emph{Sensors}, vol.~14, no.~6, pp. 9995--10\,023, 2014.

\bibitem{yang2024say}
J.~Yang, S.~Ma, Z.~Zhang, Y.~Li, S.~Xiao, J.~Wen, W.~Lu, and X.~Gao, ``Say no
  to redundant information: Unsupervised redundant feature elimination for
  active learning,'' \emph{{IEEE} Trans. Multimedia}, vol.~26, pp. 7721--7733,
  2024.

\bibitem{mavroeidis2014feature}
D.~Mavroeidis and E.~Marchiori, ``Feature selection for k-means clustering
  stability: theoretical analysis and an algorithm,'' \emph{Data Min. Knowl.
  Discov.}, vol.~28, no.~4, pp. 918--960, 2014.

\bibitem{oneto2015local}
L.~Oneto, A.~Ghio, S.~Ridella, and D.~Anguita, ``Local rademacher complexity:
  Sharper risk bounds with and without unlabeled samples,'' \emph{Neural
  Netw.}, vol.~65, pp. 115--125, 2015.

\bibitem{foster2019vector}
D.~J. Foster and A.~Rakhlin, ``l $\infty$ vector contraction for rademacher
  complexity,'' \emph{arXiv preprint arXiv:1911.06468}, vol.~6, 2019.

\bibitem{lei2019data}
Y.~Lei, {\"U}.~Dogan, D.-X. Zhou, and M.~Kloft, ``Data-dependent generalization
  bounds for multi-class classification,'' \emph{{IEEE} Trans. Inf. Theory},
  vol.~65, no.~5, pp. 2995--3021, 2019.

\bibitem{bousquet2002bennett}
O.~Bousquet, ``A bennett concentration inequality and its application to
  suprema of empirical processes,'' \emph{C. R. Math.}, vol. 334, no.~6, pp.
  495--500, 2002.

\end{thebibliography}

\end{document}